\DeclareSymbolFont{rsfs}{U}{rsfs}{m}{n}

\documentclass[10pt,a4paper,oneside,leqno,noamsfonts]{amsart}
\usepackage[utf8]{inputenc}
\usepackage[T1]{fontenc}  
\usepackage{amsthm} 

           \makeatletter
           \newcommand{\mylabel}[2]{#2\def\@currentlabel{#2}\label{#1}}
           \renewcommand\@biblabel[1]{#1.}
           \makeatother
\usepackage[english]{babel}
\usepackage[dvipsnames]{xcolor}

\usepackage[noadjust]{cite}

\usepackage{graphicx,pifont}
\usepackage[a4paper,top=2.8cm,bottom=2.8cm,left=2.7cm,
           right=2.7cm,bindingoffset=5mm,marginparwidth=60pt]{geometry}

                   \usepackage[baseline]{euflag}

\usepackage{tikz}
\usepackage{tikz-3dplot}
\usepackage{float}
\usepackage{braket,caption,comment,mathtools,stmaryrd}
\usepackage[usestackEOL]{stackengine}
\usepackage{multirow,booktabs,microtype,relsize}
\usepackage[colorlinks,bookmarks]{hyperref} 
     \usepackage[capitalise]{cleveref} 

      \hypersetup{colorlinks,%
            citecolor=britishracinggreen,%
            filecolor=black,%
            linkcolor=cobalt,%
            urlcolor=black}
      \numberwithin{equation}{section}

      \usepackage[enableskew]{youngtab} 
\usepackage{mathrsfs}

\usepackage[colorinlistoftodos]{todonotes}

\definecolor{antiquewhite}{rgb}{0.98, 0.92, 0.84}
\definecolor{buff}{rgb}{0.94, 0.86, 0.51}
\definecolor{palecopper}{rgb}{0.85, 0.54, 0.4}
\definecolor{fluorescentyellow}{rgb}{0.8, 1.0, 0.0}

\definecolor{britishracinggreen}{rgb}{0.0, 0.26, 0.15}
\definecolor{cobalt}{rgb}{0.0, 0.28, 0.67}
\DeclareSymbolFont{usualmathcal}{OMS}{cmsy}{m}{n}
\DeclareSymbolFontAlphabet{\mathcal}{usualmathcal}

\newcommand{\TT}{\mathbf{T}}

\newcommand{\BA}{{\mathbb{A}}}

\newcommand{\BC}{{\mathbb{C}}}

\newcommand{\BE}{{\mathbb{E}}}

\newcommand{\BL}{{\mathbb{L}}}

\newcommand{\BP}{{\mathbb{P}}}
\newcommand{\BQ}{{\mathbb{Q}}}

\newcommand{\BT}{{\mathbb{T}}}

\newcommand{\BZ}{{\mathbb{Z}}}

\newcommand{\CA}{{\mathcal A}}
\newcommand{\CB}{{\mathcal B}}

\newcommand{\CE}{{\mathcal E}}
\newcommand{\CF}{{\mathcal F}}

\newcommand{\CI}{{\mathcal I}}

\newcommand{\CK}{{\mathcal K}}
\newcommand{\CL}{{\mathcal L}}
\newcommand{\CM}{{\mathcal M}}
\newcommand{\CN}{{\mathcal N}}

\newcommand{\CZ}{{\mathcal Z}}

\newcommand{\pt}{{\mathsf{pt}}}
\newcommand{\ch}{{\mathrm{ch}}}

\newcommand{\fix}{\mathsf{fix}}
\newcommand{\mov}{\mathsf{mov}}

\DeclareMathOperator{\Hilb}{Hilb}
\DeclareMathOperator{\Sets}{Sets}
\DeclareMathOperator{\Sch}{Sch}

\DeclareMathOperator{\coker}{coker}

\DeclareMathOperator{\Coh}{Coh}

\DeclareMathOperator{\vir}{\mathrm{vir}}

\DeclareMathOperator{\Exp}{Exp}
\DeclareMathOperator{\Var}{Var}

\DeclareMathOperator{\rk}{rk}
\DeclareMathOperator{\op}{op}

\newcommand{\derived}{\mathbf{D}}

\usepackage{bbm}

\usepackage{subcaption}
\usepackage[outline]{contour}
\usetikzlibrary{cd,arrows,decorations.markings,positioning,shapes,calc,patterns,tqft,intersections,shapes.misc}
\tikzset{>=Latex}
\tikzset{GaugeNode/.style={circle,draw,inner sep=0pt,minimum size=10mm}}
\tikzset{FrameNode/.style={rectangle,draw,inner sep=0pt,minimum size=9mm}}
\tikzset{token/.style={circle,double,draw=black!70,fill=black!50,inner sep=0pt,minimum size=3mm}}

\newcommand{\into}{\hookrightarrow}
\newcommand{\onto}{\twoheadrightarrow}

\DeclareFontFamily{OT1}{rsfs}{}
\DeclareFontShape{OT1}{rsfs}{n}{it}{<-> rsfs10}{}
\DeclareMathAlphabet{\curly}{OT1}{rsfs}{n}{it}
\renewcommand\hom{\curly H\!om}

\newcommand\Hom{\operatorname{Hom}}

\newcommand{\RR}{\mathbf R}
\newcommand{\LL}{\mathbf L}

\newcommand\Supp{\operatorname{Supp}}

\newcommand\Tot{\operatorname{Tot}}

\newcommand{\PT}{\mathsf{PT}}
\newcommand{\DT}{\mathsf{DT}}
\newcommand{\GW}{\mathsf{GW}}

\DeclareMathOperator{\bn}{{\mathbf{n}}}
\DeclareMathOperator{\Bm}{{\mathbf{m}}}
\DeclareMathOperator{\Subsoc}{\mathrm{Subsoc}}
\DeclareMathOperator{\Soc}{\mathrm{Soc}}
\DeclareMathOperator{\CoSubsoc}{\mathrm{CoSubsoc}}
\DeclareMathOperator{\CoSoc}{\mathrm{CoSoc}}



\usepackage[all]{xy}
\usepackage{tikz}
\usepackage{tikz-cd}
\usepackage{adjustbox}
\usepackage{rotating}
\usepackage{comment}

\usetikzlibrary{matrix,shapes,intersections,arrows,decorations.pathmorphing}
\tikzset{commutative diagrams/arrow style=math font}
\tikzset{commutative diagrams/.cd,
mysymbol/.style={start anchor=center,end anchor=center,draw=none}}

\tikzset{
shift up/.style={
to path={([yshift=#1]\tikztostart.east) -- ([yshift=#1]\tikztotarget.west) \tikztonodes}
}
}

\theoremstyle{definition}

\newtheorem*{lemma*}{Lemma}
\newtheorem*{theorem*}{Theorem}
\newtheorem*{example*}{Example}
\newtheorem*{fact*}{Fact}
\newtheorem*{notation*}{Notation}
\newtheorem*{definition*}{Definition}
\newtheorem*{prop*}{Proposition}
\newtheorem*{remark*}{Remark}
\newtheorem*{corollary*}{Corollary}

\newtheorem*{conventions*}{Conventions}

\newtheorem{definition}{Definition}[section]

\newtheorem{remark}[definition]{Remark}
\newtheorem{conjecture}[definition]{Conjecture}

\newtheoremstyle{thm} 
        {3mm}
        {3mm}
        {\slshape}
        {0mm}
        {\bfseries}
        {.}
        {1mm}
        {}
\theoremstyle{thm}
\newtheorem{theorem}[definition]{Theorem}
\newtheorem{corollary}[definition]{Corollary}
\newtheorem{lemma}[definition]{Lemma}
\newtheorem{prop}[definition]{Proposition}

\newtheoremstyle{ex} 
        {3mm}
        {3mm}
        {}
        {0mm}
        {\scshape}
        {.}
        {1mm}
        {}
\theoremstyle{ex}

\newtheoremstyle{sol} 
        {3mm}
        {3mm}
        {}
        {0mm}
        {\scshape}
        {.}
        {1mm}
        {}
\theoremstyle{sol}

\usepackage{amsmath,float}
\usetikzlibrary{decorations.markings}

\usepackage{amssymb}
\usepackage{enumitem}
\usepackage{comment}

\usepackage{bm}

   \DeclareMathOperator{\oO}{\mathcal{O}}

\allowdisplaybreaks
\title[The refined  local Donaldson-Thomas  theory of  curves]{The refined  local Donaldson-Thomas  theory of  curves}

\author{Sergej Monavari}
\address{Dipartimento di Matematica “Tullio Levi-Civita”, Università degli Studi di Padova, Via Trieste 63, 35121 Padova, Italy}
\email{sergej.monavari@math.unipd.it}

\begin{document}

\maketitle

\begin{abstract}
We solve the K-theoretically refined  Donaldson-Thomas theory of local curves in arbitrary genus and  degree. Our results avoid degeneration techniques, but rather exploit direct localisation methods to reduce the refined Donaldson-Thomas partition function to the equivariant intersection theory of skew nested Hilbert schemes on smooth projective curves. 
In  the refined limit, our results  establish a formula  for the refined topological string  partition function of local curves conjecturally proposed by Aganagic-Schaeffer. 
In the second part, we show that analogous structural results  hold for the refined Pandharipande-Thomas theory of local curves. As an application, we deduce the K-theoretic DT/PT correspondence for local curves in arbitrary genus, as conjectured by Nekrasov-Okounkov.\\
Thanks to the recent machinery developed by Pardon, we expect  our explicit results on local curves to play a key role towards the proof of the refined GW/PT conjectural  correspondence of Brini-Schuler for all smooth  Calabi-Yau threefolds.
\end{abstract}
\section{Introduction}

\subsection{Overview}
While Donaldson-Thomas (DT) theory was originally formulated in \cite{Tho_Casson_inv} to \emph{numerically count} stable sheaves on smooth projective Calabi–Yau threefolds, its \emph{$K$-theoretic refinement} has recently attracted growing attention,   see e.g.~\cite{Okounk_Lectures_K_theory, Oko_Takagi, Arb_K-theo_surface, FMR_higher_rank, CKM_K_theoretic, thimm_orbi, tho_K-thDTK3, ober_refined_Enriq, Cir_M2_index, BBPT_elliptic_DT, Nek_Z_theory, Nek_Magnificent4_advances, Tho_equivariant_K_theory}. In the seminal work \cite{NO_membranes_and_sheaves}, Nekrasov-Okounkov proposed a striking interpretation of the  $K$-theoretic  DT partition function  of a smooth threefold $X$ as the \emph{refined index}  of a moduli space of M2-branes on the quasi-projective Calabi-Yau \emph{fivefold} $\Tot_X(\CL\oplus \CL')$, where $\CL, \CL'$ are line bundles on $X$  satisfying $\CL\otimes \CL'\cong \omega_X$. While the geometry of the moduli space of M2-branes remains mysterious and currently lacks  a complete foundational description, Nekrasov-Okounkov's influential suggestion enables predictions on the refined M2-brane index by studying the refined DT partition function of the corresponding threefold.

In this work we focus on the case when $X=\Tot_C(L_1\oplus L_2)$ is a \emph{local curve}, that  is the total space of any two line bundles $L_1,  L_2$ over a smooth projective curve $C$. Our motivation comes from the recent circle of ideas  of  Pardon \cite{Pardon_GWDT}, for which  curve-counting enumerative theories of a  Calabi-Yau threefold are expected to be  \emph{universally} determined by their analogues on local curves.

 Our main results are:
\begin{itemize}
    \item the full computation of the $K$-theoretic Donaldson-Thomas partition function $  \widehat{ \DT}_d(X, q)$ for all degrees $d\geq 0$ and for all genera $g\geq 0$, cf.~\Cref{thm: full DT Intro},    \item an explicit expression for the \emph{refined limit} of  $  \widehat{ \DT}_d(X, q)$  in the Calabi-Yau case,  which establishes a formula  for the refined topological string partition function of local curves, as   proposed by Aganagic-Schaeffer \cite{AS_black_holes} via string-theoretic methods, cf.~\Cref{thm: refined DT full Aga INTRO},
    \item an explicit expression for the anti-diagonal restriction of the equivariant parameters of  $  \widehat{ \DT}_d(X, q)$, which reproduces and generalises the corresponding formula in  Gromov-Witten (GW) theory of Bryan-Pandharipande \cite{BP_local_GW_curves} via the GW/DT correspondence \cite{MNOP_1}, cf.~\Cref{thm: DT anti diagonal INTRO}, 
    \item  the $K$-theoretic DT/PT correspondence for local curves in any genus, as conjectured by Nekrasov-Okounkov \cite{NO_membranes_and_sheaves}, cf.~\Cref{thm: DT/PT INTRO}.
\end{itemize}
We explain in the next sections our results in more details.
\subsection{Refined Donaldson-Thomas theory}
Let $X=\Tot_C(L_1\oplus L_2)$ be the total space of two line bundles $L_1, L_2$ on a smooth projective curve $C$. For a curve class $\beta=d[C]\in H_2(X, \BZ)$ and $n\in \BZ$, consider the \emph{Hilbert scheme} $\Hilb^n(X, \beta)$, parametrising closed subschemes $Z\subset X$ with $[\Supp Z]=\beta$ and $\chi(\oO_Z)=n$. 

The Hilbert scheme  $\Hilb^n(X, \beta)$ is in general pathologically badly behaved; nevertheless, by the work of Behrend-Fantechi \cite{BF_normal_cone}, it is endowed with a \emph{virtual fundamental structure sheaf}  in $K$-theory
\begin{align*}
    \oO_{\Hilb^n(X, \beta)}^{\vir}\in K_0(\Hilb^n(X, \beta)).
\end{align*}
Following Nekrasov-Okounkov \cite{NO_membranes_and_sheaves}, we introduce the \emph{twisted virtual structure sheaf}
\begin{align}\label{eqn: intro symm}
    \widehat{\oO}^{\vir}=\oO^{\vir}\otimes K^{1/2}_{\vir},
\end{align}
where $ K^{1/2}_{\vir}$ is a \emph{square root} of the \emph{virtual canonical bundle} $K_{\vir}$,  and define \emph{$K$-theoretic  Donaldson-Thomas invariants} as
\begin{align}\label{eqn: DT intro def}
    \widehat{\DT}_{d,n}(X)=\chi\left(\Hilb^n(X, \beta),   \widehat{\oO}^{\vir}  \right),
\end{align}
see \Cref{sec: K-th inva DT}.
Since $ \Hilb^n(X, \beta)$ is  generally not proper, the right-hand-side of \eqref{eqn: DT intro def} is a priori not well-defined via ordinary intersection theory.  Nevertheless, the algebraic torus $\TT=(\BC^*)^2$ acts on $X$ by scaling the fibers,  and the action naturally lifts to $ \Hilb^n(X, \beta) $. As the $\TT$-fixed locus $\Hilb^n(X, \beta)^\TT$ is proper (cf.~\Cref{thm: fixed}), we define the right-hand-side of \eqref{eqn: DT intro def} by 
the virtual localisation formula in $K$-theory \cite{FG_riemann_roch, Qu_virtual_pullback} as
\begin{align*}
   \chi\left(\Hilb^n(X, \beta),   \widehat{\oO}^{\vir}  \right)=\chi\left(\Hilb^n(X, \beta)^\TT,\frac{  \widehat{\oO}^{\vir}_{\Hilb^n(X, \beta)^\TT} }{\widehat{\mathfrak{e}}(N^{\vir})}  \right)\in \BQ(t_1^{1/2}, t_2^{1/2}),
\end{align*}
where $ t_1, t_2$ are the generators of the $\TT$-equivariant $K$-theory $K^0_\TT(\pt)$ and $ \widehat{\mathfrak{e}}(N^{\vir})$ is the symmetrised $K$-theoretic Euler class of the virtual normal bundle, see \Cref{sec: K-th inva DT}. We remark that the \emph{symmetrisation} of the virtual structure sheaf \eqref{eqn: intro symm} was introduced so that the DT invariants \eqref{eqn: DT intro def} resemble an algebro-geometric analogue of the $\widehat{A}$-genus of a spin manifold \cite{NO_membranes_and_sheaves}.
\smallbreak
The first main result of our work is the full computation of the associated \emph{partition function}
\begin{equation*}\label{eqn: DT partition functions Intro}
    \widehat{\DT}_d(X, q)=\sum_{n\in \BZ} \widehat{\DT}_{d,n}(X)\cdot q^n\in \BQ(t_1^{1/2}, t_2^{1/2})(\!( q )\!).
\end{equation*}

\begin{theorem}[\Cref{thm: DT_ explicit univ series}, \Cref{cor: full DT}]\label{thm: full DT Intro}
    Let  $d\geq 0$ and    $X=\Tot_C(L_1 \oplus L_2)$, where $L_1, L_2$ are line bundles over a smooth projective curve $C$ of genus $g$. We have
    \begin{align*}
     \widehat{\DT}_d(X, q)= \sum_{|\lambda|=d} \left(q^{|\lambda|}\widehat{A}_{\lambda}(q)\right)^{1-g}\cdot \left(q^{-n(\lambda)}\widehat{B}_{\lambda}(q)\right)^{\deg L_1}\cdot \left(q^{-n(\overline{\lambda})}\widehat{C}_{\lambda}(q)\right)^{\deg L_2},
\end{align*}
where the universal series are given by
\begin{align*}
      \widehat{A}_\lambda(q)&=\prod_{\Box\in \lambda}\frac{1}{[t_1^{-\ell(\Box)}t_2^{a(\Box)+1}][t_1^{\ell(\Box)+1}t_2^{-a(\Box)}]}\cdot \left.\left(\widehat{\mathsf{V}}_\lambda(q)\cdot \widehat{\mathsf{V}}_\lambda(q)|_{ t_3=t^{-1}_3}\right)\right|_{t_3=1}\\
    \widehat{B}_\lambda(q)&=\prod_{\Box\in \lambda}\frac{[t_1^{-\ell(\Box)}t_2^{a(\Box)+1}]^{\ell(\Box)}}{[t_1^{\ell(\Box)+1}t_2^{-a(\Box)}]^{\ell(\Box)+1}}\cdot \left.\left(\widehat{\mathsf{V}}_\lambda(q)|_{ t_3=t_3^{-1}}\cdot \widehat{\mathsf{V}}_\lambda(q)^{-1}|_{t_1=t_1t^2_3, t_3=t^{-1}_3}\right)\right|^{\frac{1}{2}}_{t_3=1},  \\
    \widehat{C}_\lambda(q)&=\prod_{\Box\in \lambda}\frac{[t_1^{\ell(\Box)+1}t_2^{-a(\Box)}]^{a(\Box)}}{[t_1^{-\ell(\Box)}t_2^{a(\Box)+1}]^{a(\Box)+1}}\cdot \left.\left(\widehat{\mathsf{V}}_\lambda(q)|_{ t_3=t_3^{-1}}\cdot \widehat{\mathsf{V}}_\lambda(q)^{-1}|_{t_2=t_2t^2_3, t_3=t^{-1}_3}\right)\right|^{\frac{1}{2}}_{t_3=1}.
\end{align*}
\end{theorem}
We explain now carefully all  ingredients appearing in  \Cref{thm: full DT Intro}.
\smallbreak

The sum is over all  \emph{Young diagrams} $\lambda$ of size $d$ and,  for a box $\Box\in \lambda$, we denote by $a(\Box)$ (resp.~$\ell(\Box)$) the \emph{arm} (resp.~\emph{leg}) length. For a formal variable $x$, define the \emph{symmetrised} operator
\begin{align*}
[x]=x^{1/2}-x^{-1/2}.
\end{align*}
Finally, the generating series $ \widehat{\mathsf{V}}_\lambda(q)$ is the \emph{1-leg $K$-theoretic equivariant vertex\footnote{This is the \emph{fully equivariant} refinement of the \emph{refined} topological vertex of \cite{IKV_topological_vertex}, see \cite{NO_membranes_and_sheaves}.}} with asymptotic profile $\lambda$, originally\footnote{To be precise, in \cite{MNOP_1} the vertex formalism is originally derived in equivariant cohomology, not in equivariant $K$-theory.} introduced via the vertex/edge formalism \cite{MNOP_1}, but inspired by the previous works on the \emph{topological string amplitudes} \cite{AKMV_topvert, IKV_topological_vertex}.
More explicitly,  the function $\widehat{\mathsf{V}}_\lambda(q)$ is a series whose coefficients are rational functions in the variables $t_1, t_2, t_3$. In our work,  we    combinatorially realise it as
\begin{align*}
\widehat{\mathsf{V}}_\lambda(q)=\sum_{\bn}\widehat{\mathfrak{e}}(-\mathsf{v}_{\bn})\cdot q^{|\mathbf{n}|}\in \BQ(t_1^{1/2}, t_2^{1/2}, t_3^{1/2})\llbracket q \rrbracket,
\end{align*}
where the sum is over all \emph{skew plane partitions} $\bn$ of shape $\BZ^2_{\geq 0}\setminus \lambda$ (see \Cref{def: spp}), and the operator $ \widehat{\mathfrak{e}}$ turns the Laurent polynomials 
$ \mathsf{v}_{\bn}$ (see \eqref{eqn: vertex DT}) into rational functions in the equivariant parameters. 
\smallbreak
We remark that \Cref{thm: full DT Intro} is the first  example of  a computation of the $K$-theoretic DT partition function of a \emph{non-toric} quasi-projective threefold, and in particular in the case of  a \emph{non-irreducible} curve class $\beta$.

Before sketching the main ideas of the proof of \Cref{thm: full DT Intro}, we discuss some  explicit formulas that can be derived from our main result. 
To ease the notation, we set $\kappa=t_1t_2 $ and formally define two new variables $t_4, t_5$ by
\begin{align}\label{eqn: new variables}
\begin{split}
        t_4&=q\kappa^{-1/2},\\
    t_5&=q^{-1}\kappa^{-1/2}.
    \end{split}
\end{align}
\subsubsection{Degree 0}
Degree 0 DT type invariants are typically known to be more amenable to computations, see for instance \cite{Okounk_Lectures_K_theory, FMR_higher_rank, thimm_orbi, FM_tetra}. Exploiting the previous local computation of the  $K$-theoretic DT partition function of $\BA^3$ by Okounkov \cite{Okounk_Lectures_K_theory}, we prove a more compact formula for the DT partition function of local curves.
\begin{corollary}[\Cref{cor: as in Oko}]\label{cor: intro DT 0}
Let     $X=\Tot_C(L_1 \oplus L_2)$, where $L_1, L_2$ are line bundles over a smooth projective curve $C$. We have
\begin{align*}
      \widehat{\DT}_0(X, -q)=\Exp\left(-\chi(X, T_X+\omega_X-T_X^*-\omega_X^{-1}) \cdot \frac{1}{[t_4][t_5]}\right),
\end{align*}
where $\Exp$ is the plethystic exponential \eqref{eqn: on ple}.
\end{corollary}
We remark that, for $X$ of Calabi-Yau type,  the formula of \Cref{cor: intro DT 0} already appeared  in \cite[Thm.~3.3.6]{Okounk_Lectures_K_theory}. In loc.~cit.~it is argued that the general formula for the degree zero Donaldson-Thomas invariants of a (quasi-projective) Calabi-Yau threefold should follow from the local case of $\BA^3$ by an algebraic cobordism argument, involving   degenerating techniques and relative invariants as in \cite{LP_algebraic_cobordism}. Our argument offers a direct, clear  and much shorter proof in the case of a general local curve.
\subsubsection{Degree 1}
We prove a more compact expression for the formula in \Cref{thm: full DT Intro} also in the case that $\beta=[C]$ is an \emph{irreducible} class, that is when the degree $d=1$.
\begin{corollary}[\Cref{prop: degree 1}]\label{cor: intro DT 1}
Let     $X=\Tot_C(L_1 \oplus L_2)$, where $L_1, L_2$ are line bundles over a smooth projective curve $C$ of genus $g$. Set $k_i=\deg L_i$. We have    
\begin{align*}
  \frac{ \widehat{\DT}_1(X, -q)}{   \widehat{\DT}_0(X, -q)}=(-1)^{1-g}\cdot q^{1-g}[t_1]^{g-1-k_1}[t_2]^{g-1-k_2}\cdot\left(\frac{1}{(1-t_4)(1-t_5^{-1})}\right)^{ -\frac{k_1+k_2}{2}}.
\end{align*}
In particular, if $L_1\otimes L_2 \cong \omega_C$, we have
\begin{align*}
    \frac{ \widehat{\DT}_1(X, -q)}{   \widehat{\DT}_0(X, -q)}=
    [t_1]^{g-1-k_1}[t_2]^{g-1-k_2}[t_4]^{g-1}[t_5]^{g-1}.
\end{align*}
\end{corollary}
To be precise, we prove in \Cref{prop: degree 1} the formula for the degree $d=1$ case in Pandharipande-Thomas theory. \Cref{cor: intro DT 1} is then obtained by combining the formula in loc.~cit.~with the DT/PT correspondence, see \Cref{thm: DT/PT INTRO}.

In the Calabi-Yau case, \Cref{cor: intro DT 1},  \Cref{thm: DT/PT INTRO} and a straightforward Riemann-Roch computation recover the formula of Okounkov\footnote{This formulation of Okounkov's formula was communicated to us by Yannik Schuler \cite{Schuler-private}.} \cite[Cor.~5.1.19]{Okounk_Lectures_K_theory}, stating
\begin{align*}
    \widehat{\PT}_1(X, -q)=\widehat{\mathfrak{e}}\left(-\chi(C, L_1\oplus L_2\oplus \oO\oplus \oO)\right),
\end{align*}
where $\widehat{\PT}_1(X, -q)$ is the degree 1 PT partition function (see \Cref{sec: PT}) and the line bundles $L_1, L_2, \oO, \oO$ have equivariant weights respectively $t_1, t_2, t_4, t_5$. In particular,  the hidden symmetries of the DT partition function become more transparent from the M-theoretic fivefold point of view of \cite{NO_membranes_and_sheaves}.

\subsection{Refined limit}
While \Cref{thm: full DT Intro} does compute the \emph{fully} $\TT$-equivariant   Donaldson-Thomas partition $ \widehat{\DT}_d(X,  q)$, closed explicit formulas are in general difficult to extract, away from the degree $d=0,1$ cases. 

Assume now that the local curve $X$ is Calabi-Yau, or equivalently that $L_1\otimes L_2\cong \omega_C$. Since in the Calabi-Yau case the perfect obstruction theory on $\Hilb^n(X, \beta)$ is \emph{symmetric}, we can \emph{scale}  the equivariant parameters $t_1, t_2$ to infinity while keeping the \emph{Calabi-Yau weight} $\kappa=t_1t_2$ constant, following the original treatment of the \emph{refined topological vertex} \cite{IKV_topological_vertex}.
We set
\begin{align*}
    \DT^{\mathsf{ref}}_d(X, q)=\lim_{L\to \infty}\widehat{\DT}_d(X, q)|_{t_1=L\kappa^{\frac{1}{2}}, t_2=L^{-1}\kappa^{\frac{1}{2}}}\in \BQ(\kappa^{1/2})(\!( q )\!),
\end{align*}
\begin{theorem}[\Cref{thm: refined DT full Aga}]\label{thm: refined DT full Aga INTRO}
    Let $X=\Tot_C(L_1\oplus L_2)$ be a local curve, where $L_1, L_2$ are line bundles over a smooth projective curve $C$ of genus $g$, such that $L_1\otimes L_2\cong \omega_C$. Set $\deg L_1=k_1$. We have
    \begin{align*}
       \DT^{\mathsf{ref}}_0(X, -q)&=\Exp\left((1-g)\frac{ (t_4t_5)^{\frac{1}{2}}+(t_4t_5)^{-\frac{1}{2}}
       }{[t_4][t_5]}
             \right),\\
       \frac{ \DT^{\mathsf{ref}}_d(X, -q)}{ \DT^{\mathsf{ref}}_0(X, -q)}&=(-1)^{dk_1}\cdot 
         \sum_{|\lambda|=d}\left(t_4^{\left \lVert \lambda \right \rVert^2} \prod_{\Box\in \lambda} \frac{1}{(1- t_4^{a(\Box)+1}t_5^{-\ell(\Box)})(1-t_4^{a(\Box)} t_5^{-\ell(\Box)-1})} \right)^{1-g}\cdot t_4^{\frac{k_1\cdot \left \lVert \lambda \right \rVert^2}{2}}t_5^{ \frac{k_1\cdot \left \lVert \overline{\lambda} \right \rVert^2}{2}}.
    \end{align*}
    In particular, the left-hand-side of the second identity is a rational function.
\end{theorem}
Remarkably, the refined Donaldson-Thomas partition function computed in  \Cref{thm: refined DT full Aga INTRO} establishes a conjectural formula\footnote{This formulation of the Aganagic-Schaeffer formula was communicated to us by Yannik Schuler \cite{Schuler-private}.} for the refined topological string of $\Tot_{C}(L_1\oplus L_2)$ proposed by Aganagic-Schaeffer \cite[Eqn.~(4.13)]{AS_black_holes} in string theory,  and studied with a topological quantum field theory  (TQFT) approach in the context of the refined  Ooguri-Strominger-Vafa conjecture. In particular, when $X$ is the resolved conifold $\Tot_{\BP^1}(\oO(-1)\oplus \oO(-1))$, \Cref{thm: refined DT full Aga INTRO} reproduces the well-known plethystic exponential expression \eqref{eq_ res con} for the refined topological string originally studied in \cite{IKV_topological_vertex}, and the  \emph{motivic} Pandharipande-Thomas partition function computed by Morrison-Mozgovoy-Nagao-Szendr\H{o}i \cite{MMNS_motivic_DT_conifold}. 
\subsection{Anti-diagonal restriction}
Exploiting the symmetries of the $K$-theoretic equivariant vertex $\widehat{\mathsf{V}}_\lambda(q)$, we establish a formula for the refined DT partition function in the anti-diagonal restriction of the equivariant parameters $t_1t_2=1$.
\begin{theorem}[\Cref{cor: antidiagonal DT}, \ref{cor: full DT}]\label{thm: DT anti diagonal INTRO}
    Let     $X=\Tot_C(L_1 \oplus L_2)$, where $L_1, L_2$ are line bundles over a smooth projective curve $C$ of genus $g$. Set $\deg L_i=k_i$.  In the anti-diagonal restriction, we have
\begin{multline*}
     \left.\widehat{\DT}_d(X,  -q)\right|_{t_1t_2=1}=\\
     (-1)^{d \cdot k_2}\sum_{|\lambda|=d}q^{|\lambda|(1-g)-n(\lambda)k_1-n(\overline{\lambda})k_2}\cdot
     \prod_{\Box \in \lambda}[t_1^{h(\Box)}]^{2g-2-k_1-k_2}\cdot \left( \mathsf{M}(q)^{-1} \cdot \prod_{\Box\in \lambda}(1-q^{h(\Box)})\right)^{k_1+k_2},
\end{multline*}
where $  \mathsf{M}(q)$  is the MacMahon's function \eqref{eqn: MM}.
\end{theorem}
By applying the DT/PT correspondence (see \Cref{thm: DT/PT INTRO}) and the equivariant cohomology limit (see \Cref{prop: limit equiv INTRO}), the formula in \Cref{thm: DT anti diagonal INTRO} reproduces our previous result \cite[Thm.~1.3]{Mon_double_nested}, which in turn matches the Gromov-Witten partition function of local curves computed by Bryan-Pandharipande \cite{BP_local_GW_curves} via the (unrefined) GW/PT correspondence \cite{MNOP_1, PT_curve_counting_derived}.
\smallbreak
We remark that, in the case $X$ is Calabi-Yau, the partition function in the anti-diagonal restriction of the equivariant parameters can be furthermore expressed (up to a sign) as a \emph{signed}  \emph{topological Euler characteristic} partition function.
\begin{corollary}[\Cref{cor: DT is top}]\label{cor: DT is top INTRO}
    Let $X=\Tot_C(L_1\oplus L_2)$ be a local curve, where $L_1, L_2$ are line bundles over a smooth projective curve $C$ such that $L_1\otimes L_2\cong \omega_C$. Then we have
    \begin{align*}
          \left.\widehat{\DT}_d(X,  q)\right|_{t_1t_2=1}=(-1)^{d \cdot k_2} \DT^{\mathrm{top}}_d(X,- q),
    \end{align*}
    where $\DT^{\mathrm{top}}_d(X, q)$ is the generating series of the topological Euler characteristic.
\end{corollary}
When $X$ is Calabi-Yau, \Cref{cor: DT is top INTRO} shows that the DT invariants under the Calabi-Yau specialisation $t_1t_2=1$ are  purely topological. We believe this fact  to be more than  a mere coincidence, and to relate to the \emph{Behrend's function} of $\Hilb^n(X, \beta)$. In fact, if $\Hilb^n(X, \beta)$ were proper -- something that almost never happens -- the Donaldson-Thomas invariant
\[
\int_{[\Hilb^n(X, \beta)]^{\vir}}1\in \BZ
\]
could be computed in the Calabi-Yau specialisation of the equivariant parameters, and would equate Behrend's weighted Euler characteristic \cite{Beh_DT_via_microlocal}.  By torus localisation \cite{BF_symm_pot, descombes_hyperbolic_loc}, the latter is computed as  a signed topological Euler characteristic of the fixed locus of the Calabi-Yau subtorus $\set{t_1t_2=1}\subset (\BC^*)^2$. See also \cite{JT_virtualsigned}, where a  
relation between the signed topological Euler characteristic of a scheme with a perfect obstruction theory and the virtual Euler characteristic of its $-1$-shifted symplectic cotangent bundle is drawn.
\subsection{DT/PT correspondence}
Let  $P_n(X, \beta)$ denote  the moduli space of \emph{stable pairs}, which parametrises complexes 
\[[\oO_X\xrightarrow[]{s} F]\in \derived^b(X),\]
such that $\coker (s)$ is zero-dimensional,  $[\Supp F]=\beta$ and $\chi(F)=n$. The moduli space  $P_n(X, \beta)$ was introduced by Pandharipande-Thomas \cite{PT_curve_counting_derived} to give a more geometric understanding of the GW/DT correspondence \cite{MNOP_1} and to compute BPS invariants via sheaf-theoretic methods \cite{PT_stable_pairs_BPS}.
\smallbreak
We show that the \emph{Pandharipande-Thomas partition function} $   \widehat{\PT}_d(X, q)$ (see \eqref{eqn: localized PT KK invariants} for its definition) enjoys a similar universal structure as in  \Cref{thm: full DT Intro}.

\begin{theorem}[\Cref{thm: PT_ explicit univ series}, \Cref{cor: full PT}]\label{thm: full PT Intro}
    Let  $d\geq 0$ and    $X=\Tot_C(L_1 \oplus L_2)$, where $L_1, L_2$ are line bundles over a smooth projective curve $C$ of genus $g$. We have
    \begin{align*}
     \widehat{\PT}_d(X, q)= \sum_{|\lambda|=d} \left(q^{|\lambda|}\widehat{D}_{\lambda}(q)\right)^{1-g}\cdot \left(q^{-n(\lambda)}\widehat{E}_{\lambda}(q)\right)^{\deg L_1}\cdot \left(q^{-n(\overline{\lambda})}\widehat{F}_{\lambda}(q)\right)^{\deg L_2},
\end{align*}
where the universal series are given by
\begin{align*}
      \widehat{D}_\lambda(q)&=\prod_{\Box\in \lambda}\frac{1}{[t_1^{-\ell(\Box)}t_2^{a(\Box)+1}][t_1^{\ell(\Box)+1}t_2^{-a(\Box)}]}\cdot \left.\left(\widehat{\mathsf{V}}^{\PT}_\lambda(q)\cdot \widehat{\mathsf{V}}^{\PT}_\lambda(q)|_{ t_3=t^{-1}_3}\right)\right|_{t_3=1}\\
    \widehat{E}_\lambda(q)&=\prod_{\Box\in \lambda}\frac{[t_1^{-\ell(\Box)}t_2^{a(\Box)+1}]^{\ell(\Box)}}{[t_1^{\ell(\Box)+1}t_2^{-a(\Box)}]^{\ell(\Box)+1}}\cdot \left.\left(\widehat{\mathsf{V}}^{\PT}_\lambda(q)|_{ t_3=t_3^{-1}}\cdot \widehat{\mathsf{V}}^{\PT}_\lambda(q)^{-1}|_{t_1=t_1t^2_3, t_3=t^{-1}_3}\right)\right|^{\frac{1}{2}}_{t_3=1},  \\
    \widehat{F}_\lambda(q)&=\prod_{\Box\in \lambda}\frac{[t_1^{\ell(\Box)+1}t_2^{-a(\Box)}]^{a(\Box)}}{[t_1^{-\ell(\Box)}t_2^{a(\Box)+1}]^{a(\Box)+1}}\cdot \left.\left(\widehat{\mathsf{V}}^{\PT}_\lambda(q)|_{ t_3=t_3^{-1}}\cdot \widehat{\mathsf{V}}^{\PT}_\lambda(q)^{-1}|_{t_2=t_2t^2_3, t_3=t^{-1}_3}\right)\right|^{\frac{1}{2}}_{t_3=1}.
\end{align*}
\end{theorem}
The generating series $ \widehat{\mathsf{V}}^{\PT}_\lambda(q)$ is the 1-leg $K$-theoretic equivariant vertex with asymptotic profile $\lambda$ coming from PT theory \cite{PT_vertex}. In our work,  we    combinatorially realise it as
\begin{align*}
\widehat{\mathsf{V}}^{\PT}_\lambda(q)=\sum_{\Bm}\widehat{\mathfrak{e}}(-\mathsf{v}_{\Bm})\cdot q^{|\mathbf{n}|}\in \BQ(t_1^{1/2}, t_2^{1/2}, t_3^{1/2})\llbracket q \rrbracket,
\end{align*}
where the sum is over all \emph{reverse plane partitions} $\Bm$ of shape $ \lambda$ (see \Cref{def: spp}), and the operator $ \widehat{\mathfrak{e}}$ turns the Laurent polynomials 
$ \mathsf{v}_{\Bm}$ (see \eqref{eqn: PT vertex Lauren}) into rational functions in the equivariant parameters.

As a corollary of \Cref{thm: full DT Intro}, \ref{thm: full PT Intro} we deduce the following $K$-theoretic DT/PT correspondence, confirming a  conjecture of Nekrasov-Okounkov \cite{NO_membranes_and_sheaves} for all local curves.
\begin{corollary}[\Cref{thm: DT/PT}]\label{thm: DT/PT INTRO}
    Let $X=\Tot_C(L_1\oplus L_2)$ be a local curve, where $L_1, L_2$ are line bundles over a smooth projective curve. Then, the following correspondence holds
\begin{align*}
        \widehat{\DT}_d(X, q)=\widehat{\DT}_0(X, q)\cdot \widehat{\PT}_d(X, q).
\end{align*}
\end{corollary}
The proof of \Cref{thm: DT/PT INTRO} follows by combining  \Cref{thm: full DT Intro}, \ref{thm: full PT Intro} with the $K$-theoretic \emph{vertex} DT/PT correspondence
\begin{align*}
\widehat{\mathsf{V}}_\lambda(q)=\widehat{\mathsf{V}}_{\varnothing}(q)\cdot \widehat{\mathsf{V}}^{\PT}_\lambda(q),
\end{align*}
which was recently established by Kuhn-Liu-Thimm \cite{KLT_DTPT} using  wall-crossing techniques. We remark that by the vertex/edge formalism (see e.g.~\cite{MNOP_1, PT_vertex}), the vertex DT/PT correspondence of \cite{KLT_DTPT} 
immediately extends to  all quasi-projective toric threefolds.  Our  \Cref{thm: DT/PT INTRO} represents the first instance of the $K$-theoretic DT/PT correspondence in the \emph{non-toric} case. Specialising \Cref{thm: DT/PT INTRO} to the equivariant cohomology limit gives a rigourous derivation of the unrefined DT/PT correspondence for local curves.
\smallbreak
Denote by $   \DT^{\mathrm{top}}_d(X, q)$ and $   \PT^{\mathrm{top}}_d(X, q)$ the generating series of \emph{topological Euler characteristics},  respectively for the Hilbert schemes and for the moduli spaces of stable pairs of a local curve $X$. We provide an explicit formula for  the former in terms of  \emph{hooklengths} (see \Cref{cor: top DT}) which, together with our previous result \cite[Cor~3.2]{Mon_double_nested}, yields the DT/PT correspondence for topological Euler characteristics.
\begin{theorem}[\Cref{thm: DT/PT}]\label{thm: DT PT top intro}
       Let $X=\Tot_C(L_1\oplus L_2)$ be a local curve, where $L_1, L_2$ are line bundles over a smooth projective curve. Then, the following correspondence holds
\begin{align*}
       \DT^{\mathrm{top}}_d(X, q)=  \DT^{\mathrm{top}}_0(X, q)\cdot   \PT^{\mathrm{top}}_d(X, q).
\end{align*}
\end{theorem}
We remark that even this \emph{numerical} DT/PT correspondence appears to be new in the literature, as  the classical wall-crossing techniques of Toda \cite{Toda_curve_counting_DT_PT} (see also \cite{ST_stoppa_thomas, DR_local_motivic, Bri_Hall_algebras_curve_counting} for analogous settings)  seem not to have been employed outside the Calabi-Yau regime.
\subsection{Strategy of the proof}
Our strategy to prove \Cref{thm: full DT Intro} avoids degenerating the base curve $C$, but relies solely on localisation methods. To achieve this, we introduce a novel moduli space of \emph{nested sheaves}.
\subsubsection{Skew nested Hilbert schemes}
Let $C$ be a smooth projective curve, $\lambda$ a Young diagram and $\bn=(n_\Box)_{\Box\in \BZ^{2}_{\geq 0}\setminus \lambda }$ a \emph{skew plane partition}, that is a labelling  of $\BZ^{2}_{\geq 0}\setminus \lambda $ non-increasing along rows and columns, with only finitely many non-zero entries. We introduce the \emph{skew nested Hilbert scheme of points} $C^{[\bn]}$ to be the moduli space parametrising flags of  zero-dimensional subschemes $(Z_\Box)_{\Box\in \BZ^{2}_{\geq 0}\setminus \lambda }\subset C$, \emph{nested} according to the opposite poset structure of $\BZ^2$ and  such that $\chi(\oO_{Z_\Box})=n_{\Box}$. In other words, a typical  point of $C^{[\bn]}$ parametrises flags of the form
\begin{equation*}
  \begin{tikzcd}
  &&&Z_{03}\arrow[r, phantom, "\supset"]\arrow[d, phantom, "\cup"]&\dots\\
    &&Z_{12}\arrow[d, phantom, "\cup"]\arrow[r, phantom, "\supset"]&Z_{13}\arrow[d, phantom, "\cup"]\arrow[r, phantom, "\supset"]&\dots\\
        Z_{20}\arrow[d, phantom, "\cup"]\arrow[r, phantom, "\supset"]&Z_{21}\arrow[d, phantom, "\cup"]\arrow[r, phantom, "\supset"]&Z_{22}\arrow[d, phantom, "\cup"]\arrow[r, phantom, "\supset"]&\dots &\\
          \dots&\dots&\dots& &
  \end{tikzcd}  
\end{equation*}
We show that the skew nested Hilbert schemes $C^{[\bn]}$ are reduced, local complete intersections (see \Cref{prop: dim of skew nested}) and provide a formula for their motives in $K_0(\Var_\BC)$ (see \Cref{motives of spp}).
\subsubsection{Localisation}
The connected components of the $\TT$-fixed locus $\Hilb^n(X, \beta)^\TT$  are skew nested Hilbert schemes of points $C^{[\bn]}$, for suitable skew plane partitions $\mathbf{n}$ and Young diagram $\lambda$. In fact, by pushing forward via $X\to C$, a $\TT$-fixed quotient $[\oO_X\onto \oO_Z]$ corresponds an equivariant decomposition $\bigoplus_{(i,j)\in \BZ^2}[\oO_C\onto F_{ij}]$ on $C$, where every $F_{ij}$ is  a rank  at most 1 sheaf on $C$. These data produce divisors $Z_{ij}\subset C$ satisfying the nesting conditions dictated  by $\bn$, in other words an element of $ C^{[\mathbf{n}]}$, see \Cref{thm: fixed}.

On each connected component, by Graber-Pandharipande localisation \cite{GP_virtual_localization} there is an induced virtual fundamental class $[C^{[\mathbf{n}]}]^{\vir}$, which we show to coincide with the ordinary fundamental class, see \Cref{thm: equality virtual classes}. The DT partition function is therefore reduced to the $\TT$-equivariant intersection numbers
\begin{align*}
\widehat{\chi}\left(C^{[\mathbf{n}]},\widehat{\mathfrak{e}}(-N_{C,L_1,L_2, \bn}^{\vir})  \right)\in \BQ(t_1^{1/2}, t_2^{1/2}),
\end{align*}
where $ N_{C,L_1,L_2, \bn}$ is the virtual normal bundle and $\widehat{\chi}(\cdot)$ is defined in \eqref{eqn: chi hat}. The generating series of the latter is controlled by three universal series (\Cref{thm: universal series})
\begin{align}\label{eqn: Universal INTRO}
   \sum_{\mathbf{n}}q^{|\mathbf{n}|} \widehat{\chi}\left(C^{[\mathbf{n}]},\widehat{\mathfrak{e}}(-N_{C,L_1,L_2, \bn}^{\vir})  \right)= \widehat{A}_{\lambda}(q)^{1-g}\cdot \widehat{B}_{\lambda}(q)^{\deg L_1}\cdot \widehat{C}_{\lambda}(q)^{\deg L_2},
\end{align}
which we show by adapting the  strategy of \cite{EGL_cobordism} to the techniques developed in  our previous work \cite{Mon_double_nested}.
\subsubsection{Toric computations}
By the universal structure \eqref{eqn: Universal INTRO} the universal series are determined by their constant terms, and the evaluation of the generating series for the triples 
\begin{align*}
    (\BP^1, \oO, \oO), \quad 
     (\BP^1, \oO, \oO(-2)), \quad 
      (\BP^1, \oO(-2), \oO).
\end{align*}
The constant terms satisfy as well a suitable universal structure (see \Cref{cor: virt norm operators}), which reduces their computation to the combinatorics of the Hilbert scheme of points $\Hilb^n(\BA^2)$, see \Cref{prop: virtual normal const as univer}. 

Once suitably normalising \eqref{eqn: Universal INTRO} by their constant terms (see \Cref{eqn: lemma reduced}), specialising to the genus 0 case $C=\BP^1$ allows a further $\BC^*$-localisation on $(\BP^1)^{[\bn]}$. The $\BC^*$-fixed locus of the latter consists of pairs of skew plane partitions, which are in bijective correspondence with pairs of \emph{plane partitions} with one infinite leg. After a smart choice of a $\BC^*$-equivariant lift of the $\TT$-equivariant virtual normal bundle $N^{\vir}$, each $\TT\times \BC^*$-localised contribution to the left-hand-side of \eqref{eqn: Universal INTRO} takes the form of the ($K$-theoretic version of the) vertex terms of \cite{MNOP_1}. By summing over all skew plane partitions of fixed shape, it follows then that 
\begin{align*}
     \sum_{\mathbf{n}}q^{|\mathbf{n}|} \chi\left((\BP^1)^{[\mathbf{n}]},\widehat{\mathfrak{e}}(-N_{\BP^1,L_1,L_2, \bn}^{\mathrm{red}})  \right)=\left.\left(\widehat{\mathsf{V}}_\lambda(q)\cdot \widehat{\mathsf{V}}_\lambda(q)|_{t_1=t_1t_3^{-\deg L_1}, t_2=t_2t_3^{-\deg L_2}, t_3=t_3^{-1}}\right)\right|_{t_3=1},
\end{align*}
see \Cref{cor: vertex top}.
\subsection{Equivariant cohomology}
Ordinary DT theory deals with the study of  partition functions of the form
\begin{align}\label{eqn: INTRO cohom}
    \DT_d(X, q)=\sum_{n\in \BZ}q^n\cdot  \int_{[  \Hilb^n(X, \beta)]^{\vir}}1\in \BQ(s_1,s_2)(\!( q )\!),
\end{align}
where the integration takes values in equivariant cohomology, see e.g.~\cite{OP_local_theory_curves}. We confirm that our results are \emph{ipso facto} refined, meaning that  \eqref{eqn: INTRO cohom} is a limit of \eqref{eqn: DT intro def}.
\begin{prop}[\Cref{prop: limit equiv}]\label{prop: limit equiv INTRO}
  Let $X=\Tot_C(L_1\oplus L_2)$ be a local curve, where $L_1, L_2$ are line bundles over a smooth projective curve $C$ of genus $g$. We have
    \[\lim_{b\to 0}b^{d(2-2g+\deg L_1+\deg L_2)}\cdot \widehat{\DT}_d(X, q)|_{t_1=e^{bs_1}, t_2= e^{bs_2}}=\DT_d(X, q).\]
\end{prop}
Notice that the degree of the normalisation in \Cref{prop: limit equiv INTRO} relates to the virtual dimension of $\Hilb^n(X, \beta)$ by
\[
\mathrm{vdim} \Hilb^n(X, \beta)=d(2-2g+\deg L_1+\deg L_2).
\]
\subsection{Refined GW/PT}
In the very recent seminal paper \cite{BS_refinedGW}, Brini-Schuler introduced  \emph{refined} Gromov-Witten theory as follows. Let $X$ be a smooth quasi-projective Calabi-Yau threefold, with a torus $\overline{\BT}$ action having proper fixed locus, and scaling non-trivially the canonical bundle $\omega_X$. Set $Z=\Tot_X(\oO\oplus \oO)\to X$ to be a Calabi-Yau fivefold, endowed with an extra $\BC_q^*$-action  scaling the fibers, so that the $\BT=\overline{\BT}\times \BC_q^*$ preserves the volume form of $Z$. For a curve class $\beta\in H_2(X, \BZ)$,  Brini-Schuler define the \emph{refined} Gromov-Witten partition function of $X$ to be
\begin{align}\label{eqn. refined GW}
\GW_\beta(Z,\BT)=\sum_{g\geq 0}\int_{[\overline{ M}_g(Z, \beta)]^{\vir}}1,
\end{align}
where $\overline{ M}_g(Z, \beta)$ is the moduli space of disconnected stable maps to $Z$, and the integration is defined by $\BT$-equivariant residues \cite{GP_virtual_localization}. The key feature of \eqref{eqn. refined GW} is that, thanks to Mumford's relations, a suitable specialisation of the $\BT$-equivariant parameters recovers the ordinary Gromov-Witten theory of $X$. Leveraging on  the fivefold perspective in GW theory, Brini-Schuler proposed a natural refinement of the  \emph{Gopakumar-Vafa invariants} of $X$ \cite{GV_I, GV_II} in line with the physics expectations of \cite{CKK_refined_BPS}, as well as a \emph{refined} GW/PT correspondence \cite[Conj.~1.3]{BS_refinedGW}
\begin{align}\label{eqn: ref GW DT}
    \GW_\beta(Z,\BT)\stackrel{?}{=}\ch_{\BT}\, \widehat{\PT}_\beta(X,-q), 
\end{align}
 where  $\widehat{\PT}_\beta(X,q)$ is the $K$-theoretic ($\overline{\BT}$-equivariant) PT partition function of $X$ in class $\beta$ in the  variable $q$ seen as a torus coordinate of $\BT$,  and $\ch_{\BT}$ is the $\BT$-equivariant Chern character. Remarkably, the correspondence \eqref{eqn: ref GW DT} encodes  the variable keeping track of the holomorphic Euler characteristic in PT theory in the $\BT$-equivariance, offering a more conceptual understanding of the analytic change of variable of \cite{MNOP_1},  more symmetrically from the fivefold point of view. Following the circle of ideas of Pardon \cite{Pardon_GWDT}, we believe that a possible path to solve \eqref{eqn: ref GW DT} would be to prove the conjectural relation in the local curves case. In this case, our main results offer a solution for the PT side, and allow a tangible prediction for the GW side.

\subsection{TQFT} In \cite{BP_local_GW_curves} (resp.~\cite{OP_local_theory_curves}) the authors followed a TQFT approach to compute partition functions arising in ordinary Gromov-Witten (resp.~Donaldson-Thomas) theory of local curves. Roughly speaking, this amounts to first apply the \emph{degeneration formulas} to reduce the absolute invariants of higher genera local curves to \emph{relative} invariants of local $\BP^1$, and then to compute the "building blocks" of the TQFT via $\BC^*$-localisation. This strategy appears widely in the literature, see e.g.~\cite{BP_local_GW_curves, OP_local_theory_curves,  GI_klein, Oko_Takagi,  AKMV_topvert, IKV_topological_vertex, AS_black_holes} for some examples. A drawback of this approach is that it forces to compute relative invariants to determine the partition functions of \emph{absolute} invariants, involving cumbersome degeneration formulas.

This present work completes the program initiated in \cite{Mon_double_nested}, effectively computing the partition functions of \emph{refined} Donaldson-Thomas invariants of local curves without relying on  degeneration techniques and relative invariants. Our key insight is to obtain the universal expressions in \Cref{thm: full DT Intro} by analysing the \emph{geometry} of the fixed loci of the relevant Hilbert schemes, rather than the \emph{enumerative} invariants attached to them, thus revealing at once the universal structures underlying both the unrefined and the refined Donaldson-Thomas theory.

\subsection{Future directions} We conclude by proposing further directions that can be approached with the novel techniques developed in the present paper.
\subsubsection{Rationality}
As a consequence of the conjectural refined GW/PT correspondence \eqref{eqn: ref GW DT}, the PT partition function is expected to be a \emph{rational} function, see \cite[Sec.~1.4.2]{BS_refinedGW}. For local curves, we establish rationality in degree one (\Cref{cor: intro DT 1}), in the refined limit (\Cref{thm: refined DT full Aga INTRO}), and under the antidiagonal specialization of the equivariant parameters (\Cref{thm: DT anti diagonal INTRO}). We believe that rationality in the general case is plausibly within reach, by adapting the approach of \cite{PP_descendants_local_curves}, who treated the unrefined case by reducing to the rationality of the \emph{capped} 1-leg vertex, see also \cite{KOO_2_legDT}.

\subsubsection{Insertions} Motivated by the M2-brane index perspective \cite{NO_membranes_and_sheaves}, we addressed $K$-theoretic DT theory without further insertions. It would be interesting to adapt  the work of Schimpf \cite{Sch_double} to our setting,  who addressed the case of \emph{descendent} insertions for the unrefined PT theory of local curves, whose upshot is that \emph{the structure of stable pair invariants
is induced by the structure of the Bethe roots}, see \cite[pag.~8]{Sch_double}.

\subsubsection{Local orbicurves}
Chuang-Diaconescu-Donagi-Pantev \cite{CDDP_parabolic} gave a string theoretic derivation of the conjecture of Hausel-Letellier-Rodriguez--Villegas on the cohomology of character varieties with marked points \cite{HLRV_character} (see also \cite{Mell_char_var_MD, Mell_Higgs_no_punc}). Their formula is identified with a \emph{refined BPS expansion} in the PT theory of
a \emph{local root stack}, an orbifold version of our local curves. It would be intriguing to approach \cite[Conj.~(1.5)]{CDDP_parabolic} by adapting our techniques to the orbifold setting. In fact, the seemingly mysterious change of variables proposed in \cite[Conj.~(1.5)]{CDDP_parabolic} aligns precisely with the introduction of the variables $t_4, t_5$ as in \eqref{eqn: new variables}.

\subsubsection{Higher rank theory}
Nekrasov-Okounkov \cite[Sec.~5.4]{NO_membranes_and_sheaves} proposed a \emph{pragmatic} approach to \emph{higher rank} refined PT theory of a Calabi-Yau threefold $X$ via virtual localisation, and conjecturally related it to the M2-brane index of an  $\CA_{r-1}$-surface fibration over $X$. It would be interesting to recast their higher rank theory with our techniques in the case $X\to C$ is a local curve, where the skew nested Hilbert schemes $C^{[\bn]}$ would be suitably replaced by some moduli spaces of skew  nested higher rank sheaves on $C$.

\subsection*{Acknowledgments}
 We are grateful to Jim Bryan,  Martijn Kool, Danilo Lewański,  Henry Liu, Miguel Moreira, Georg Oberdieck, Tudor Pădurariu, Andrea Ricolfi, Andrea Sangiovanni  and Sebastian Schlegel Mejia for helpful discussions. We especially thank Yannik Schuler for many insightful conversations on the M-theoretic interpretation of refined Donaldson–Thomas theory, for pointing out the prediction for the topological string partition function in \cite{AS_black_holes} and for gently sharing a fundamental note with me \cite{Schuler-private}. We are grateful to the anonymous referee for useful comments. S.M. was supported by the FNS Project 200021-196960 ``Arithmetic aspects of moduli spaces on
curves'', and by  the HORIZON-MSCA-2024-PF-01 Project 
101203281 ``Moduli Spaces of Sheaves: Geometry and Invariants'', funded by the Research and Innovation framework programme Horizon Europe {\normalsize\euflag}.  S.M. is a  member of GNSAGA of INDAM, Italy.
\section{Skew nested Hilbert schemes}
\subsection{Young diagrams}\label{sec: notation partitions}
By definition, a \emph{partition} $\lambda$ of $d\in \BZ_{\geq 0}$ is a finite sequence of positive integers 
\[
\lambda=(\lambda_0\geq \lambda_1\geq \lambda_2\geq \dots),
\]
where 
\[
|\lambda|=\sum_{i}\lambda_i=d.
\]
A partition $\lambda$ can be equivalently pictorially described by its associated \emph{Young diagram}, which is the collection of $d$ boxes  in $\BZ^2$ located at $(i,j)$ where $0\leq j< \lambda_{i}$.\footnote{In our conventions, $(i,j)$  labels the corner of the box closest to the origin. In displaying Young diagrams, we denote by $i$ the vertical (negative) axis and by $j$ the horizontal (positive) axis.} The \emph{conjugate partition} $\overline{\lambda}$ is obtained by reflecting the Young diagram of $\lambda$ about the $i=j$ line.

\begin{definition}\label{def: spp}
Let $\lambda$ be a Young diagram.
\begin{itemize}
    \item A \emph{skew plane partition} of shape $\BZ_{\geq 0}^2\setminus \lambda$ is a collection of  nonnegative integers $\bn=(n_{\Box})_{\Box\in \BZ_{\geq 0}^2\setminus\lambda}$ non-increasing along rows and columns, with only finitely many non-zero entries. In other words, we have  $n_{\Box}\leq n_{\Box'} $ whenever $\Box\geq \Box'$\footnote{This means that if $\Box=(i,j), \Box'=(i',j')$, then $i\geq i'$ and $j\geq j'$.}.
    \item A \emph{reversed plane partition} of shape $\lambda$  is a collection of  nonnegative integers $\mathbf{m}=(m_{\Box})_{\Box\in \lambda}$ non-decreasing along rows and columns. In other words, we have  $m_{\Box}\leq m_{\Box'} $ if $\Box\leq \Box'$.
\end{itemize}
\end{definition}
The \emph{size} of a skew (resp.~reversed) plane partition $\bn$ (resp. $\Bm$) is
\begin{align*}
    |\mathbf{n}|=\sum_{\Box\in \BZ^2_{\geq 0}\setminus\lambda}n_{\Box},\\
       |\mathbf{m}|=\sum_{\Box\in \lambda}m_{\Box}.
\end{align*}
\begin{figure}[!h]
    \centering
   \yng(3,2,2,1) \quad \young(015,12,13,5) \quad \young(:::311,::511,::4,:32,1)
    \caption{Respectively from the left, a  Young diagram of size 8, a reversed plane partition of size 18 and a skew plane partition of size 22.}
    \label{fig:my_label}
\end{figure}

\subsubsection{Socle} 
Let $\lambda$ be a Young diagram. Recall from \cite[Def.~2.16]{GLMRS_double-nested-1} that the \emph{socle} of $\lambda$ is the set of maximal elements of $\lambda$ with respect to the partial order of $\lambda$, equivalently, 
\[
\Soc(\lambda) =\Set{ (i,j) \in \lambda \, | \, (i+1, j), (i,j+1) \notin \lambda}\subset\lambda,
\]
and the \emph{subsocle} of $\lambda$ is 
\[
\Subsoc(\lambda)
=\Set{ (i,j) \in \lambda \, | \,
(i+1, j), (i,j+1) \in \lambda, (i+1,j+1) \notin \lambda}\subset\lambda.
\]
Let $\Bm$ be a reverse plane partition of shape $\lambda$. Following \cite[Prop.~2.17]{GLMRS_double-nested-1}, we define the quantity
\begin{align*}
    \omega(\Bm)=\sum_{\Box\in \Soc(\lambda)}m_{\Box}-\sum_{\Box\in \Subsoc(\lambda)}m_{\Box}.
\end{align*}
\begin{definition}
    Let $\lambda$ be a Young diagram. We define the \emph{cosocle} of $\lambda$ to be the set of minimal elements of $\BZ^2_{\geq 0}\setminus\lambda$ with respect to the partial order, equivalently, 
    \[
\CoSoc(\lambda) =\Set{ (i,j) \in \BZ^2_{\geq 0}\setminus\lambda \, | \, (i-1, j), (i,j-1) \notin \BZ^2_{\geq 0}\setminus\lambda}\subset\BZ^2_{\geq 0}\setminus\lambda,
\]
and the \emph{cosubsocle} of $\lambda$ to be
\[
\CoSubsoc(\lambda)
=\Set{ (i,j) \in  \BZ^2_{\geq 0}\setminus\lambda \, | \,
(i-1, j), (i,j-1) \in \BZ^2_{\geq 0}\setminus\lambda, (i-1,j-1) \notin \BZ^2_{\geq 0}\setminus\lambda}\subset\BZ^2_{\geq 0}\setminus\lambda.
\]
    \end{definition}
Let $\bn$ be a skew plane partition of shape $\BZ^2_{\geq 0}\setminus\lambda$. We define the quantity
\begin{align*}
      \omega(\bn)=\sum_{\Box\in \CoSoc(\lambda)}n_{\Box}-\sum_{\Box\in \CoSubsoc(\lambda)}n_{\Box}.
\end{align*}
\begin{figure}[!h]
    \centering
   \young(\,\times\bullet,\,\,,\times\bullet,\bullet)  \quad \young(:::\bullet\,\,,::\bullet\times\,,::\,,:\bullet\times,\bullet\times)
    \caption{We denote by "$\bullet$" the (co)socle  and by "$\times$" the (co)subsocle of the Young diagram in \Cref{fig:my_label}.}
    \label{fig:socle}
\end{figure}
\subsubsection{Hooklength} Let $\lambda$ be a Young diagram and $\Box=(i,j)\in \lambda$. We define the \emph{arm length} $a(\Box)$, the   \emph{leg length} $\ell(\Box)$ and the \emph{hooklength} $h(\Box)$ by
\begin{align*}
    a(\Box)&=|\set{(i,j')\in \lambda| j<j'}|,\\
    \ell(\Box)&=|\set{(i',j)\in \lambda|i<i'}|,\\
    h(\Box)&=a(\Box)+\ell(\Box)+1.
\end{align*}
The \emph{squared norm} of $\lambda$ is defined as 
\begin{align*}
\left \lVert \lambda \right \rVert^2 = \sum_{i\geq 0} \lambda_i^2,
\end{align*}
and set the quantity
\[
n(\lambda)=\sum_{i\geq 0}i\cdot\lambda_{i}.
\]
We will use the following standard identities throughout the paper
\begin{align}\label{eqn: comb ident}
\begin{split}
         n(\lambda)&=\sum_{\Box \in \lambda} \ell(\Box),\\
    n(\overline{\lambda})&=\sum_{\Box \in \lambda} a(\Box),\\
    \frac{\left \lVert \lambda \right \rVert^2}{2}&=n(\overline{\lambda})+\frac{|\lambda|}{2}.
    \end{split}
\end{align}
\begin{figure}[!h]
    \centering
   \young(\,\,\,\,\,\,\,\,,\,\Box\bullet\bullet,\,\bullet\,,\,\bullet,\,\bullet)
    \caption{The number of "$\bullet$" on the right (resp.~below) of $\Box$  is the arm (resp.~leg) length of $\Box$.}
    \label{fig:arm leg}
\end{figure}
\subsection{Skew nested Hilbert schemes}
Let $X$ be a quasi-projective scheme, $\lambda$ a Young diagram and $\Bm=(m_\Box)_{\Box\in \lambda}$ a reverse plane partition of shape $\lambda$. Recall that we defined in \cite[Def.~2.2]{Mon_double_nested} the \emph{ double nested Hilbert functor of points}
\[
 \underline{\Hilb}^{\Bm}(X): \Sch^{\op}\to \Sets,
\]
\[  T\mapsto \left\{ \begin{tabular}{l|l}$
(\CZ_{\Box})_{\Box\in \lambda}\subset X\times T$ & $\CZ_{\Box}\mbox{ a } T\mbox{-flat closed subscheme with }  \CZ_{\Box}|_t \mbox{ 0-dimensional for all } t\in T,$\\ 
    &  $\chi(\oO_{\CZ_{\Box}|_t})=m_\Box, \mbox{ such that } \CZ_{\Box} \subset \CZ_{\Box'} \mbox{ for } \Box\leq \Box'.$
    \end{tabular}
    \right\}, \]
which is representable by  the \emph{double nested Hilbert scheme}  $  X^{[\Bm]}$, cf.~\cite[Prop.~2.4]{Mon_double_nested}. Moreover, if $X=C$ is a smooth quasi-projective curve, the double nested Hilbert scheme $C^{[\Bm]}$ is a reduced local complete intersection of pure dimension $\omega(\Bm)$, cf.~\cite[Thm.~4.2]{GLMRS_double-nested-1}.  The closed points of $X^{[\Bm]}$ parametrise flags of the form
\begin{equation*}
  \begin{tikzcd}
    Z_{00}\arrow[r, phantom, "\subset"]\arrow[d, phantom, "\cap"] &Z_{01}\arrow[d, phantom, "\cap"]\arrow[r, phantom, "\subset"]&Z_{02}\arrow[d, phantom, "\cap"]\arrow[r, phantom, "\subset"]\arrow[d, phantom, "\cap"]&Z_{03}\arrow[r, phantom, "\subset"]\arrow[d, phantom, "\cap"]&\dots\\
      Z_{10}\arrow[r, phantom, "\subset"]\arrow[d, phantom, "\cap"]&Z_{11}\arrow[d, phantom, "\cap"]\arrow[r, phantom, "\subset"]&Z_{12}\arrow[d, phantom, "\cap"]\arrow[r, phantom, "\subset"]&Z_{13}\arrow[d, phantom, "\cap"]\arrow[r, phantom, "\subset"]&\dots\\
        Z_{20}\arrow[d, phantom, "\cap"]\arrow[r, phantom, "\subset"]&Z_{21}\arrow[d, phantom, "\cap"]\arrow[r, phantom, "\subset"]&Z_{22}\arrow[d, phantom, "\cap"]\arrow[r, phantom, "\subset"]&\dots &\\
          \dots&\dots&\dots& &
  \end{tikzcd}  
\end{equation*}
where each $Z_{ij}\subset X$ is a zero-dimensional closed subscheme. \\

Similarly, let  $\bn$ be a skew plane partition of shape $\BZ^2_{\geq 0}\setminus \lambda$. We define the \emph{skew nested Hilbert functor of points} $ \underline{\Hilb}^{\bn}(X)$ by 
\[
 \underline{\Hilb}^{\bn}(X): \Sch^{\op}\to \Sets,
\]
\[  T\mapsto \left\{ \begin{tabular}{l|l}$
(\CZ_{\Box})_{\Box\in \BZ^2_{\geq 0}\setminus\lambda}\subset X\times T$ & $\CZ_{\Box}\mbox{ a } T\mbox{-flat closed subscheme with }  \CZ_{\Box}|_t \mbox{ 0-dimensional for all } t\in T,$\\ 
    &  $\chi(\oO_{Z_{\Box}})=n_\Box, \mbox{ such that } \CZ_{\Box} \subset \CZ_{\Box'} \mbox{ for } \Box\geq \Box'.$
    \end{tabular}
    \right\}. \]
A typical closed point of $\underline{\Hilb}^{\bn}(X)$ parametrises flags of the form
\begin{equation*}
  \begin{tikzcd}
  &&&Z_{03}\arrow[r, phantom, "\supset"]\arrow[d, phantom, "\cup"]&\dots\\
    &&Z_{12}\arrow[d, phantom, "\cup"]\arrow[r, phantom, "\supset"]&Z_{13}\arrow[d, phantom, "\cup"]\arrow[r, phantom, "\supset"]&\dots\\
        Z_{20}\arrow[d, phantom, "\cup"]\arrow[r, phantom, "\supset"]&Z_{21}\arrow[d, phantom, "\cup"]\arrow[r, phantom, "\supset"]&Z_{22}\arrow[d, phantom, "\cup"]\arrow[r, phantom, "\supset"]&\dots &\\
          \dots&\dots&\dots& &
  \end{tikzcd}  
\end{equation*}
where each $Z_{ij}\subset X$ is a zero-dimensional closed subscheme.
\begin{prop}\label{prop: dim of skew nested}
Let $X$ be a  quasi-projective scheme, $\lambda$ a Young diagram and  $\bn$ a skew plane partition of shape $\BZ^2_{\geq 0}\setminus \lambda$. Then the moduli functor $ \underline{\Hilb}^{\bn}(X)$ is representable by a quasi-projective scheme, which we denote by $X^{[\bn]}$. 

Moreover, if  $X=C$ is a smooth quasi-projective curve, then the skew nested Hilbert scheme $C^{[\bn]} $ is a reduced  local complete intersection of pure dimension $\omega(\bn)$.
\end{prop}
\begin{proof}
Set $\bn=(n_{\Box})_{\Box\in \BZ^2_{\geq 0}\setminus \lambda}$ and let $N\gg 0$ be a  sufficiently large integer such  that $n_{ij}=0$ for $i,j\geq N$. Define the Young diagrams $\mu, \tilde{\lambda}$ by
\begin{align*}
    \mu=\set{(i,j)\in \BZ^2_{\geq 0}| i,j\leq N},
\end{align*}
\[
\Tilde{\lambda}=\set{(N-i, N-j)\in \BZ^2_{\geq 0}| (i,j)\in \mu\setminus \lambda}\subset \BZ^2_{\geq 0},
\]
where $\tilde{\lambda}$ is the "complement" Young diagram inside the $N\times N$ rectangle $\mu$. Define the reverse plane partition $\Bm$ of shape $\Tilde{\lambda}$ by
\[
m_{ij}=n_{N-i, N-j}, \quad (i,j)\in \Tilde{\lambda}.
\]
Clearly, the moduli functor $ \underline{\Hilb}^{\bn}(X)$ is isomorphic to the moduli functor $ \underline{\Hilb}^{\Bm}(X)$, completing the proof by \cite[Prop.~2.4]{Mon_double_nested} and \cite[Thm.~4.2]{GLMRS_double-nested-1}.
\end{proof}
\subsection{The motives}
We conclude this section presenting an expression for the \emph{motive} $[C^{[\bn]}]\in K_0(\Var_\BC)$, where $C$ is a smooth quasi-projective curve. We remark that, as shown in the proof of \Cref{prop: dim of skew nested}, every skew nested Hilbert scheme $C^{[\bn]}$ is isomorphic to a double nested Hilbert scheme $C^{[\Bm]}$, for a suitable reverse plane partition $\Bm$. Although the motives of the latter were already studied in \cite[Cor.~5.7]{GLMRS_double-nested-1}, we present here an alternative packaging of the invariants of the generating series of the motives $[C^{[\bn]}]$, more suited in view of their application to Donaldson-Thomas theory, see \Cref{cor: top DT}.

We review the basic notions of the Grothendieck ring of varieties and its associated power structure, following the presentation in \cite[Sec.~5]{GLMRS_double-nested-1}.
\subsubsection{Grothendieck ring of varieties}
 The \emph{Grothendieck ring of varieties}, denoted $K_0(\Var_{\BC})$, is the free abelian group generated by isomorphism classes  $[X]$ of 
 finite type $\BC$-varieties\footnote{The results in this section extend naturally to any algebraically closed field $\mathbf{k}$.}, modulo the relations   $[X] = [Z] + [{X\setminus Z}]$ whenever $Z \into X$ is a closed subvariety of $X$.  The operation $ [X]\cdot [Y] = [X\times Y]$
defines  naturally a ring structure on $K_0(\Var_{\BC})$.
\smallbreak
The Grothendieck ring of varieties is naturally equipped with a \emph{power structure},  a morphism
\begin{align*}
   (1+tK_0(\Var_{\BC})\llbracket t \rrbracket) \times K_0(\Var_{\BC})&\to 1+tK_0(\Var_{\BC})\llbracket t \rrbracket\\
   (A(t),m) &\mapsto A(t)^m,
\end{align*}
satisfying natural compatibilities, see e.g.~\cite{GLMps,GLMps2}. In particular, it is explained in \cite{GLMps2} how to extend the power structure on $K_0(\Var_{\BC})$ to the \emph{multivariable setting}, cf.~also \cite[Sec.~5.1.1]{GLMRS_double-nested-1}.
\subsubsection{The generating series}
Let $\lambda$ be a Young diagram, and let   $\mathbf{q}=(q_{\Box})_{\Box\in \BZ^2_{\geq 0}}$  be a collection of variables indexed by the lattice coordinates of $\BZ^2_{\geq 0}$. For a skew plane partition $\bn$ of shape $ \BZ^2_{\geq 0}\setminus \lambda$, set
\begin{align*}
    \mathbf{q}^{\bn}=\prod_{\Box\in \BZ^2_{\geq 0}}q_{\Box}^{n_\Box}.
\end{align*}
We define the generating series of  motives $C^{[\bn]}$ to be
\begin{align*}
\mathsf{Z}_{C,\lambda}(\mathbf{q})=\sum_{\bn}[C^{[\bn]}]\cdot \mathbf{q}^{\bn}   \in K_0(\Var_\BC)\llbracket \mathbf{q}\rrbracket,
\end{align*}
where the sum is taken over all skew plane partitions $\bn$ of shape $\BZ^{2}_{\geq 0}\setminus \lambda$. Similarly, we set 
\begin{align*}
\mathsf{Z}^{\mathrm{spp}}_{\lambda}(\mathbf{q})=\sum_{\bn}\mathbf{q}^{\bn}   \in \BZ\llbracket \mathbf{q}\rrbracket
\end{align*}
to be the generating series of skew plane partitions of shape $ \BZ^{2}_{\geq 0}\setminus \lambda$.
\begin{prop}\label{motives of spp}
  Let $\lambda$ be a Young diagram and $C$ be a smooth quasi-projective curve. There is an identity
        \[
        \mathsf{Z}_{C,\lambda}(\mathbf{q})=\mathsf{Z}^{\mathrm{spp}}_{\lambda}(\mathbf{q})^{[C]}.
        \]
        In particular, we have
        \[
       \left.  \mathsf{Z}_{C,\lambda}(\mathbf{q})\right|_{q_{\Box}=q}= \left(\prod_{d\geq 1}\frac{1}{(1-q^d)^d} \prod_{\Box\in \lambda} \frac{1}{1-q^{h(\Box)}}\right)^{[C]}.
        \]
\end{prop}
\begin{proof}
Let $ (\BA^{1})^{[\bn]}_0\subset (\BA^1)^{[\bn]}$ denote the \emph{punctual} skew nested Hilbert scheme, that is the locus where all the subschemes parametrised are supported at the origin $0\in \BA^1$. Clearly, we have $ (\BA^{1})^{[\bn]}_0\cong \pt$.
By a simple routine -- coming at once from  the  geometric description of the power structure -- we have
\begin{align*}
     \mathsf{Z}_{C,\lambda}(\mathbf{q})&=\left( \sum_{\bn}[(\BA^1)_0^{[\bn]}]\cdot \mathbf{q}^{\bn}  \right)^{[C]}\\
     &=\mathsf{Z}^{\mathrm{spp}}_{\lambda}(\mathbf{q})^{[C]},
\end{align*}
see e.g.~\cite[Thm.~5.5]{GLMRS_double-nested-1} for an analogous argument.

    The second claim follows by \cite[Thm.~2.1]{Sagan_comb_proof}, which computes  the generating series 
    \[
   \sum_{\bn}q^{|\mathbf{n}|}=\prod_{d\geq 1}\frac{1}{(1-q^d)^d} \prod_{\Box\in \lambda} \frac{1}{1-q^{h(\Box)}}.
    \]
\end{proof}

\section{Donaldson-Thomas theory of local curves}
\subsection{Hilbert schemes}\label{sec: Hilbert schemes}
Let $X=\Tot_C(L_1\oplus L_2)$ be a \emph{local curve}, that is the total space of two line bundles $L_1, L_2$ on a smooth projective curve $C$.  For a curve class $\beta=d[C]\in H_2(X, \BZ)$ and $n\in \BZ$, consider the Hilbert scheme $\Hilb^n(X, \beta)$, parametrising flat families of closed subschemes $Z\subset X$ with $[\Supp Z]=\beta $ and $\chi(\oO_Z)=n$. Equivalently, this moduli space parametrises flat families of quotients
\[[\oO_X\onto \oO_Z]\in \Coh(X)\]
with fixed Chern character $\ch(\oO_Z)$.\\
By the work of   Huybrechts-Thomas \cite{HT_obstruction_theory},  the deformation theory of complexes   gives rise to  a \emph{perfect obstruction theory} on $\Hilb^n(X, \beta)$ 
\begin{equation}\label{eqn: obstruction theory}
     \BE=\RR\pi_*\RR\hom(\CI, \CI)^\vee_0[-1]\to \BL_{\Hilb^n(X, \beta)},
\end{equation}
where 
$(\cdot)_0$ denotes the trace-free part,
 $\pi:X\times \Hilb^n(X, \beta)\to \Hilb^n(X, \beta)$ is the canonical projection and 
 \begin{align*}
     0\to \CI\to \oO\to \oO_{\CZ} \to 0
 \end{align*}
 denotes the universal exact sequence on $ \Hilb^n(X, \beta)\times X$.

By \cite{BF_normal_cone, CFK_virtual_fundamental_dg}, the Hilbert scheme $ \Hilb^n(X, \beta)$  is endowed with \emph{virtual fundamental cycles} in homology and $K$-theory
\begin{align*}
    [ \Hilb^n(X, \beta)]^{\vir}\in A_*(\Hilb^n(X, \beta)),\\
    \oO_{\Hilb^n(X, \beta)}^{\vir}\in K_0(\Hilb^n(X, \beta)).
\end{align*}
 However,   since $ \Hilb^n(X, \beta)$ is  generally not proper,  we cannot directly define  invariants by means of intersection theory. Nevertheless, the algebraic torus $\TT=(\BC^*)^2$ acts on $X$ by scaling the fibers,  and the action naturally lifts to $ \Hilb^n(X, \beta) $, endowing the perfect obstruction theory with a natural $\TT$-equivariant structure by \cite{Ric_equivariant_Atiyah} and consequently $\TT$-equivariant virtual cycles by the $\TT$-equivariant version of \cite{BF_normal_cone, CFK_virtual_fundamental_dg}. Moreover, the $\TT$-fixed locus $ \Hilb^n(X, \beta)^\TT$ is proper (cf. \Cref{thm: fixed}), therefore by Graber-Pandharipande \cite{GP_virtual_localization} there is a natural  induced perfect obstruction theory
 \begin{align*}
        \BE|^{\fix}_{\Hilb^n(X, \beta)^\TT}\to \BL_{\Hilb^n(X, \beta)^\TT}
 \end{align*}
 on  $ \Hilb^n(X, \beta)^\TT$ along with natural virtual cycles. We define \emph{$\TT$-equivariant  Donaldson-Thomas invariants} as 
 \begin{align}\label{eqn: localized PT invariants}
    \DT_{d,n}(X)=\int_{[  \Hilb^n(X, \beta)]^{\vir}}1\in \BQ(s_1,s_2),
\end{align}
where the right-hand-side is defined by
 Graber-Pandharipande virtual localisation formula \cite{GP_virtual_localization} as
\begin{align*}
 \int_{[  \Hilb^n(X, \beta)]^{\vir}}1=\int_{[  \Hilb^n(X, \beta)^{\TT}]^{\vir}}\frac{1}{e(N^{\vir})}\in \BQ(s_1,s_2),
\end{align*}
where $s_1, s_2$ are the generators of $\TT$-equivariant cohomology $H^*_\TT(\pt)$ and the virtual normal bundle is defined by the movable part
\begin{align}\label{eqn: virtual normal bundle}
    N^{\vir}=(\BE|^\vee_{  \Hilb^n(X, \beta)^{\TT}})^{\mov}\in K^0_\TT( \Hilb^n(X, \beta)^{\TT}).
\end{align}
\subsubsection{$K$-theoretic invariants}\label{sec: K-th inva DT}

For a scheme $\CM$ endowed with a perfect obstruction $\BE$, denote by $K_{\vir}=\det \BE$ the \emph{virtual canonical bundle} of $\CM$. Following Nekrasov-Okounkov \cite{NO_membranes_and_sheaves}, we define the \emph{twisted virtual structure sheaf}
\begin{align*}
    \widehat{\oO}^{\vir}_\CM=\oO_\CM^{\vir}\otimes K^{1/2}_{\vir}\in K^0\left(\CM, \BZ\left[\tfrac{1}{2} \right]\right),
\end{align*}
where $ K^{1/2}_{\vir}$ is a \emph{square root}\footnote{The invariants \eqref{eqn: localized DT KK invariants} do not depend on the choice of the square root $K^{1/2}_{\vir}$, since all possible choices differ by a 2-torsion element in the Picard group, as explained in \cite[Sec.~2.6]{Arb_K-theo_surface} via a virtual Hirzebruch-Riemann-Roch argument \cite{FG_riemann_roch}.} of $K_{\vir}$, that is $\left(K^{1/2}_{\vir}\right)^{\otimes 2}=K_{\vir}$. Notice that, while the square root of a line bundle may not exist as a genuine line bundle, it does exist\footnote{To be precise, to define the square root $K_{\vir}^{1/2}$ in $\TT$-equivariant $K$-theory, it is necessary to pass to the \emph{localised} $K$-theory, see \cite[Sec.~7]{OT_1}.} as a class in $K$-theory after inverting 2, see \cite[Sec.~5.1]{OT_1}. We define \emph{$K$-theoretic  Donaldson-Thomas invariants} as
\begin{align}\label{eqn: localized DT KK invariants}
    \widehat{\DT}_{d,n}(X)=\chi\left(\Hilb^n(X, \beta),   \widehat{\oO}^{\vir}  \right)\in \BQ(t_1^{1/2}, t_2^{1/2}),
\end{align}
where the right-hand-side is defined by the virtual localisation formula in $K$-theory \cite{FG_riemann_roch, Qu_virtual_pullback} as
\begin{align*}
   \chi\left(\Hilb^n(X, \beta),   \widehat{\oO}^{\vir}  \right)=\chi\left(\Hilb^n(X, \beta)^\TT,\frac{  \widehat{\oO}^{\vir}_{\Hilb^n(X, \beta)^\TT} }{\widehat{\mathfrak{e}}(N^{\vir})}  \right)\in \BQ(t_1^{1/2}, t_2^{1/2}),
\end{align*}
where $ t_1, t_2$ are the generators of the $\TT$-equivariant $K$-theory (cf.~\Cref{sec: torus weights}) and $\widehat{\mathfrak{e}}$ is defined, for every vector bundle $\CF$,  as
\begin{align*}
    \widehat{\mathfrak{e}}(\CF) =(\det \CF)^{1/2} \otimes \sum_{i\geq 0} (-1)^i\Lambda^i(\CF^*) \in K^0_\TT(\Hilb^n(X, \beta)^\TT),
\end{align*}
 while $\widehat{\mathfrak{e}}(N^{\vir})$ is defined for the $K$-theory class $N^{\vir}$ in \cite[Sec.~7]{OT_1}.
 
The main goal of this work is to  compute the associated partition functions
\begin{equation}\label{eqn: DT partition functions both}
\begin{split}
        \DT_d(X, q)&=\sum_{n\in \BZ} \DT_{d,n}(X)\cdot q^n\in \BQ(s_1,s_2)(\!( q )\!),\\
    \widehat{\DT}_d(X, q)&=\sum_{n\in \BZ} \widehat{\DT}_{d,n}(X)\cdot q^n\in \BQ(t_1^{1/2}, t_2^{1/2})(\!( q )\!).
    \end{split}
\end{equation}
Finally, we introduce the  \emph{topological Euler characteristic} analogue of partition functions \eqref{eqn: DT partition functions both} as
\begin{align*}
    \DT^{\mathrm{top}}_d(X, q)=  \sum_{n\in \BZ}e(\Hilb^n(X, \beta))\cdot q^n\in \BZ(\!(q)\!).
\end{align*}
\subsection{Torus representations and their weights}\label{sec: torus weights}
Let $\BT = (\BC^*)^r$ be an algebraic torus, with character lattice $\widehat{\BT}\cong \BZ^r$. Let $K^0_{\BT}(\pt)\cong \BZ[\widehat{\BT}]$ be the $\BT$-representation ring. Any $\BT$-representation $V$ splits as a sum of $1$-dimensional representations called the \emph{weights} of $V$. Each weight corresponds to a character $\mu \in \widehat{\BT}$, and in turn each character corresponds to a monomial $t^\mu = t_1^{\mu_1} \cdots t_r^{\mu_r}$ in the coordinates of $\BT$.  We will  sometimes identify a (virtual) $\BT$-representation with its character. Given a virtual $\BT$-representation $V$, we denote by $\overline{V}$ its dual as a $\BT$-representation.

More generally, if  $X$ is a scheme with a trivial $\BT$-action, every $\BT$-equivariant quasi-coherent sheaf on $X$ decomposes as $F=\bigoplus_{\mu\in \widehat{\BT}}F_\mu\otimes t^\mu $, and its $K$-theory satisfies
\[
K^0_\BT(X)\cong K^0(X)\otimes  \BZ[t_1^{\pm 1},\dots,  t_r^{\pm 1}].
\]
For an integer $a\in \BZ$, we adopt the notation 
\begin{align*}
   \oO_X^a=a\cdot [\oO_X]\in K^0_\BT(X),
\end{align*}
and we will write $\oO_X=1\in  K^0_\BT(X)$ when clear from the context.
\subsection{The fixed locus}\label{sec: fixed locus}
We prove in this section that the connected components of the $\TT$-fixed locus $\Hilb^n(X, \beta)^\TT$ correspond to  skew nested Hilbert schemes of points  $C^{[\mathbf{n}]}$, for some Young diagrams $\lambda$ and some skew plane partitions $\bn$ of shape  $\BZ^2_{\geq 0}\setminus\lambda$. 

For   a Young diagram $\lambda$  and $g,k_1, k_2\in \BZ$ we  define
\begin{align}\label{eqn: f lambda g}
\begin{split}
    \mathbf{f}_{\lambda}(g,k_1,k_2)&=\sum_{(i,j)\in \lambda}(1-g-i\cdot k_1-j\cdot k_2 )\\
    &=|\lambda|(1-g)-k_1\cdot n(\lambda)-k_2\cdot n(\overline{\lambda})\in \BZ.
\end{split}
\end{align}
\begin{theorem}\label{thm: fixed}
Let $C$ be a smooth projective curve of genus $g$ and $L_1, L_2$ line bundles on $C$. Set $X=\Tot_C(L_1\oplus L_2)$ and $\beta=d[C]$. Then,  there exists an isomorphism of schemes
\begin{align*}
    \Hilb^n(X, \beta)^\TT= \coprod_{|\lambda|= d}\coprod_{\mathbf{n}} C^{[\mathbf{n}]},
\end{align*}
where the disjoint union is over all Young diagrams $\lambda$ of size $d$ and skew plane partitions $\bn$ of shape  $\BZ^2_{\geq 0}\setminus\lambda$ satisfying   $n=|\mathbf{n}|+ \mathbf{f}_{\lambda}( g,\deg L_1, \deg L_2)$.
In particular, $ \Hilb^n(X, \beta)^\TT$ is proper.
 \end{theorem}
 \begin{proof}
 The closed $\TT$-fixed points in $  \Hilb^n(X, \beta)^\TT$ are given by  the $\TT$-equivariant quotient $[\oO_X\onto \oO_Z]$ on $X$, see e.g.~\cite{Kool_fixed_points}.
 
     Let $p:X\to C$ denote the natural projection. Given a $\TT$-equivariant quasi-coherent sheaf $F$ on $X$, its pushdown $p_*F$ along the affine morphism $p$ decomposes into  $\TT$-weight spaces. In particular, for all closed $\TT$-invariant subschemes $Z\subset X$ we have that
     \begin{align*}
    p_*\oO_X&=\bigoplus_{i,j\geq 0}L_1^{-i}\otimes L_2^{-j}\otimes t_1^{-i}t_2^{-j},\\
    p_*\oO_Z&=\bigoplus_{(i,j)\in \BZ^2}F_{ij}\otimes t_1^{i}t_2^{j},
\end{align*}
where $F_{ij}$ are quasi-coherent sheaves on $C$. By pushing down the $\TT$-equivariant quotient $[\oO_X\onto \oO_Z]$, we see that
\begin{align}\label{eqn: quotients}
\begin{cases}
    L_1^{-i}\otimes L_2^{-j}\onto F_{ij} &\quad i,j\leq 0,\\
    0\onto F_{ij} &\quad \mbox{else,}
\end{cases}
\end{align}
which  implies that for $i,j\leq 0 $ we have that either 
\[F_{ij}\cong \oO_{Z_{ij}}\otimes  L_1^{-i}L_2^{-j}\]
for some zero-dimensional closed subscheme $Z_{ij}\subset C$, or
\[
F_{ij}\cong L_1^{-i}L_2^{-j}
\]
is a line bundle, and that $F_{ij}=0$ if either $i,j$ is positive.

Since $p$ is affine, we recover the $\oO_X$-module structure of $\oO_Z$ by the $   p_*\oO_X$-action that $ p_*\oO_Z$ carries. This is generated by the action of the $-1$ pieces $L_1^{-1}\otimes t_1^{-1},L_2^{-1}\otimes t_2^{-1}$, so we find that the $\oO_X$-module structure is determined by the maps
\begin{align}\label{eqn: maps weight -1}
\begin{split}
     &F_{ij}\otimes L_1^{-1}\to F_{i-1,j},\\
 &F_{ij}\otimes L_2^{-1}\to F_{i,j-1},
\end{split}
\end{align}
for all $(i,j)\in \BZ^2$. The maps in \eqref{eqn: maps weight -1} clearly commute with the quotients in \eqref{eqn: quotients}. If we set 
\begin{align*}
    G_{ij}=F_{-i,-j}\otimes L_1^{i}\otimes L_2^{j},
\end{align*}
combining \eqref{eqn: quotients}, \eqref{eqn: maps weight -1} yields a commutative diagram of coherent sheaves on $C$
\begin{equation}\label{eqn: diagram with G_i}
\begin{tikzcd}
&& \oO_C\arrow[rr, equal]\arrow[ld, equal]\arrow[ddd]& & \oO_C\arrow[rr, equal]\arrow[ld, equal]\arrow[ddd]& & \oO_C\arrow[r, equal] \arrow[ld, equal]\arrow[ddd]&\dots\\
 &\oO_C\arrow[rr, crossing over, equal]\arrow[ld, equal]& &  \oO_C\arrow[rr, crossing over, equal]\arrow[ld, equal]& & \oO_C\arrow[r, crossing over, equal]\arrow[ld, equal]&\dots&\\
    \dots& &  \dots& & \dots&  & &\\
  && G_{00}\arrow[rr]\arrow[ld]& & G_{01}\arrow[rr]\arrow[ld]& & G_{02}\arrow[r] \arrow[ld]&\dots\\
&G_{10}\arrow[rr]\arrow[ld]\arrow[from=uuu, crossing over]& &  G_{11}\arrow[rr]\arrow[ld]\arrow[from=uuu, crossing over]& & G_{12}\arrow[r]\arrow[ld]\arrow[from=uuu, crossing over]&\dots&\\
    \dots& &  \dots& & \dots& & &\\
\end{tikzcd}
\end{equation}
We now restrict to the indices $i,j\geq 0$. If $G_{ij}$ is a line bundle, then necessarily $G_{ij}\cong \oO_C$. Denote by $\lambda \subset \BZ^2_{\geq 0}$ the collection of lattice-points $(i,j)$ such that $G_{ij}\cong \oO_C$. We have that $\lambda$ is a Young diagram. In fact, by the commutativity of the diagram \eqref{eqn: diagram with G_i}, for any coherent sheaf $G_{ij}$ we have surjective maps
\begin{equation*}
    \begin{tikzcd}
        G_{ij}\arrow[r, two heads]\arrow[d, two heads]& G_{i,j+1}\\
        G_{i+1,j}&
    \end{tikzcd}
\end{equation*}
which implies that if $G_{ij}$ is torsion for some $(i,j)$, then $G_{lk}$ is torsion for all $(l,k)\geq (i,j)$. In particular, all torsion $G_{ij}$ correspond to   zero-dimensional closed subschemes $Z_{ij}\subset C$ of length $n_{ij}$, satisfying the nesting condition
\begin{equation*}
    \begin{tikzcd}
        Z_{ij}\arrow[r,  phantom, "\supset"]\arrow[d, phantom, "\cup"]& Z_{i,j+1}\\
        Z_{i+1,j}&
    \end{tikzcd}
\end{equation*}
Finally, recall that $[\Supp Z]=d[C]\in H_2(X,\BZ)$. This implies that $p_*\oO_Z$ is a rank $d$ coherent sheaf on $C$, which yields that $|\lambda|=d$. 

Therefore, to each $\TT$-fixed point in $  \Hilb^n(X, \beta)^\TT$, we associated a configuration of  zero-dimensional subschemes of $C$, nesting according to the skew plane partition $\bn=(n_{\Box})_{\Box\in \BZ^2_{\geq 0}\setminus\lambda}\in $ of shape  $\BZ^2_{\geq 0}\setminus\lambda$, in other words a closed point of $C^{[\bn]}$. Conversely, to each element of $C^{[\bn]}$ we can associate a $\TT$-fixed point in $  \Hilb^n(X, \beta)^\TT$ by reversing the above construction. By a Riemann-Roch computation, we have that  $n=|\mathbf{n}|+ \mathbf{f}_{\lambda}( g,\deg L_1, \deg L_2)$.

The bijection on closed points we just exhibited  naturally generalises to flat families (see e.g.~\cite[Prop.~3.1]{Mon_double_nested} for an analogous argument), yielding the desired isomorphism.
 \end{proof}

 As a corollary, we compute the generating series of the topological Euler characteristic of $\Hilb^n(X, \beta)$.
 \begin{corollary}\label{cor: top DT}
 Let $C$ be a smooth projective curve of genus $g$ and $L_1,L_2$ line bundle on $C$ and $\beta=d[C]$.  Set $X=\Tot_C(L_1\oplus L_2)$. Then for any $d> 0$ we have
 \begin{align*}
   \DT^{\mathrm{top}}_d(X, q)= \sum_{|\lambda|= d}q^{ d(1-g)-k_1\cdot n(\lambda)-k_2\cdot n(\overline{\lambda})}\cdot \left(\prod_{d\geq 1}\frac{1}{(1-q^d)^d} \prod_{\Box\in \lambda} \frac{1}{1-q^{h(\Box)}}\right)^{2-2g    }.
 \end{align*}
 \end{corollary}
 \begin{proof}
 The topological Euler characteristic of a $\BC$-scheme with a $\TT$-action  is the same of its $\TT$-fixed locus, therefore by \Cref{thm: fixed}
 \begin{align*}
      \sum_{n\in \BZ}e(\Hilb^n(X, \beta))\cdot q^n&=\sum_{|\lambda|= d}\sum_{\mathbf{n}} e\left(C^{[\mathbf{n}]}\right)\cdot q^{|\mathbf{n}|+ \mathbf{f}_{\lambda}( g,\deg L_1, \deg L_2)}\\
      &=\sum_{|\lambda|= d}q^{ \mathbf{f}_{\lambda}(g,\deg L_1, \deg L_2)}\cdot \left(\prod_{d\geq 1}\frac{1}{(1-q^d)^d} \prod_{\Box\in \lambda} \frac{1}{1-q^{h(\Box)}}\right)^{2-2g    },
 \end{align*}
 where the last line follows by specialising the second motivic identity in \Cref{motives of spp} to the topological Euler characteristic.
 \end{proof}
 \subsection{K-theory class of  the perfect obstruction theory}\label{sec: class pot}
By Graber-Pandharipande localisation \cite{GP_virtual_localization}, the perfect obstruction theory  \eqref{eqn: obstruction theory} induces a perfect obstruction theory on each skew nested Hilbert scheme  $C^{[\bn]}$, with notation as in \Cref{sec: fixed locus}. We compute in this section the induced virtual fundamental class   $[C^{[\bn]}]^{\vir}$ and show that it coincides with the ordinary fundamental class.\\

We start by describing the class in $\TT$-equivariant $K$-theory  of the restriction of the dual of the  perfect obstruction theory on $\Hilb^n(X, \beta) $ to  $C^{[\bn]}$. Consider the commutative diagram
\begin{equation}\label{eqn: diagram comm pot}
    \begin{tikzcd}
	{C\times \Hilb^n(X, \beta)} && {C\times C^{[\bn]}} \\
	{X\times \Hilb^n(X, \beta)} && {X\times C^{[\bn]} } \\
	\Hilb^n(X, \beta) && {C^{[\bn]}}
	\arrow["{\pi''}"', from=2-1, to=1-1]
	\arrow["{j''}", hook', from=1-3, to=1-1]
	\arrow[ hook, from=1-3, to=2-3]
	\arrow["{\pi}", bend left=60, from=1-3, to=3-3]
	\arrow["\pi", from=2-1, to=3-1]
	\arrow["{p'}", bend left=30, from=2-3, to=1-3]
	\arrow[hook', from=2-3, to=2-1]
	\arrow["{\pi'}"', from=2-3, to=3-3]
	\arrow["j"', hook', from=3-3, to=3-1]
\end{tikzcd}
\end{equation}
where by  convenience we still denote by $\pi$ the natural projection $C\times C^{[\bn]}\to C^{[\bn]}$. Recall the universal exact sequences 
 \begin{align}\label{eqn: short exact univ}
 \begin{split}
      0\to \CI\to &\oO\to \oO_{\CZ} \to 0,\\
      0\to \oO(-\CZ_{ij})\to &\oO\to \oO_{\CZ_{ij}} \to 0, \quad (i,j)\in \BZ^2_{\geq 0}\setminus \lambda,
 \end{split}
 \end{align}
 respectively on $ \Hilb^n(X, \beta)\times X$ and $C^{[\bn]} \times C$. To simplify notation, we will often omit writing the pullback map $\pi^*$
  when the context is clear. Set the $K$-theory classes on $C\times C^{[\bn]}$
\begin{align*}
    \CA&=\sum_{(i,j)\in \lambda} L_1^{-i}L_2^{-j}\cdot t_1^{-i}t_2^{-j},\\
    \CB&= \sum_{(i,j)\in \BZ^2_{\geq 0}\setminus\lambda}\oO_{\CZ_{ij}}\otimes L_1^{-i}L_2^{-j}\cdot t_1^{-i}t_2^{-j},\\
    \CN&=\oO-L_1 \cdot t_1 -L_2\cdot t_2+L_1L_2\cdot t_1t_2.
\end{align*}
By repeating the proof of \Cref{thm: fixed} over the universal family, we have an identity of $K$-theory classes
\begin{align}\label{eqn: push O z}
    p'_*\oO_\CZ= \CA+\CB \in K^0_\TT(C^{[\bn]}).
\end{align}
 \begin{prop}\label{prop: K theort class}
Let $C$ be a smooth projective curve of genus $g$ and $L_1,L_2$ line bundles on $C$. Set $k_i=\deg L_i$ for $i=1, 2$.     We have an identity in $K_\TT^0(C^{[\bn]})$
\begin{multline*}
      \BE|^\vee_{C^{[\bn]}}=  \sum_{(i,j)\in \BZ^2_{\geq 0}\setminus\lambda}\left(\pi_*\left(\oO_{\CZ_{ij}}\otimes L_1^{-i}L_2^{-j}\right)\cdot t_1^{-i}t_2^{-j} - \left(\pi_*\left(\oO_{\CZ_{ij}}\otimes L_1^{-i-1}L_2^{-j-1}\otimes \omega_C\right)\right)^\vee\cdot t_1^{i+1}t_2^{j+1}\right.\\
 \left.-\pi_*\left(\oO_{\CZ_{ij}}\otimes L_1^{-i}L_2^{-j}\otimes \CN\otimes \CA^* \right)\cdot t_1^{-i}t_2^{-j}+\left(\pi_*\left(\oO_{\CZ_{ij}}\otimes L_1^{-i}L_2^{-j}\otimes \CN^*\otimes \CA^* \otimes \omega_C\right)\right)^\vee\cdot t_1^{i}t_2^{j}\right)\\
      -  \sum_{(i,j), (l,k)\in\BZ^2_{\geq 0}\setminus\lambda } \RR\pi_*\RR\hom( \oO_{\CZ_{ij}}\otimes \CN^*,\oO_{\CZ_{lk}}\otimes L_1^{i-l}L_2^{j-k})\cdot t_1^{i-l}t_2^{j-k}\\
       - \sum_{(i,j), (l,k)\in \lambda} \left(\oO^{1-g+(i-l)k_1+(j-k)k_2}-\oO^{1-g+(i-l+1)k_1+(j-k)k_2}t_1-\oO^{1-g+(i-l)k_1+(j-k+1)k_2}t_2\right. \\
       \left.+\oO^{1-g+(i-l+1)k_1+(j-k+1)k_2}t_1t_2\right)\cdot t_1^{i-l}t_2^{j-k} \\
        +\sum_{(i,j)\in \lambda}\left(\oO_{C^{[\bn]}}^{1-g -i\cdot k_1-j\cdot k_2}\cdot t_1^{-i}t_2^{-j}- \oO_{C^{[\bn]}}^{g-1-(i+1)k_1-(j+1) k_2}\cdot t_1^{i+1}t_2^{j+1}\right)
\end{multline*}
 \end{prop}
 \begin{proof}
By \eqref{eqn: diagram comm pot}, \eqref{eqn: short exact univ} we have an identity
\begin{align*}
      \BE|^\vee_{C^{[\bn]}}&=\RR\pi'_*\RR\hom( \oO,\oO_{\CZ})+\RR\pi'_*\RR\hom(\oO_\CZ, \oO)-\RR\pi'_*\RR\hom( \oO_\CZ,\oO_\CZ)\in K_\TT^0(C^{[\bn]}).
\end{align*}
We now prove  a series of identities  for the summands of $  \BE|^\vee_{C^{[\bn]}}$. Clearly, we have
\begin{align*}
    \RR\pi'_*\RR\hom( \oO,\oO_{\CZ})&=  \RR\pi_*p'_*\oO_{\CZ}.
\end{align*}
Set the projective completion of $X$ to be $\overline{X}=\BP_C(L_1\oplus L_2\oplus \oO_C)$ and consider the commutative diagram
\begin{equation}\label{eqn: diagram proj compl}
    \begin{tikzcd}
	{X\times C^{[\bn]}} & {\overline{X}\times C^{[\bn]}} & {C\times C^{[\bn]}} \\
	X & {\overline{X}} & C
	\arrow["\iota'", hook, from=1-1, to=1-2]
	\arrow["p'", bend left=30, two heads, from=1-1, to=1-3]
	\arrow["\pi'", from=1-1, to=2-1]
	\arrow["\overline{p}' ", two heads, from=1-2, to=1-3]
	\arrow[from=1-2, to=2-2]
	\arrow["\pi", from=1-3, to=2-3]
	\arrow["\iota", hook, from=2-1, to=2-2]
	\arrow["p", bend right =30, two heads, from=2-1, to=2-3]
	\arrow["\overline{p}",  two heads, from=2-2, to=2-3]
\end{tikzcd}
\end{equation}
By applying adjunction and $\TT$-equivariant Grothendieck duality with respect to \eqref{eqn: diagram proj compl}, we have
\begin{align*}
    \RR\pi'_*\RR\hom(\oO_\CZ, \oO)&=\RR\pi_*\RR\overline{p}'_*\RR\iota'_*\RR\hom(\oO_\CZ, \oO)\\
    &=\RR\pi_*\RR\overline{p}'_*\RR\hom(\RR\iota'_*\oO_{\CZ}, \RR\iota'_*\oO)\\
    &=\RR\pi_*\RR\overline{p}'_*\RR\hom(\RR\iota'_*\oO_{\CZ}, \oO)\\
    &=\RR\pi_*\RR\overline{p}'_*\RR\hom(\RR\iota'_*\oO_{\CZ},\overline{p}'^* (L_1L_2)\otimes  \omega_{\overline{p}'}[2])[-2]\otimes t_1t_2\\
    &=\RR\pi_*\RR\hom( p'_* \oO_{\CZ}, L_1L_2)[-2]\otimes t_1t_2\\
    &=\RR\pi_*\RR\hom( p'_* \oO_{\CZ}, L_1L_2\otimes \omega_C^{-1}\otimes \omega_\pi[1])[-3]\otimes t_1t_2\\
    &=\left(\RR\pi_* \left(p'_* \oO_\CZ\otimes \omega_C\otimes L_1^{-1}L_2^{-1}\right)\right)^\vee[-3]\otimes t_1t_2.
\end{align*}
Let $i:C\to X$  and  $i':C\times C^{[\mathbf{n}]}\to X\times C^{[\mathbf{n}]}$ denote the zero sections.  By \cite[Lemma 5.4.9]{CG_representation_theory}, for every $\TT$-equivariant coherent sheaf $\CF\in K^{\TT}_0(C)$, we have
 \begin{align*}
   \mathbf{L}i^*i_*\CF&=\sum_{i=0}^2 (-1)^i\Lambda^i N_{C/X}^*\otimes \CF\\
  &= \CN^*\otimes \CF\in K_{\TT}^0(C),  
 \end{align*}
where $ N_{C/X}=L_1\otimes t_1\oplus L_2\otimes t_2$ is the $\TT$-equivariant normal bundle; an analogous formula holds for $ \mathbf{L}i'^*i'_*\CF$.  Functoriality with respect to the diagram \eqref{eqn: diagram comm pot} yields
\begin{align}\label{eqn: ZZ}
\begin{split}
      \RR\pi'_*\RR\hom( \oO_\CZ,\oO_\CZ)&=\RR\pi_*p'_*\RR\hom( i'_*p'_*\oO_\CZ,i'_*p'_*\oO_\CZ)\\
    &=\RR\pi_*p'_*i'_*\RR\hom( \LL i'^* i'_*p'_*\oO_\CZ,p'_*\oO_\CZ)\\
    &= \RR\pi_*\RR\hom( \CN^*\otimes p'_*\oO_\CZ,p'_*\oO_\CZ).
\end{split}
\end{align}
Notice that, for all line bundles $L$ on $C$, by proper base change we obtain
\begin{align}\label{eqn: PUSH of trivial}
    \RR\pi_*L=\oO_{C^{[\bn]}}^{\chi(L)}\in K^0(C^{[\bn]}).
\end{align}
By \eqref{eqn: push O z}, we have
\begin{align*}
  \RR\pi_*  p'_*\oO_\CZ =\bigoplus_{(i,j)\in \lambda}\oO_{C^{[\bn]}}^{1-g -i\cdot\deg L_1-j\cdot\deg L_2}\cdot t_1^{-i}t_2^{-j}+
     \bigoplus_{(i,j)\in \BZ^2_{\geq 0}\setminus\lambda}\RR\pi_*\left(\oO_{\CZ_{ij}}\otimes L_1^{-i}L_2^{-j}\right)\cdot t_1^{-i}t_2^{-j},
     \end{align*}
     \begin{multline*}
         \left(\RR\pi_* \left(p'_* \oO_\CZ\otimes \omega_C\otimes L_1^{-1}L_2^{-1}\right)\right)^\vee=     \bigoplus_{(i,j)\in \lambda}\oO_{C^{[\bn]}}^{g-1-(i+1)\deg L_1-(j+1)\cdot\deg L_2}\cdot t_1^{i}t_2^{j}+ \\ 
         \bigoplus_{(i,j)\in \BZ^2_{\geq 0}\setminus\lambda}\left(\RR\pi_*\left(\oO_{\CZ_{ij}}\otimes L_1^{-i-1}L_2^{-j-1}\otimes \omega_C\right)\right)^\vee\cdot t_1^{i}t_2^{j}.
     \end{multline*}
By applying Grothendieck duality and \eqref{eqn: ZZ}, we have 
\begin{multline*}
      \RR\pi'_*\RR\hom( \oO_\CZ,\oO_\CZ)= \RR\pi_*\RR\hom( \CA\otimes \CN^*,\CA)+ \RR\pi_*\RR\hom( \CA\otimes \CN^*,\CB)\\
      + \RR\pi_*\RR\hom( \CB\otimes \CN^*,\CA)+ \RR\pi_*\RR\hom( \CB\otimes \CN^*,\CB)\\
      = \sum_{(i,j)\in \BZ^2_{\geq 0}\setminus\lambda}\RR\pi_*\left(\oO_{\CZ_{ij}}\otimes L_1^{-i}L_2^{-j}\otimes \CN\otimes \CA^* \right)\cdot t_1^{-i}t_2^{-j}\\
      -\sum_{(i,j)\in \BZ^2_{\geq 0}\setminus\lambda}\left(\RR\pi_*\left(\oO_{\CZ_{ij}}\otimes L_1^{-i}L_2^{-j}\otimes \CN^*\otimes \CA^* \otimes \omega_C\right)\right)^\vee\cdot t_1^{i}t_2^{j}\\
      +\sum_{(i,j), (l,k)\in\BZ^2_{\geq 0}\setminus\lambda } \RR\pi_*\RR\hom( \oO_{\CZ_{ij}}\otimes \CN^*,\oO_{\CZ_{lk}}\otimes L_1^{i-l}L_2^{j-k})\cdot t_1^{i-l}t_2^{j-k}+\RR \pi_*\left(\CA\otimes \CA^*\otimes \CN \right).
\end{multline*}
More explictly, we express
\begin{multline*}
    \RR \pi_*\left(\CA\otimes \CA^*\otimes \CN \right)=\\
    \sum_{(i,j), (l,k)\in \lambda} \RR\pi_*\left(L_1^{i-l}L_2^{j-k}-L_1^{i-l+1}L_2^{j-k}t_1-L_1^{i-l}L_2^{j-k+1}t_2+L_1^{i-l+1}L_2^{j-k+1}t_1t_2 \right)\cdot t_1^{i-l}t_2^{j-k},
\end{multline*}
which by \eqref{eqn: PUSH of trivial} can be written as a formal sum of trivial line bundles on $C^{[\bn]}$. 

Finally, notice that for any line bundle $L$ on $C$ and any $(i,j)\in \BZ^2_{\geq 0}\setminus \lambda$, by cohomology and base change we have
\begin{align*}
    \RR\pi_*(\oO_{\CZ_{ij}}\otimes L)=\pi_*(\oO_{\CZ_{ij}}\otimes L),
\end{align*}
which is an genuine vector bundle of rank $n_{ij}$.
 \end{proof} 
By Graber-Pandharipande localisation \cite{GP_virtual_localization}, each skew nested Hilbert scheme $C^{[\bn]}$ is endowed with a virtual fundamental class $   [C^{[\mathbf{n}]}]^{\vir}$ induced by the perfect obstruction theory
\[
\BE|^{\fix}_{C^{[\bn]}}\to \BL_{C^{[\bn]}}.
\]
 \begin{theorem}\label{thm: equality virtual classes}
 There are  identities of cycles
 \begin{align*}
     [C^{[\mathbf{n}]}]^{\vir}&=[C^{[\mathbf{n}]}]\in A_{\omega(\bn)}(C^{[\mathbf{n}]}),\\
       \oO_{C^{[\mathbf{n}]}}^{\vir}&=  \oO_{C^{[\mathbf{n}]}}\in K_0(C^{[\mathbf{n}]}).
 \end{align*}
 \end{theorem}
 \begin{proof}
   By \Cref{prop: dim of skew nested} the skew nested Hilbert scheme $C^{[\bn]}$ is a reduced local complete intersection  of pure dimension $\omega(\bn)$, which implies that $\BL_{C^{[\bn]}}$ is perfect of amplitude $[-1,0]$ and of virtual  rank $\omega(\bn)$. It follows that if the complex $ \BE|^{\fix}_{C^{[\bn]}}$ has virtual rank $\omega(\bn)$, then the  map in  the perfect obstruction theory is an isomorphism, which implies the required identities, see e.g.~\cite[Lemma 3.3]{Sch_double}. We are therefore reduced to compute the virtual rank of the $\TT$-fixed part $\BE|^{\fix}_{C^{[\bn]}}$.

Notice first that, for all $(i,j), (l,k)\in \BZ^2_{\geq 0}\setminus \lambda$ and all line bundles $L$ on $C$, By Riemann-Roch we have
\begin{align}\label{eqn: vanish rank}
    \rk \RR\pi_*\RR\hom( \oO_{\CZ_{ij}}\otimes L,\oO_{\CZ_{lk}})=0.
\end{align}
   Assume first that $\lambda=\varnothing$. In this case, $\CA=0$ and  by \Cref{prop: K theort class} and \eqref{eqn: vanish rank} we compute
   \begin{align*}
       \BE|^{\vee, \fix}_{C^{[\bn]}}=\pi_*\oO_{\CZ_{ij}},
   \end{align*}
   which is a vector bundle of rank $n_{00}=\omega(\bn)$.
     Assume now that $\lambda \neq \varnothing$. Analysing the  $\TT$-weights contributing to the fixed part of $    \BE|^{\vee}_{C^{[\bn]}}$ yields
     \begin{multline*}
          \sum_{(i,j)\in \BZ^2_{\geq 0}\setminus\lambda}\left(\pi_*\left(\oO_{\CZ_{ij}}\otimes L_1^{-i}L_2^{-j}\right)\cdot t_1^{-i}t_2^{-j} - \left(\pi_*\left(\oO_{\CZ_{ij}}\otimes L_1^{-i-1}L_2^{-j-1}\otimes \omega_C\right)\right)^\vee\cdot t_1^{i+1}t_2^{j+1}\right.\\
 \left.-\pi_*\left(\oO_{\CZ_{ij}}\otimes L_1^{-i}L_2^{-j}\otimes \CN\otimes \CA^* \right)\cdot t_1^{-i}t_2^{-j}+\left(\pi_*\left(\oO_{\CZ_{ij}}\otimes L_1^{-i}L_2^{-j}\otimes \CN^*\otimes \CA^* \otimes \omega_C\right)\right)^\vee\cdot t_1^{i}t_2^{j}\right)^{\fix}=\\
  -\sum_{(i,j)\in \BZ^2_{\geq 0}\setminus\lambda} \sum_{(l,k)\in\lambda} \left( \pi_*\left(\oO_{\CZ_{ij}}\otimes L_1^{l-i}L_2^{k-j}\otimes (\oO-L_1 \cdot t_1 -L_2\cdot t_2+L_1L_2\cdot t_1t_2) \right)\cdot t_1^{l-i}t_2^{k-j}\right)^{\fix}\\
  =\sum_{(i,j)\in \BZ^2_{\geq 0}\setminus\lambda}\left(\sum_{\substack{(l,k)\in\lambda\\ (i-l,k-j)\in \set{(1,0), (0,1)}}}  \pi_*\oO_{\CZ_{ij}} -\sum_{\substack{(l,k)\in\lambda\\ (i-l,k-j)\in \set{(1,1)}}}  \pi_*\oO_{\CZ_{ij}} \right)\\
 = \sum_{(i,j)\in \CoSoc(\lambda)} \pi_*\oO_{\CZ_{ij}} -\sum_{(i,j)\in \CoSubsoc(\lambda)} \pi_*\oO_{\CZ_{ij}},
     \end{multline*}
where the last line follows by a box-by-box analysis of the contribution to the above sum. By taking ranks and exploiting that
\[
\rk \pi_*\oO_{\CZ_{ij}}=n_{ij}, 
\]
we conclude that the rank of the fixed part above is $\omega(\bn)$.

Secondly, consider the fixed part 
\begin{multline*}
   \left(   - \sum_{(i,j), (l,k)\in \lambda} \left(\oO^{1-g+(i-l)k_1+(j-k)k_2}-\oO^{1-g+(i-l+1)k_1+(j-k)k_2}t_1-\oO^{1-g+(i-l)k_1+(j-k+1)k_2}t_2\right. \right.\\
       \left.+\oO^{1-g+(i-l+1)k_1+(j-k+1)k_2}t_1t_2\right)\cdot t_1^{i-l}t_2^{j-k} \\
       \left. +\sum_{(i,j)\in \lambda}\left(\oO_{C^{[\bn]}}^{1-g -i\cdot k_1-j\cdot k_2}\cdot t_1^{-i}t_2^{-j}- \oO_{C^{[\bn]}}^{g-1-(i+1)k_1-(j+1) k_2}\cdot t_1^{i+1}t_2^{j+1}\right)\right)^{\fix}=\\
        \oO_{C^{[\bn]}}^{1-g} - \sum_{\substack{(i,j), (l,k)\in \lambda\\ (i,j)=(l,k)}}\oO_{C^{[\bn]}}^{1-g}+ \sum_{\substack{(i,j), (l,k)\in \lambda\\ (i+1,j)=(l,k)}}\oO_{C^{[\bn]}}^{1-g}+ \sum_{\substack{(i,j), (l,k)\in \lambda\\ (i,j+1)=(l,k)}}\oO_{C^{[\bn]}}^{1-g}- \sum_{\substack{(i,j), (l,k)\in \lambda\\ (i+1,j+1)=(l,k)}}\oO_{C^{[\bn]}}^{1-g}.
\end{multline*}
Now, the family indices of the last four sums are respectively in correspondance with the \emph{vertices},  \emph{vertical edges}, \emph{horizontal edges} and \emph{squares} of the Young diagram $\lambda$, with notation as in \cite[Sec.~2.1]{Mon_double_nested}. Then, by \cite[Lemma 2.1]{Mon_double_nested} we have
\begin{align*}
        1 - \sum_{\substack{(i,j), (l,k)\in \lambda\\ (i,j)=(l,k)}}1+ \sum_{\substack{(i,j), (l,k)\in \lambda\\ (i+1,j)=(l,k)}}1+\sum_{\substack{(i,j), (l,k)\in \lambda\\ (i,j+1)=(l,k)}}1- \sum_{\substack{(i,j), (l,k)\in \lambda\\ (i+1,j+1)=(l,k)}}1=0,
\end{align*}
by which we conclude the proof.
\end{proof}
\section{Universality}\label{sec: universal}
\subsection{Universal expression}
Let $C$ be a smooth projective curve, $L_1, L_2$ line bundles on $C$ and set $X=\Tot_C(L_1\oplus L_2)$. To ease notation, given a class $\CE\in K^0_{\TT}(C^{[\mathbf{n}]})$, we set
\begin{align}\label{eqn: chi hat}
\widehat{\chi}\left(C^{[\mathbf{n}]},\CE\right) =\chi\left(C^{[\mathbf{n}]},\widehat{\oO}_{C^{[\mathbf{n}]}}\otimes \CE\right).
\end{align}
By the results of \Cref{sec: class pot},  the partition functions $   \DT_d(X, q)$ and $  \widehat{\DT}_d(X, q)$ can be expressed as
\begin{align}\label{eqn: DT as localised}
\begin{split}
        \DT_d(X, q)&=\sum_{|\lambda|= d}  q^{\mathbf{f}_{\lambda}( g,\deg L_1, \deg L_2)}\sum_{\mathbf{n}}q^{|\mathbf{n}|}\cdot   \int_{C^{[\mathbf{n}]}}e(-N_{C,L_1,L_2, \bn}^{\vir}),\\
     \widehat{\DT}_d(X, q)&=\sum_{|\lambda|= d}  q^{\mathbf{f}_{\lambda}( g,\deg L_1, \deg L_2)}\sum_{\mathbf{n}}q^{|\mathbf{n}|}\cdot  \widehat{\chi}\left(C^{[\mathbf{n}]},\widehat{\mathfrak{e}}(-N_{C,L_1,L_2, \bn}^{\vir})  \right),
     \end{split}
\end{align}
where $N^{\vir}_{C,L_1,L_2, \bn}$ denotes the virtual normal bundle of $C^{[\bn]}\hookrightarrow\Hilb^n(X, \beta)$ and $\bn$ is a skew plane partition of shape $\BZ^2_{\geq 0}\setminus \lambda$.
\smallbreak
Throughout this section, we fix a Young diagram  $\lambda$. We prove in this section that the generating series
\begin{align}\label{eqn: gen integral on double nested}
\begin{split}
       \sum_{\mathbf{n}}q^{|\mathbf{n}|}\cdot \int_{C^{[\mathbf{n}]}}e(-N_{C,L_1,L_2, \bn}^{\vir})&\in \BQ(s_1,s_2)\llbracket q \rrbracket,\\
   \sum_{\mathbf{n}}q^{|\mathbf{n}|}\cdot \widehat{\chi}\left(C^{[\mathbf{n}]},\widehat{\mathfrak{e}}(-N_{C,L_1,L_2, \bn}^{\vir})  \right)&\in \BQ(t_1^{1/2}, t_2^{1/2})\llbracket q \rrbracket
    \end{split}
\end{align}
are controlled by their constant term and  three universal series, depending only on the Young diagram $\lambda$. 

\subsubsection{Multiplicativity} We begin by showing that the generating series \eqref{eqn: gen integral on double nested} are multiplicative with respect to disjoint unions of triples $(C, L_1, L_2)$.
\begin{lemma}\label{lemma: multiplicativity}
Let $(C,L_1,L_2)$ be a triple where $C=C'\sqcup C''$ and  $L_i=L'_i\oplus L''_i$ for $i=1,2$, where $L'_i$ are line bundles on $C'$ and $L''_i$ are line bundles on $C''$. Then 
\begin{multline*}
     \sum_{\mathbf{n}}q^{|\mathbf{n}|} \int_{C^{[\mathbf{n}]}}e(-N_{C,L_1,L_2, \bn}^{\vir})= \sum_{\mathbf{n}}q^{|\mathbf{n}|} \int_{C'^{[\mathbf{n}]}}e(-N_{C',L'_1,L'_2, \bn}^{\vir})\cdot  \sum_{\mathbf{n}}q^{|\mathbf{n}|} \int_{C''^{[\mathbf{n}]}}e(-N_{C'',L''_1,L''_2, \bn}^{\vir}),\\
        \sum_{\mathbf{n}}q^{|\mathbf{n}|}   \widehat{\chi}\left(C^{[\mathbf{n}]},\widehat{\mathfrak{e}}(-N_{C,L_1,L_2, \bn}^{\vir})  \right)=   \sum_{\mathbf{n}}q^{|\mathbf{n}|}   \widehat{\chi}\left(C'^{[\mathbf{n}]},\widehat{\mathfrak{e}}(-N_{C',L'_1,L'_2, \bn}^{\vir})  \right)\cdot   \sum_{\mathbf{n}}q^{|\mathbf{n}|}   \widehat{\chi}\left(C''^{[\mathbf{n}]},\widehat{\mathfrak{e}}(-N_{C'',L''_1,L''_2, \bn}^{\vir})  \right).
\end{multline*}
where the sums are over all skew plane partitions  of shape  $\BZ^2_{\geq 0}\setminus \lambda$. 
\end{lemma}
\begin{proof}
    The proof is analogous to the one of \cite[Prop.~5.2]{Mon_double_nested}. In fact, it follows directly by the decomposition of $C^{[\bn]}$ into connected components as
    \begin{align*}
C^{[\mathbf{n}]}=\coprod_{\mathbf{n}'+\mathbf{n}''=\mathbf{n}} C'^{[\mathbf{n}']}\times C''^{[\mathbf{n}'']},
\end{align*}
the decomposition of the virtual normal bundle as
    \begin{align*}
  N^{\vir}_{C,L_1,L_2,\bn}|_{C'^{[\mathbf{n}']}\times C''^{[\mathbf{n}'']}}= N^{\vir}_{C',L'_1,L'_2, \bn'}\boxplus N^{\vir}_{C'',L''_1,L''_2, \bn''}
\end{align*}
and the multiplicativity of the operators $ e, \widehat{\mathfrak{e}}$.
\end{proof}
\subsubsection{Chern number dependence} We show that the dependence of the  invariants  \eqref{eqn: gen integral on double nested} on $(C, L_1, L_2)$ is only on the Chern numbers $(g(C), \deg L_1, \deg L_2)$.
\begin{prop}\label{prop: chern dependence integrals}
Let $C$ be a genus $g$  irreducible smooth projective curve and $L_1, L_2$ line bundles on $C$. Then the invariants \eqref{eqn: gen integral on double nested} depend on $(C, L_1, L_2)$ only  on $(g, \deg L_1, \deg L_2$).
\end{prop}
\begin{proof}
As in the proof of \Cref{prop: dim of skew nested}, there exists a Young diagram $\Tilde{\lambda}$ and a reverse plane partition $\Bm$ of shape $\Tilde{\lambda}$ such that
\begin{align*}
    C^{[\bn]}\cong C^{[\Bm]}.
\end{align*}
Moreover, under this identification, the universal subschemes $\CZ_{\Box}\subset C\times   C^{[\bn]}$ are naturally identified with universal subschemes $\CZ_{\Tilde{\Box}}\subset  C\times   C^{[\Bm]}$, where the box $\tilde{\Box}\in \tilde{\lambda}$ corresponds to $\Box\in \lambda$ as described in the proof of \Cref{prop: dim of skew nested}. 

Let $C^{(m)}$ denote the $m$-th symmetric power of $C$ and  define the smooth  variety
\begin{align*}
    A_{C,\Bm}&=C^{(m_{0,0})}\times \prod_{\substack{(i,j)\in \tilde{\lambda}\\ i\geq 1}} C^{(m_{i,j}-m_{i-1,j})}\times \prod_{\substack{(l,k)\in \tilde{\lambda}\\ k\geq 1}} C^{(m_{l,k}-m_{l,k-1})}\\
    &=\prod_{a}C^{[a]}.
\end{align*}
By \cite[Thm.~4.1, 4.2]{GLMRS_double-nested-1}, there is a regular embedding 
\[i:C^{[\Bm]}\cong Z(s)\hookrightarrow  A_{C,\Bm},\]
which records the subscheme in position $(0,0)$ and all the possible differences of the nested subschemes. The regular embedding $i$ realises $ C^{[\Bm]}$ as the zero locus of a section $s$ of a vector bundle $\CE$ on $A_{C,\Bm}$, cf.~\cite[Thm.~4.1]{GLMRS_double-nested-1}, which implies that 
\begin{align}\label{eqn_ reg emb}
i_*[C^{[\Bm]}]=e(\CE)\cap [A_{C,\Bm}]\in H_*(A_{C,\Bm}).    
\end{align}
We claim that the virtual normal bundle satisfies
\begin{align}\label{eqn_ virt restr}
N^{\vir}_{C, L_1, L_2,\bn}=i^* \Tilde{N}^{\vir}_{C, L_1, L_2}, 
\end{align}
for a certain class $\Tilde{N}^{\vir}_{C, L_1, L_2}\in K^0_\TT(A_{C, \mathbf{m}})$.
To prove the claim, notice that by \Cref{prop: K theort class} we have that $ N^{\vir}_{C, L_1, L_2,\bn}$ is a linear combination of classes in $K$-theory of the form 
\begin{align*}
    \oO^{A(g, \deg L_1, \deg L_2)}_{ C^{[\bn]}}\in K^0_\TT(C^{[\mathbf{n}]}),
\end{align*}
where $ A(g, \deg L_1, \deg L_2)$ is a function of $ g, \deg L_1, \deg L_2$ and 
\begin{align*}
    \RR\pi_*\RR\hom(\oO_{\CZ_{ij}}\otimes L_1^cL_2^b, \oO_{\CZ_{lk}})\otimes t^\mu\in K^0_\TT(C^{[\mathbf{n}]}),
\end{align*}
for some weight $\mu$, some indices $(i,j), (l,k)$ and  some line bundle $ L_1^cL_2^b$. By the universal sequence
\[
   0\to \oO(-\CZ_{ij})\to \oO\to \oO_{\CZ_{ij}} \to 0, \quad (i,j)\in \BZ^2_{\geq 0}\setminus \lambda, 
\]
the latter can be expressed as a combination of the classes
\begin{align*}
  &  \RR\pi_*\RR\hom(\oO({-\CZ_{ij}})\otimes L_1^cL_2^b, \oO(-{\CZ_{lk}})),\\
   &   \RR\pi_*\RR\hom( L_1^cL_2^b, \oO(-{\CZ_{lk}})),\\
    &    \RR\pi_*\RR\hom(\oO({-\CZ_{ij}})\otimes L_1^cL_2^b, \oO).
\end{align*}
Since each universal subscheme $\CZ_{ij}$ is the restriction of a linear combination of universal subschemes on $ A_{C,\Bm} $, this proves the claim. 

Combining \eqref{eqn_ reg emb}, \eqref{eqn_ virt restr} we have
\begin{align*}
    \int_{C^{[\mathbf{n}]}}e(-N_{C,L_1,L_2, \bn}^{\vir})=\int_{A_{C, \mathbf{n}}}e(\CE-\Tilde{N}_{C,L_1,L_2}^{\vir}).
\end{align*}
Denote by
\[ 0\to \CI_a\to \oO\to \oO_{\CZ_a}\to 0\]
the universal sequences on $ A_{C, \mathbf{n}}$. It follows that the $K$-theory class of $\CE-\Tilde{N}_{C,L_1,L_2}^{\vir}$ is a linear combination of classes of the form 
 \begin{align*}
     \RR\pi_*\RR\hom\left(\bigotimes_{i\in I}\CI_i,\bigotimes_{j\in J}\CI_j\otimes L_k\right)\otimes t^{\mu},
 \end{align*}
  where $L_k$ are line bundles on $C$ and $I,J$ are families of indices (possibly with repetitions). We conclude the proof by \cite[Prop.~5.3]{Mon_double_nested} and noticing that all line bundles $L_k$ possibly occuring are a linear combination of $L_1, L_2, \omega_C$.

  The case of the $K$-theoretic invariants follows by applying virtual Grothendieck-Riemann-Roch \cite{FG_riemann_roch} and performing an analogous reasoning, see e.g.~\cite[Prop.~9.2]{Mon_double_nested}.
\end{proof}
\subsubsection{Constant term} We show that the constant terms of the generating series \eqref{eqn: gen integral on double nested} are invertible and are determined by three universal series. 
\smallbreak
Set the following virtual $\TT$-representations
\begin{align}\label{eqn: Z lambda}
\begin{split}
     \mathsf{Z}_\lambda&= \sum_{(i,j)\in \lambda} t_1^{-i}t_2^{-j},\\
    T_\lambda&=\mathsf{Z}_\lambda +\overline{\mathsf{Z}}_\lambda t_1t_2-(1-t_1)(1-t_2)\mathsf{Z}_\lambda\overline{\mathsf{Z}}_\lambda,
\end{split}
\end{align}
where recall that we defined $\overline{\mathsf{Z}}_\lambda$ as the dual $\TT$-representation of $\mathsf{Z}_\lambda$.

Let $\mathbf{0}$ be the trivial skew plane partition of shape $\BZ^2_{\geq 0}\setminus \lambda$ and size 0. We have that $C^{[\mathbf{0}]}\cong \pt$ consists in only one reduced point, and therefore $N^{\vir}_{C, L_1, L_2,\mathbf{0} }$ can be identified with an element of $\BZ[t_1^{\pm 1}, t_2^{\pm 1}]$.
\begin{prop}\label{prop: virtual normal const as univer}
    Let $C$ be a genus $g$  irreducible smooth projective curve and $L_1, L_2$ line bundles on $C$. Then 
    \begin{align*}
        N^{\vir}_{C, L_1, L_2,\mathbf{0} }=(1-g)\cdot T_\lambda+\deg L_1\cdot t_1\frac{\partial}{\partial t_1}T_\lambda+\deg L_2\cdot t_2\frac{\partial}{\partial t_2}T_\lambda.
    \end{align*}
    \begin{proof}
    Set $k_i=\deg L_i$. 
        By \Cref{prop: K theort class} we have that 
          \begin{multline}\label{eqn: DT pref normal}
       \BE|^\vee_{C^{[\mathbf{0}]}}= 
       - \sum_{(i,j), (l,k)\in \lambda} \left(\oO^{1-g+(i-l)k_1+(j-k)k_2}-\oO^{1-g+(i-l+1)k_1+(j-k)k_2}t_1\right. \\
       \left.-\oO^{1-g+(i-l)k_1+(j-k+1)k_2}t_2+\oO^{1-g+(i-l+1)k_1+(j-k+1)k_2}t_1t_2\right)\cdot t_1^{i-l}t_2^{j-k} \\
        +\sum_{(i,j)\in \lambda}\left(\oO_{C^{[\bn]}}^{1-g -i\cdot k_1-j\cdot k_2}\cdot t_1^{-i}t_2^{-j}- \oO_{C^{[\bn]}}^{g-1-(i+1)k_1-(j+1) k_2}\cdot t_1^{i+1}t_2^{j+1}\right).
\end{multline}
Notice that the differential operator $t_1\frac{\partial}{\partial t_1}$ acts on a $\TT$-weight $t_1^at_2^b$ by
\[
t_1\frac{\partial}{\partial t_1} t_1^at_2^b=a\cdot t_1^at_2^b,
\]
and similarly $t_2\frac{\partial}{\partial t_2} $. The claimed identity follows simply by rearranging the terms in \eqref{eqn: DT pref normal}. To conclude the proof, notice that the virtual representation $T_\lambda$ computes the weight decomposition of the $\TT$-equivariant tangent space of $\Hilb^{|\lambda|}(\BA^2)$ at the isolated and reduced $\TT$-fixed point corresponding to the Young diagram $\lambda$, which implies that $T_\lambda$ is $\TT$-movable, see e.g.~\cite[Prop.~3.4.17]{Okounk_Lectures_K_theory}.
    \end{proof}
\end{prop}
Over a point, the operators  $e, \widehat{\mathfrak{e}}$ act in an explicit way on virtual $\TT$-representations, which we now recall. Let $V=\sum_{(\mu_1, \mu_2)}t_1^{\mu_1}t_2^{\mu_2}-\sum_{(\nu_1, \nu_2)}t_1^{\nu_1}t_2^{\nu_2}$ be a virtual $\TT$-representation, where we assume that $(\nu_1, \nu_2)\neq (0,0)$. For a formal variable $x$, denote the \emph{symmetrised} operator
\begin{align}\label{eqn: brack}
[x]=x^{1/2}-x^{-1/2}.
\end{align}
 We have
\begin{align*}
    e(V)&=\frac{\prod_{(\mu_1, \mu_2)}(\mu_1s_1+\mu_2s_2)}{\prod_{(\nu_1, \nu_2)}(\nu_1s_1+\nu_2s_2)},\\
    \widehat{\mathfrak{e}}(V)&=\frac{\prod_{(\mu_1, \mu_2)}[t_1^{\mu_1}t_2^{\mu_2}]}{\prod_{(\nu_1, \nu_2)}[t_1^{\nu_1}t_2^{\nu_2}]},
\end{align*}
see for instance the discussions in  \cite[Sec.~6.1, 7.1]{FMR_higher_rank}.
\smallbreak
Applying the operators $e, \widehat{\mathfrak{e}}$ to \Cref{prop: virtual normal const as univer} we immediately obtain the following corollary.
\begin{corollary}\label{cor: virt norm operators}
     Let $C$ be a genus $g$  irreducible smooth projective curve and $L_1, L_2$ line bundles on $C$. Then
     \begin{align*}
         e(-   N^{\vir}_{C, L_1, L_2,\mathbf{0} })&=e\left(-T_\lambda\right)^{1-g}\cdot e\left(-t_1\frac{\partial}{\partial t_1}T_\lambda\right)^{\deg L_1}\cdot e\left(-t_2\frac{\partial}{\partial t_2}T_\lambda\right)^{\deg L_2}\in \BQ(s_1,s_2),\\
              \widehat{\mathfrak{e}}(-   N^{\vir}_{C, L_1, L_2,\mathbf{0} })&= \widehat{\mathfrak{e}}\left(-T_\lambda\right)^{1-g}\cdot  \widehat{\mathfrak{e}}\left(-t_1\frac{\partial}{\partial t_1}T_\lambda\right)^{\deg L_1}\cdot  \widehat{\mathfrak{e}}\left(-t_2\frac{\partial}{\partial t_2}T_\lambda\right)^{\deg L_2}\in \BQ(t_1^{1/2}, t_2^{1/2}).
\end{align*}
In particular, $ e(-   N^{\vir}_{C, L_1, L_2,\mathbf{0} }),  \widehat{\mathfrak{e}}(-   N^{\vir}_{C, L_1, L_2,\mathbf{0} })$ are invertible.
\end{corollary}
\subsubsection{Universal expression}
We are ready to  prove that the generating series \eqref{eqn: gen integral on double nested} are controlled by three universal series.
\begin{theorem}\label{thm: universal series}
Let $C$ be a genus $g$ smooth irreducible projective curve and $L_1, L_2$  line bundles over $C$. There is an identity
\begin{align*}
    \sum_{\mathbf{n}}q^{|\mathbf{n}|} \int_{C^{[\mathbf{n}]}}e(-N_{C,L_1,L_2,\bn}^{\vir})= A_{\lambda}(q)^{1-g}\cdot B_{\lambda}(q)^{\deg L_1}\cdot C_{\lambda}(q)^{\deg L_2}\in \BQ(s_1,s_2)\llbracket q \rrbracket,
\end{align*}
where the sum is over all skew plane partitions $\bn$ of shape $\BZ^2_{\geq 0}\setminus \lambda$, and  $A_{\lambda}(q),B_{\lambda}(q),C_{\lambda}(q)\in \BQ(s_1,s_2)\llbracket q \rrbracket$ are fixed universal series  which only depend on $ \lambda$. 
Similarly, there is an identity
\begin{align*}
    \sum_{\mathbf{n}}q^{|\mathbf{n}|} \widehat{\chi}\left(C^{[\mathbf{n}]},\widehat{\mathfrak{e}}(-N_{C,L_1,L_2, \bn}^{\vir})  \right)= \widehat{A}_{\lambda}(q)^{1-g}\cdot \widehat{B}_{\lambda}(q)^{\deg L_1}\cdot \widehat{C}_{\lambda}(q)^{\deg L_2}\in  \BQ(t_1^{1/2}, t_2^{1/2})\llbracket q \rrbracket,
\end{align*}
where the sum is over all skew plane partitions $\bn$ of shape $\BZ^2_{\geq 0}\setminus \lambda$, and  $\widehat{A}_{\lambda}(q),\widehat{B}_{\lambda}(q),\widehat{C}_{\lambda}(q)\in  \BQ(t_1^{1/2}, t_2^{1/2})\llbracket q \rrbracket$ are fixed universal series  which only depend on $ \lambda$. 
\end{theorem}
\begin{proof}
The proof is analogous to \cite[Thm.~5.1]{Mon_double_nested}. We prove here the claim only for the first identity, as the second one follows by an analogous discussion. 

    Consider the map
    \[
Z:\CK:=\set{(C,L_1,L_2)| C \mbox{ smooth projective  curve}, L_1, L_2 \mbox{ line bundles}}\to 1+\BQ(s_1,s_2)\llbracket q \rrbracket
\]
given by 
\begin{align*}
    Z(C,L_1, L_2)=   e(-   N^{\vir}_{C, L_1, L_2,\mathbf{0} })^{-1}\cdot\sum_{\mathbf{n}}q^{|\mathbf{n}|} \int_{C^{[\mathbf{n}]}}e(-N_{C,L_1,L_2, \bn}^{\vir}),
\end{align*}
where $   e(-   N^{\vir}_{C, L_1, L_2,\mathbf{0} })$ is invertible by \Cref{cor: virt norm operators}.
By  \Cref{lemma: multiplicativity} the map $Z(\cdot)$ is \emph{multiplicative} and by  \Cref{prop: chern dependence integrals}  $Z(C, L_1,  L_2)$ depend only  on the Chern numbers of  $(C, L_1,  L_2)$. This implies that $Z$ factors through
\[\CK\xrightarrow{\gamma} \BZ^3\xrightarrow{Z'}  1+\BQ(s_1,s_2)\llbracket q \rrbracket,\]
where $\gamma(C,L_1, L_2)=(1-g, \deg L_1, \deg L_2)$, and that $Z'$ is a morphism of monoids from $(\BZ^3,+) $ to $( 1+\BQ(s_1,s_2)\llbracket q \rrbracket, \cdot)$. A set of generators of $ \BZ^3$ is given by the elements
 \begin{align*}
     e_1=\gamma(\BP^1, \oO,\oO),\quad e_2=\gamma(\BP^1, \oO(1),\oO),\quad  e_3=\gamma(\BP^1, \oO,\oO(1)),
 \end{align*}
  and the image of a  generic triple $(C,L_1,L_2) $ can be written as
 \begin{align*}
     \gamma(C,L_1,L_2)=(1-g-\deg L_1-\deg L_2)\cdot e_1+\deg L_1\cdot  e_2+\deg L_2 \cdot e_3.
 \end{align*}
 We conclude that
 \begin{align*}
     Z'(C,L_1,L_2)=Z'(e_1)^{1-g}\cdot (Z'(e_1)^{-1}Z'(e_2))^{\deg L_1}\cdot(Z'(e_1)^{-1}Z'(e_3))^{\deg L_2}.
 \end{align*}
 Combining this decomposition with  \Cref{cor: virt norm operators} completes the proof, and yields the desired universal series.
 \end{proof}

\section{Computations}\label{sec: computations}

Throughout this section, we fix a Young diagram $\lambda$, and compute the universal series of \Cref{thm: universal series}.
\subsection{Constant terms}
We evaluate  the universal series appearing in \Cref{cor: virt norm operators}, which contribute to the constant term of the Donaldson-Thomas partition functions. 
\smallbreak
Recall that we defined in \eqref{eqn: brack} the operator $[x]=x^{1/2}-x^{-1/2}$.
\begin{prop}\label{eqn: costant DT expl}
    Let $\lambda$ be a Young diagram. We have
    \begin{align*}
        &e\left(-T_\lambda\right)=\prod_{\Box\in \lambda}\frac{1}{(-\ell(\Box)s_1+(a(\Box)+1)s_2)((\ell(\Box)+1)s_1-a(\Box)s_2)},\\
          &e\left(  -t_1\frac{\partial}{\partial t_1}T_\lambda\right)=\prod_{\Box\in \lambda}\frac{( -\ell(\Box) s_1+(a(\Box)+1)s_2)^{\ell(\Box)}}{( (\ell(\Box)+1)s_1-a(\Box)s_2)^{\ell(\Box)+1}},\\
          &  e\left(  -t_2\frac{\partial}{\partial t_2}T_\lambda\right)=\prod_{\Box\in \lambda}\frac{(\ell(\Box)+1)s_1-a(\Box)s_2)^{a(\Box)}}{(-\ell(\Box)s_1+(a(\Box)+1)s_2)^{a(\Box)+1}},\\
          &\widehat{\mathfrak{e}}\left(-T_\lambda\right)=\prod_{\Box\in \lambda}\frac{1}{[t_1^{-\ell(\Box)}t_2^{a(\Box)+1}][t_1^{\ell(\Box)+1}t_2^{-a(\Box)}]},\\
          &\widehat{\mathfrak{e}}\left(  -t_1\frac{\partial}{\partial t_1}T_\lambda\right)=\prod_{\Box\in \lambda}\frac{[t_1^{-\ell(\Box)}t_2^{a(\Box)+1}]^{\ell(\Box)}}{[t_1^{\ell(\Box)+1}t_2^{-a(\Box)}]^{\ell(\Box)+1}},\\
          &   \widehat{\mathfrak{e}}\left(  -t_2\frac{\partial}{\partial t_2}T_\lambda\right)=\prod_{\Box\in \lambda}\frac{[t_1^{\ell(\Box)+1}t_2^{-a(\Box)}]^{a(\Box)}}{[t_1^{-\ell(\Box)}t_2^{a(\Box)+1}]^{a(\Box)+1}}.
    \end{align*}
\end{prop}
\begin{proof}
    By \cite[Ex.~3.4.22]{Okounk_Lectures_K_theory} there is an identity of $\TT$-representations
    \begin{align*}
        T_\lambda=\sum_{\Box \in \lambda} (t_1^{-\ell(\Box)}t_2^{a(\Box)+1}+ t_1^{\ell(\Box)+1}t_2^{-a(\Box)}).
    \end{align*}
    The claimed identities follow by applying the operators $e, \widehat{\mathfrak{e}}$.
\end{proof}
We remark that the series $ e\left(-T_\lambda\right)$ reproduces (up to a sign) the so-called \emph{Jack-Plancherel measure} on Young diagrams, see e.g.~\cite{Young_Jack}.
\subsubsection{Anti-diagonal restriction}\label{sec: andtifiagonal edge}
We show now that, under the anti-diagonal restriction of the equivariant parameters, the universal series in \Cref{eqn: costant DT expl} admit  more compact formulas in terms of the hooklengths of $\lambda$.
\begin{corollary}\label{cor: edge antidiagonal}
    Let $\lambda$ be a Young diagram. We have
    \begin{align*}
        &e\left(-T_\lambda\right)|_{s_1+s_2=0}=(-s_1^{-2})^{|\lambda|}\prod_{\Box\in \lambda}\frac{1}{h(\Box)^2},\\
        &\left.e\left(  -t_1\frac{\partial}{\partial t_1}T_\lambda\right)\right|_{s_1+s_2=0}=(-1)^{n(\lambda)}s_1^{-|\lambda|}\prod_{\Box\in \lambda}\frac{1}{h(\Box)},\\
         &\left.e\left(  -t_2\frac{\partial}{\partial t_2}T_\lambda\right)\right|_{s_1+s_2=0}=(-1)^{n(\overline{\lambda})}(-s_1)^{-|\lambda|}\prod_{\Box\in \lambda}\frac{1}{h(\Box)},\\
          &\widehat{\mathfrak{e}}\left(-T_\lambda\right)|_{t_1t_2=1}=(-1)^{|\lambda|}\prod_{\Box\in \lambda}\frac{1}{[t_1^{h(\Box)}]^2},\\
          &\left.\widehat{\mathfrak{e}}\left(  -t_1\frac{\partial}{\partial t_1}T_\lambda\right)\right|_{t_1t_2=1}=(-1)^{n(\lambda)}\prod_{\Box\in \lambda}\frac{1}{[t_1^{h(\Box)}]},\\
          &   \left.\widehat{\mathfrak{e}}\left(  -t_2\frac{\partial}{\partial t_2}T_\lambda\right)\right|_{t_1t_2=1}=(-1)^{n(\overline{\lambda})+|\lambda|}\prod_{\Box\in \lambda}\frac{1}{[t_1^{h(\Box)}]}.
    \end{align*}
\end{corollary}
\begin{proof}
   The proof follows by \Cref{eqn: costant DT expl} and the identities \eqref{eqn: comb ident}.
\end{proof}
\subsection{Normalised series}
By the universal expression of \Cref{thm: universal series}, the partition functions $\DT_d(X, q), \widehat{\DT}_d(X, q)$ are reduced to the computations of the \emph{normalised} series
\begin{align}\label{eqn: red Z gen}
\begin{split}
        Z(C,L_1, L_2)&=  e(-   N^{\vir}_{C, L_1, L_2,\mathbf{0} })^{-1}\cdot\sum_{\mathbf{n}}q^{|\mathbf{n}|} \int_{C^{[\mathbf{n}]}}e(-N_{C,L_1,L_2, \bn}^{\vir}),\\
      Z^K(C,L_1, L_2)&= \widehat{  \mathfrak{e}}(-   N^{\vir}_{C, L_1, L_2,\mathbf{0} })^{-1}\cdot\sum_{\mathbf{n}}q^{|\mathbf{n}|} \widehat{\chi}\left(C^{[\mathbf{n}]},\widehat{\mathfrak{e}}(-N_{C,L_1,L_2, \bn}^{\vir})  \right),
      \end{split}
\end{align}
where $\bn$ is a skew plane partition of shape $\BZ^{2}_{\geq 0}\setminus \lambda$ and  $(C, L_1, L_2)$ is  among
\begin{align*}
    (\BP^1, \oO, \oO), \quad 
     (\BP^1, \oO, \oO(-2)), \quad 
      (\BP^1, \oO(-2), \oO).
\end{align*}
Define the classes in $K$-theory 
\begin{multline}\label{eqn: DT normal red}
    N^{\mathrm{red}}_{C, L_1, L_2, \bn}=  \left(\sum_{(i,j)\in \BZ^2_{\geq 0}\setminus\lambda}\left(\pi_*\left(\oO_{\CZ_{ij}}\otimes L_1^{-i}L_2^{-j}\right)\cdot t_1^{-i}t_2^{-j} - \left(\pi_*\left(\oO_{\CZ_{ij}}\otimes L_1^{-i-1}L_2^{-j-1}\otimes \omega_C\right)\right)^\vee\cdot t_1^{i+1}t_2^{j+1}\right.\right.\\
 \left.-\pi_*\left(\oO_{\CZ_{ij}}\otimes L_1^{-i}L_2^{-j}\otimes \CN\otimes \CA^* \right)\cdot t_1^{-i}t_2^{-j}+\left(\pi_*\left(\oO_{\CZ_{ij}}\otimes L_1^{-i}L_2^{-j}\otimes \CN^*\otimes \CA^* \otimes \omega_C\right)\right)^\vee\cdot t_1^{i}t_2^{j}\right)\\
  \left.    -  \sum_{(i,j), (l,k)\in\BZ^2_{\geq 0}\setminus\lambda } \RR\pi_*\RR\hom( \oO_{\CZ_{ij}}\otimes \CN^*,\oO_{\CZ_{lk}}\otimes L_1^{i-l}L_2^{j-k})\cdot t_1^{i-l}t_2^{j-k},\right)^{\mov},
\end{multline}
and 
\begin{multline}\label{eqn: DT prefact normal}
    P_{C, L_1, L_2}= - \sum_{(i,j), (l,k)\in \lambda} \left(\oO^{1-g+(i-l)k_1+(j-k)k_2}-\oO^{1-g+(i-l+1)k_1+(j-k)k_2}t_1\right. \\
       \left.-\oO^{1-g+(i-l)k_1+(j-k+1)k_2}t_2+\oO^{1-g+(i-l+1)k_1+(j-k+1)k_2}t_1t_2\right)\cdot t_1^{i-l}t_2^{j-k} \\
        +\sum_{(i,j)\in \lambda}\left(\oO^{1-g -i\cdot k_1-j\cdot k_2}\cdot t_1^{-i}t_2^{-j}- \oO^{g-1-(i+1)k_1-(j+1) k_2}\cdot t_1^{i+1}t_2^{j+1}\right),
\end{multline}
and notice that 
\begin{align*}
       N_{C,L_1,L_2, \bn}^{\vir}=  P_{C, L_1, L_2}+ N_{C,L_1,L_2, \bn}^{\mathrm{red}}.
\end{align*}
\begin{lemma}\label{eqn: lemma reduced}
Let $C$ be a smooth irreducible projective curve and $L_1, L_2$  line bundles over $C$. There are identities   
\begin{align*}
 Z(C,L_1, L_2)&= \sum_{\mathbf{n}}q^{|\mathbf{n}|} \int_{C^{[\mathbf{n}]}}e(-N_{C,L_1,L_2, \bn}^{\mathrm{red}}),\\
     Z^K(C,L_1, L_2)&=\sum_{\mathbf{n}}q^{|\mathbf{n}|} \widehat{\chi}\left(C^{[\mathbf{n}]},\widehat{\mathfrak{e}}(-N_{C,L_1,L_2, \bn}^{\mathrm{red}})  \right).
\end{align*}
\end{lemma}
\begin{proof}
  We prove the first claim, as the second follows analogously.
      By  \Cref{prop: K theort class} and the multiplicativity of $e$, we have that
    \begin{align*}
        e(-N_{C,L_1,L_2, \bn}^{\vir})&=e(-  P_{C, L_1, L_2})\cdot e(-N_{C,L_1,L_2, \bn}^{\mathrm{red}})\\
        &=  e(-N_{C,L_1,L_2, \mathbf{0}}^{\vir})\cdot e(-N_{C,L_1,L_2, \bn}^{\mathrm{red}}).
    \end{align*}
    Integrating over $C^{[\bn]}$ concludes the proof.
\end{proof}
\subsection{Genus zero}\label{sec: genus 0 DT}
As recalled above, by the universality step we are reduced to perform computations only in the genus 0 case. Therefore, we assume for the rest of the section that the base curve $C=\BP^1$ is a projective line. \\

Let $\BC^*$ act on the points of $\BP^1$ by $t_3 \cdot x_0= t_3x_0 $ around the affine chart at 0, and by $t_3\cdot x_\infty=t_3^{-1}x_\infty$ around the affine chart at $\infty$. Let $L$ be a line bundle on $\BP^1$. By the identification $L\cong \oO_{\BP^1}(\deg L\cdot [\infty])$, the line bundle $L$ is endowed with a $\BC^*$-equivariant structure, whose fibers over the $\BC^*$-fixed points of $\BP^1$ are the $\BC^*$-representations
\begin{align*}
        L|_p=\begin{cases}
    \BC\otimes 1, & p=0,\\
    \BC\otimes t_3^{-\deg L}, & p=\infty.
    \end{cases}
\end{align*}
Set $\overline{\TT}=\TT\times \BC^*$. The $\BC^*$-action on $\BP^1$ naturally lifts to a $\overline{\TT}$-action on  the skew nested Hilbert schemes $(\BP^1)^{[\bn]}$, where $\TT$ acts trivially. In particular, the $\BC^*$-equivariant structure on line bundles $L$ on $\BP^1$ naturally lifts the $\TT$-equivariant  $K$-theory class $ N_{C,L_1,L_2, \bn}^{\mathrm{red}}$  to a $\overline{\TT}$-equivariant $K$-theory class. Therefore, the series $ Z(\BP^1,L_1, L_2),  Z^K(\BP^1,L_1, L_2)$ are amenable for $\overline{\TT}$-localisation.
\smallbreak
By a small abuse of notation, consider $\CA, \CN$ as  $K$-theory classes on $\BP^1$. With respect to the equivariant structure on line bundles described above, we have
\begin{align*}
    \CA|_{0}&=\mathsf{Z}_\lambda,\\
    \CA|_{\infty}&=\mathsf{Z}_\lambda|_{t_1=t_1t_3^{-\deg L_1}, t_2=t_2t_3^{-\deg L_2}},        \\
    \CN|_{0}&=(1-t_1)(1-t_2),\\
     \CN|_{\infty}&=(1-t_1t_3^{-\deg L_1})(1-t_2t_3^{-\deg L_2}),
\end{align*}
where $\mathsf{Z}_\lambda$ was introduced in \eqref{eqn: Z lambda}.
\smallbreak
Let $\bn$ be a skew plane partition of shape $\BZ^{2}_{\geq 0}\setminus \lambda$. Define the $\overline{\TT}$-representations
\begin{equation}\label{eqn: vertex DT}
\begin{split}
    \mathsf{Z}_{\bn}&=\sum_{(i,j)\in \BZ^{2}_{\geq 0}\setminus \lambda}\sum_{\alpha=0}^{n_{ij}-1}t_1^{-i}t_2^{-j}t_3^{-\alpha},\\
    \mathsf{v}_{\bn}&=  \mathsf{Z}_{\bn}-  \overline{\mathsf{Z}}_{\bn}\cdot t_1t_2t_3-(1-t_1)(1-t_2)\left(-t_3\cdot     \mathsf{Z}_{\lambda}\overline{\mathsf{Z}}_{\bn} +\overline{\mathsf{Z}}_{\lambda}\mathsf{Z}_{\bn}+(1-t_3)\mathsf{Z}_{\bn}\overline{\mathsf{Z}}_{\bn}\right).
    \end{split}
\end{equation}
Denote by $s_3$ the generator of the equivariant cohomology $H^*_{\BC^*}(\pt)$. Recall that, for 
\[V=\sum_{(\mu_1, \mu_2, \mu_3)}t_1^{\mu_1}t_2^{\mu_2}t_3^{\mu_3}-\sum_{(\nu_1, \nu_2, \nu_3)}t_1^{\nu_1}t_2^{\nu_2}t_3^{\nu_3}\]  a virtual $\overline{\TT}$-representation such that $(\nu_1, \nu_2, \nu_3)\neq (0,0,0)$, we have
\begin{align*}
    e(V)&=\frac{\prod_{(\mu_1, \mu_2, \mu_3)}(\mu_1s_1+\mu_2s_2+\mu_3s_3)}{\prod_{(\nu_1, \nu_2, \nu_3)}(\nu_1s_1+\nu_2s_2+\nu_3s_3)},\\
    \widehat{\mathfrak{e}}(V)&=\frac{\prod_{(\mu_1, \mu_2, \mu_3)}[t_1^{\mu_1}t_2^{\mu_2}t_3^{\mu_3}]}{\prod_{(\nu_1, \nu_2, \nu_3)}[t_1^{\nu_1}t_2^{\nu_2}t_3^{\nu_3}]}.
\end{align*}

\begin{theorem}\label{thm: toric loc}
    Let $\bn$ be a skew plane partition of shape $\BZ^2_{\geq 0}\setminus \lambda$. We have identities of $\overline{\TT}$-representations
    \begin{align*}
        \int_{(\BP^1)^{[\mathbf{n}]}}e(-N_{\BP^1,L_1,L_2, \bn}^{\mathrm{red}})&=\left.\left(\sum_{\bn=\bn_1+\bn_2}e(- \mathsf{v}_{\bn_1})\cdot e(-  \mathsf{v}_{\bn_2})|_{s_1=s_1-\deg L_1 s_3, s_2=s_2-\deg L_2 s_3, s_3=-s_3}\right)\right|_{s_3=0},\\
         \widehat{\chi}\left((\BP^1)^{[\mathbf{n}]},\widehat{\mathfrak{e}}(-N_{\BP^1,L_1,L_2, \bn}^{\mathrm{red}})  \right)&=\left.\left(\sum_{\bn=\bn_1+\bn_2}\widehat{\mathfrak{e}}(- \mathsf{v}_{\bn_1})\cdot \widehat{\mathfrak{e}}(-  \mathsf{v}_{\bn_2})|_{t_1=t_1t_3^{-\deg L_1}, t_2=t_2t_3^{-\deg L_2}, t_3=t_3^{-1}}\right)\right|_{t_3=1}.
    \end{align*}
\end{theorem}
\begin{proof}
    We prove the first claim, as the second one follows by an analogous reasoning. Denote by $T^{\vir}$ the virtual tangent space of $(\BP^1)^{[\mathbf{n}]}$. By $\overline{\TT}$-equivariant virtual localisation, we have
    \begin{align*}
            \int_{(\BP^1)^{[\mathbf{n}]}}e(-N_{\BP^1,L_1,L_2, \bn}^{\mathrm{red}})=\left.\left(\int_{((\BP^1)^{[\mathbf{n}]})^{\overline{\TT}}}e(-N_{\BP^1,L_1,L_2, \bn}^{\mathrm{red}}-T^{\vir})|_{((\BP^1)^{[\mathbf{n}]})^{\overline{\TT}}}\right)\right|_{s_3=0}.
    \end{align*}
    Notice that a $\overline{\TT}$-fixed flag $(Z_\Box)_\Box\in ((\BP^1)^{[\mathbf{n}]})^{\overline{\TT}}$ decomposes as 
    \begin{align*}
         (Z_\Box)_\Box&=(Z^0_\Box)_\Box + (Z^\infty_\Box)_\Box\\
        &= \bn_1\cdot [0]+\bn_2\cdot [\infty],
    \end{align*}
    for some skew plane partitions $\bn_1, \bn_2$. 
    This implies that 
    the $\overline{\TT}$-fixed locus $((\BP^1)^{[\mathbf{n}]})^{\overline{\TT}}$ is reduced, zero-dimensional and  in bijection with 
    \begin{align*}
        \set{(\bn_1, \bn_2)| \bn_1, \bn_2 \mbox{ skew plane partitions of shape } \BZ^2_{\geq 0}\setminus \lambda \mbox{ such that } \bn= \mathbf{n}_1+\mathbf{n}_2}.
    \end{align*}
    In particular, this implies that the virtual $\overline{\TT}$-representation $T^{\vir}_{(\bn_1, \bn_2)} $ is  $\overline{\TT}$-movable at all fixed points corresponding to $(\bn_1, \bn_2)$.

    Fix now a $\overline{\TT}$-fixed point $(\bn_1, \bn_2)$. By construction, there is  an identity of virtual $\overline{\TT}$-representations
    \begin{align*}
      N_{\BP^1,L_1,L_2, \bn}^{\mathrm{red}}+T^{\vir}=\BE|_{(\BP^1)^{[\mathbf{n}]}}^\vee- P_{C, L_1, L_2}.
    \end{align*}
    Taking the fiber over the fixed point corresponding to $(\bn_1, \bn_2)$ yields
    \begin{multline}\label{eqn: sum 0 inft}
         ( N_{\BP^1,L_1,L_2, \bn}^{\mathrm{red}}+T^{\vir})|_{(\bn_1, \bn_2)}=\\
        \sum_{p\in \{0, \infty\}}\left(\sum_{(i,j)\in \BZ^2_{\geq 0}\setminus\lambda}\left(H^0(\oO_{Z^p_{ij}}\otimes  L_1^{-i}L_2^{-j})\cdot t_1^{-i}t_2^{-j} - H^0( \oO_{Z^p_{ij}}\otimes L_1^{-i-1}L_2^{-j-1}\otimes \omega_{\BP^1})^*\cdot t_1^{i+1}t_2^{j+1}\right.\right.\\
 \left.-H^0(\oO_{Z^p_{ij}}\otimes L_1^{-i}L_2^{-j}\otimes \CN\otimes \CA^*)\cdot t_1^{-i}t_2^{-j}+H^0(\oO_{Z^p_{ij}}\otimes L_1^{-i}L_2^{-j}\otimes \CN^*\otimes \CA^*\otimes \omega_{\BP^1})^*\cdot t_1^{i}t_2^{j}\right)\\
    \left.  -  \sum_{(i,j), (l,k)\in\BZ^2_{\geq 0}\setminus\lambda } \RR\Hom( \oO_{Z^p_{ij}}\otimes \CN^*,\oO_{Z^p_{lk}}\otimes L_1^{i-l}L_2^{j-k})\cdot t_1^{i-l}t_2^{j-k}\right).
    \end{multline}
    Let $U_0, U_\infty$ be the open affine charts in $\BP^1$ around respectively $0, \infty$. We have that 
    \begin{align*}
        \omega_{\BP^1}|_{U_p}=\begin{cases}
            \oO_{U_0}\otimes t_3^{-1} & p=0,\\
            \oO_{U_\infty}\otimes t_3  & p=\infty,
        \end{cases}
    \end{align*}
    where we remark that the equivariant structure on $\omega_{\BP^1}$ is canonically induced by the $\BC^*$-action on $\BP^1$. Analogously, for any line bundle $L$ on $\BP^1$ we have that 
     \begin{align}\label{eqn: struct at inf}
       L|_{U_p}=\begin{cases}
            \oO_{U_0} & p=0,\\
            \oO_{U_\infty}\otimes t_3^{-\deg L}  & p=\infty.
        \end{cases}
    \end{align}
Notice that, for all $(i,j), (l,k)$,  by \cite[Lemma 11.3.3]{Ric_Book} we have an identity of $\overline{\TT}$-representations
    \begin{align*}
        \RR\Hom( \oO_{Z^0_{ij}}\otimes \CN^*,\oO_{Z^0_{lk}}\otimes L_1^{i-l}L_2^{j-k})=(1-t_3)\cdot H^0(\oO_{Z^0_{ij}})^*\cdot H^0(\oO_{Z^0_{lk}}) \cdot (1-t_1)(1-t_2),
    \end{align*}
    and that
    \begin{align*}
      \mathsf{Z}_{\bn_1}&= \sum_{(i,j)\in \BZ^2_{\geq 0}\setminus\lambda}H^0(\oO_{Z^0_{ij}}\otimes  L_1^{-i}L_2^{-j})\cdot t_1^{-i}t_2^{-j},\\
     \overline{\mathsf{Z}}_{\bn_1}\cdot t_1t_2t_3 &=\sum_{(i,j)\in \BZ^2_{\geq 0}\setminus\lambda} H^0( \oO_{Z^0_{ij}}\otimes L_1^{-i-1}L_2^{-j-1}\otimes \omega_{\BP^1})^*\cdot t_1^{i+1}t_2^{j+1}.
    \end{align*}
    By the latter identities (and their natural analogs for the terms involving $\CN, \CA$) we have that the contribution of $p=0$ to the right-hand-side of \eqref{eqn: sum 0 inft} is
    \begin{align*}
        \mathsf{v}_{\bn_1}=\mathsf{Z}_{\bn_1}- \overline{\mathsf{Z}}_{\bn_1}\cdot t_1t_2t_3-(1-t_1)(1-t_2)(-t_3\mathsf{Z}_{\lambda}\overline{\mathsf{Z}}_{\bn_1} +\overline{\mathsf{Z}}_{\lambda}\mathsf{Z}_{\bn_1}+(1-t_3)\mathsf{Z}_{\bn_1}\overline{\mathsf{Z}}_{\bn_1}).
    \end{align*}
    Similarly, using the equivariant structure \eqref{eqn: struct at inf} around $p=\infty$, we have that the contribution of $p=\infty$ to the right-hand-side of \eqref{eqn: sum 0 inft} is
    \begin{align}\label{eqn: V2}
          \mathsf{v}_{\bn_2}|_{t_1=t_1t_3^{-\deg L_1}, t_2=t_2t_3^{-\deg L_2}, t_3=t_3^{-1}}.
    \end{align}
    Summing all up, we have proved that 
\[
 ( N_{\BP^1,L_1,L_2, \bn}^{\mathrm{red}}+T^{\vir})|_{(\bn_1, \bn_2)}=  \mathsf{v}_{\bn_1}+    \mathsf{v}_{\bn_2}|_{t_1=t_1t_3^{-\deg L_1}, t_2=t_2t_3^{-\deg L_2}, t_3=t_3^{-1}}.
\]
By \cite{MNOP_1} the term $ \mathsf{v}_{\bn_1}$ is $\overline{\TT}$-movable, which implies that the measure $e(- \mathsf{v}_{\bn_1}) $ is well-defined. Similarly, the term \eqref{eqn: V2} is $\overline{\TT}$-movable,  
    and therefore 
    \begin{align*}
          \int_{(\BP^1)^{[\mathbf{n}]}}e(-N_{\BP^1,L_1,L_2, \bn}^{\mathrm{red}})&=\left.\left(\sum_{\bn=\bn_1+\bn_2}e(- \mathsf{v}_{\bn_1})\cdot e(-  \mathsf{v}_{\bn_2}|_{t_1=t_1t_3^{-\deg L_1}, t_2=t_2t_3^{-\deg L_2}, t_3=t_3^{-1}})\right)\right|_{s_3=0}\\
          &=\left.\left(\sum_{\bn=\bn_1+\bn_2}e(- \mathsf{v}_{\bn_1})\cdot e(-  \mathsf{v}_{\bn_2})|_{s_1=s_1-\deg L_1 s_3, s_2=s_2-\deg L_2 s_3, s_3=-s_3}\right)\right|_{s_3=0}.
    \end{align*}
\end{proof}
Define the generating series
\begin{align*}
    \mathsf{V}_\lambda(q)&=\sum_{\bn}e(-\mathsf{v}_{\bn})\cdot q^{|\mathbf{n}|}\in \BQ(s_1, s_2, s_3)\llbracket q \rrbracket,\\
\widehat{\mathsf{V}}_\lambda(q)&=\sum_{\bn}\widehat{\mathfrak{e}}(-\mathsf{v}_{\bn})\cdot q^{|\mathbf{n}|}\in \BQ(t_1^{1/2}, t_2^{1/2}, t_3^{1/2})\llbracket q \rrbracket,
\end{align*}
where the sum is over all skew plane partitions of shape $\BZ^2_{\geq 0}\setminus \lambda$. Summing the identities of \Cref{thm: toric loc} over all skew plane partitions, we obtain the following corollary.

\begin{corollary}\label{cor: vertex top}
   Let $L_1, L_2$ be line bundles on $\BP^1$. We have identities of generating series
    \begin{align*}
  \sum_{\mathbf{n}}q^{|\mathbf{n}|} \int_{(\BP^1)^{[\mathbf{n}]}}e(-N_{\BP^1,L_1,L_2, \bn}^{\mathrm{red}})&=  \left(\mathsf{V}_\lambda(q)\cdot \mathsf{V}_\lambda(q)|_{s_1=s_1-\deg L_1 s_3, s_2=s_2-\deg L_2 s_3, s_3=-s_3}\right)|_{s_3=0},\\
    \sum_{\mathbf{n}}q^{|\mathbf{n}|} \widehat{\chi}\left((\BP^1)^{[\mathbf{n}]},\widehat{\mathfrak{e}}(-N_{\BP^1,L_1,L_2, \bn}^{\mathrm{red}})  \right)&=\left.\left(\widehat{\mathsf{V}}_\lambda(q)\cdot \widehat{\mathsf{V}}_\lambda(q)|_{t_1=t_1t_3^{-\deg L_1}, t_2=t_2t_3^{-\deg L_2}, t_3=t_3^{-1}}\right)\right|_{t_3=1}.
\end{align*}
\end{corollary}
By the above result, we can finally compute the universal series of \Cref{thm: universal series}.
\begin{theorem}\label{thm: DT_ explicit univ series}
   The universal series are given by
\begin{align*}
    A_\lambda(q)&=\prod_{\Box\in \lambda}\frac{1}{(-\ell(\Box)s_1+(a(\Box)+1)s_2)((\ell(\Box)+1)s_1-a(\Box)s_2)}\cdot \left(\mathsf{V}_\lambda(q)\cdot \mathsf{V}_\lambda(q)|_{ s_3=-s_3}\right)|_{s_3=0},\\
    B_\lambda(q)&=\prod_{\Box\in \lambda}\frac{( -\ell(\Box) s_1+(a(\Box)+1)s_2)^{\ell(\Box)}}{( (\ell(\Box)+1)s_1-a(\Box)s_2)^{\ell(\Box)+1}}\cdot \left.\left(\mathsf{V}_\lambda(q)|_{ s_3=-s_3}\cdot \mathsf{V}_\lambda(q)^{-1}|_{s_1=s_1+2s_3, s_3=-s_3}\right)\right|^{\frac{1}{2}}_{s_3=0}, \\
    C_\lambda(q)&=\prod_{\Box\in \lambda}\frac{(\ell(\Box)+1)s_1-a(\Box)s_2)^{a(\Box)}}{(-\ell(\Box)s_1+(a(\Box)+1)s_2)^{a(\Box)+1}}\cdot \left.\left(\mathsf{V}_\lambda(q)|_{ s_3=-s_3}\cdot \mathsf{V}_\lambda(q)^{-1}|_{s_2=s_2+2s_3, s_3=-s_3}\right)\right|^{\frac{1}{2}}_{s_3=0},\\
      \widehat{A}_\lambda(q)&=\prod_{\Box\in \lambda}\frac{1}{[t_1^{-\ell(\Box)}t_2^{a(\Box)+1}][t_1^{\ell(\Box)+1}t_2^{-a(\Box)}]}\cdot \left.\left(\widehat{\mathsf{V}}_\lambda(q)\cdot \widehat{\mathsf{V}}_\lambda(q)|_{ t_3=t^{-1}_3}\right)\right|_{t_3=1}\\
    \widehat{B}_\lambda(q)&=\prod_{\Box\in \lambda}\frac{[t_1^{-\ell(\Box)}t_2^{a(\Box)+1}]^{\ell(\Box)}}{[t_1^{\ell(\Box)+1}t_2^{-a(\Box)}]^{\ell(\Box)+1}}\cdot \left.\left(\widehat{\mathsf{V}}_\lambda(q)|_{ t_3=t_3^{-1}}\cdot \widehat{\mathsf{V}}_\lambda(q)^{-1}|_{t_1=t_1t^2_3, t_3=t^{-1}_3}\right)\right|^{\frac{1}{2}}_{t_3=1},  \\
    \widehat{C}_\lambda(q)&=\prod_{\Box\in \lambda}\frac{[t_1^{\ell(\Box)+1}t_2^{-a(\Box)}]^{a(\Box)}}{[t_1^{-\ell(\Box)}t_2^{a(\Box)+1}]^{a(\Box)+1}}\cdot \left.\left(\widehat{\mathsf{V}}_\lambda(q)|_{ t_3=t_3^{-1}}\cdot \widehat{\mathsf{V}}_\lambda(q)^{-1}|_{t_2=t_2t^2_3, t_3=t^{-1}_3}\right)\right|^{\frac{1}{2}}_{t_3=1}.
\end{align*}
\end{theorem}
\begin{proof}
We spell out the argument for the first three equalities, as the other ones follow from an analogous reasoning.   

Set the normalised universal series
\begin{align*}
       A'_\lambda(q)&=e\left( -T_\lambda\right)^{-1}\cdot A_\lambda(q),\\
       B'_\lambda(q)&=e\left(  -t_1\frac{\partial}{\partial t_1}T_\lambda\right)^{-1}\cdot B_\lambda(q), \\
       C'_\lambda(q)&=e\left(  -t_2\frac{\partial}{\partial t_2}T_\lambda\right)^{-1}\cdot C_\lambda(q).
\end{align*}
Notice that, by construction, the series $A'_\lambda(q), B'_\lambda(q), C'_\lambda(q) $ are in $1+\BQ(s_1, s_2)\llbracket q\rrbracket$. 
By the proof of \Cref{thm: universal series} and \Cref{eqn: lemma reduced}, we have that
\begin{align*}
     \sum_{\mathbf{n}}q^{|\mathbf{n}|} \int_{C^{[\mathbf{n}]}}e(-N_{C,L_1,L_2, \bn}^{\mathrm{red}})= A'_{\lambda}(q)^{1-g}\cdot B'_{\lambda}(q)^{\deg L_1}\cdot C'_{\lambda}(q)^{\deg L_2}.
\end{align*}
By applying \Cref{cor: vertex top} to the tuples 
\begin{align*}
    (\BP^1, \oO, \oO), \quad 
     (\BP^1, \oO, \oO(-2)), \quad 
      (\BP^1, \oO(-2), \oO),
\end{align*}
we end up with a system of three equations. Extracting the logarithm and solving the resulting system uniquely determines the three required universal series, thus proving the claim.
\end{proof}
\begin{remark}
Let $(L_1, L_2)\neq (L_1', L_2')$ be any two pairs of line bundles on $\BP^1$ satisfying $L_1\otimes L_2 \cong L_1'\otimes L_2'\cong \omega_{\BP^1}$. Then in the  proof of \Cref{thm: DT_ explicit univ series} we could have alternatively  taken 
 \begin{align*}
    (\BP^1, \oO, \oO), \quad 
     (\BP^1, L_1, L_2), \quad 
      (\BP^1, L_1', L_2'),
\end{align*}
to end up with a system of three equations, whose solution determines the universal series. By the uniqueness of the universal series, this yields (a priori non-trivial) relations involving the generating series $\mathsf{V}_{\lambda}(q), \widehat{\mathsf{V}}_{\lambda}(q)$.
\end{remark}
\subsection{The 1-leg vertex}\label{Sec: DT 1 leg vert}
The key point of \Cref{cor: vertex top} is that, for all Young diagrams $\lambda$,  the generating series \eqref{eqn: red Z gen} are reduced to   the knowledge of the 
generating series $\mathsf{V}_\lambda(q), \widehat{\mathsf{V}}_\lambda(q)$. 

Remarkably, the latter reproduce the \emph{vertex partition functions} obtained by starting from the vertex/edge formalism introduced in \cite{MNOP_1} in the case of  one infinite leg. We briefly recall the combinatorial data used to express the original  vertex formalism.

A \emph{plane partition} is a (possibly infinite) collection of boxes $\pi\in \BZ^3_{\geq 0}$, such that if one among $(i+1,j,k), (i,j+1,k), (i,j,k+1)$ is in $\pi$, then $(i,j,k)$ belongs to $ \pi$ as well. We say that a plane partition $\pi$ is a \emph{1-leg plane partition with asymptotic profile} $\lambda$ if, for $N\gg 0$, we have that 
\begin{align*}
    \pi\setminus \left([0,N]\times [0, N]\times [0, N]\right)=\set{(i,j,k)| (i,j)\in \lambda \mbox{ and } k\geq N+1}.
\end{align*}
Following \cite[Sec.~4.4]{MNOP_1}, we define the \emph{normalised size} of $\pi$ to be
\begin{align*}
    |\pi|=|\set{(i,j,k)\in \pi| (i,j)\notin \lambda}|.
\end{align*}
To each skew plane partition $\bn$ of shape $\mathbb{\BZ}^2_{\geq 0}\setminus \lambda$, we associate a 1-leg plane partition $\pi$   with asymptotic profile $\lambda$ by setting
\[
\pi=\set{(i,j,k)\in \BZ^3_{\geq 0}| (i,j)\in \lambda}\cup \set{(i,j,k)\in \BZ^3_{\geq 0}| (i,j)\notin \lambda \mbox{ and } k\leq n_{ij}-1}.
\]
This association is easily seen to provide a bijection  between the sets
\begin{align*}
    \set{\mbox{skew plane partitions of shape } \mathbb{\BZ}^2_{\geq 0}\setminus \lambda} \longleftrightarrow \set{\mbox{1-leg plane partitions with asymptotic profile } \lambda },
\end{align*}
which furthermore preserves the size.

\begin{figure}[H]
    \centering

\tdplotsetmaincoords{70}{120}
\begin{tikzpicture}[tdplot_main_coords, scale=0.8]

\newcommand{\drawcube}[3]{
  \pgfmathsetmacro{\x}{#1}
  \pgfmathsetmacro{\y}{#2}
  \pgfmathsetmacro{\z}{#3}
  
  \fill[gray!40] (\x,\y,\z+1) -- ++(1,0,0) -- ++(0,1,0) -- ++(-1,0,0) -- cycle;
  \fill[gray!30] (\x+1,\y,\z) -- ++(0,0,1) -- ++(0,1,0) -- ++(0,0,-1) -- cycle;
  \fill[gray!20] (\x,\y+1,\z) -- ++(1,0,0) -- ++(0,0,1) -- ++(-1,0,0) -- cycle;
  
  \draw[thick] (\x,\y,\z) -- ++(1,0,0) -- ++(0,1,0) -- ++(-1,0,0) -- cycle;
  \draw[thick] (\x,\y,\z) -- ++(0,0,1);
  \draw[thick] (\x+1,\y,\z) -- ++(0,0,1);
  \draw[thick] (\x+1,\y+1,\z) -- ++(0,0,1);
  \draw[thick] (\x,\y+1,\z) -- ++(0,0,1);
  \draw[thick] (\x,\y,\z+1) -- ++(1,0,0) -- ++(0,1,0) -- ++(-1,0,0) -- cycle;
}

\foreach \x/\y/\z in {
  0/0/0, 1/0/0, 2/0/0,
  0/1/0, 1/1/0,
  0/2/0,
  0/0/1, 1/0/1
}{
  \drawcube{\x}{\y}{\z}
}

\foreach \x in {3,...,8} {
  \drawcube{\x}{0}{0}
}

\draw[dotted, thick] (9,0,0) -- ++(2,0,0);

\end{tikzpicture}
  \caption{A 1-leg plane partition of normalised size 5, with asymptotic profile a Young diagram of size 1.}
    \label{fig: DT 1 vertex}
\end{figure}
\smallbreak
Under the above correspondence it is immediate to verify that, given a skew plane partition $\bn$, the vertex term $\mathsf{v}_{\bn}$ reproduces precisely\footnote{Due to different conventions on the torus $\overline{\TT}$-action in \cite{MNOP_1}, the vertex terms $\mathsf{v}_{\bn}$ are matched only after the change of variables $t_i\mapsto t_i^{-1}$, for $i=1, 2,3$. We remark that the  convention on the torus action adapted in this paper is consistent with the one used in \cite{Okounk_Lectures_K_theory, Arb_K-theo_surface, Mon_double_nested}, while the one of \cite{MNOP_1} is consistent, for instance, with \cite{FMR_higher_rank, Mon_PhD}. } the \emph{normalised vertex} term of \cite[Sec.~4.9]{MNOP_1} of the corresponding 1-leg plane partition $\pi$.
The expressions in  \Cref{cor: full DT} can be computed using \emph{relative} invariants, combining the formula for the \emph{capped} 1-leg vertex\footnote{To be precise, \cite{KOO_2_legDT} deals with the 1-leg vertex in Pandharipande-Thomas theory. Nevertheless,  proving a formula for the  1-leg vertex in any  of the two theories yields equivalent results, see \Cref{sec: DT/PT} and \cite{KLT_DTPT}.}  $\widehat{\mathsf{V}}_\lambda(q)$  \cite[Eqn.~(29)]{KOO_2_legDT},  a plethystic epression as a function  of the torus equivariant parameters and a basis of the representation ring of the symmetric group, and  the correct  \emph{glueing operator} (see e.g.~\cite{KOO_2_legDT}).

\subsection{Degree zero}
We prove  an explicit closed formula for the universal series in the degree 0 case using the language of plethystic exponentials. Since there is only one Young diagram of size 0, by \eqref{eqn: DT as localised} the resulting invariants compute the Donaldson-Thomas partition functions in degree 0. 
\smallbreak
 Given a  formal power series 
 \[f(t_1, t_2, t_3; q)\in   \BQ(t_1, t_2, t_3)[\![q]\!],\]
 such that $f(t_1, t_2, t_3; 0)=0$, its \emph{plethystic exponential} is defined as 
\begin{align}\label{eqn: on ple} 
\Exp(f(t_1, t_2, t_3;q)) &:= \exp\Big( \sum_{n=1}^{\infty} \frac{1}{n} f(t_1^n, t_2^n, t_3^n;q^n) \Big),
\end{align}
viewed as an element of $1+\BQ(t_1, t_2, t_3)[\![q]\!]$. Recall that the plethystic exponential enjoys the useful identities
\begin{align*}
 \Exp(f+g)&=\Exp(f)\cdot \Exp(g),\\
    \Exp(q)&=\frac{1}{1-q}.
\end{align*}
Recall that the \emph{MacMahon} function is defined as
\begin{align}\label{eqn: MM}
    \mathsf{M}(q)=\prod_{d\geq 1}\frac{1}{(1-q^d)^d}.
\end{align}
\begin{theorem}\label{thm: explicit DT 0 universal}
Let $\lambda=\varnothing$. The universal series satisfy
    \begin{align*}
    A_\varnothing(q)&=\mathsf{M}(-q)^{-2\frac{(s_1+s_2)^2}{s_1s_2}},\\
    B_\varnothing(q)&=\mathsf{M}(-q)^{-1}, \\
    C_\varnothing(q)&=\mathsf{M}(-q)^{-1},\\
      \widehat{A}_\varnothing(-q)&=\Exp\left(2\frac{[t_1t_2]^2}{[t_1][t_2]}\frac{1}{[(t_1t_2)^{1/2}q][(t_1t_2)^{1/2}q^{-1}]} \right),\\
    \widehat{B}_\varnothing(-q)&=\Exp\left(\frac{1}{2}\frac{(t_1t_2)^{\frac{1}{2}}+(t_1t_2)^{-\frac{1}{2}}}{[(t_1t_2)^{1/2}q][(t_1t_2)^{1/2}q^{-1}]} \right),  \\
    \widehat{C}_\varnothing(-q)&=\Exp\left(\frac{1}{2}\frac{(t_1t_2)^{\frac{1}{2}}+(t_1t_2)^{-\frac{1}{2}}}{[(t_1t_2)^{1/2}q][(t_1t_2)^{1/2}q^{-1}]} \right).
\end{align*}
\end{theorem}
\begin{proof}
    By \cite[Thm.~3.3.6]{Okounk_Lectures_K_theory}\footnote{See also \cite[Thm.~5.1]{thimm_orbi} for a more general and detailed proof.} there is an identity
    \begin{align*}
        \widehat{\mathsf{V}}_\varnothing(-q)=\Exp\left(\frac{[t_1t_2][t_1t_3][t_2t_3]}{[t_1][t_2][t_3]}\frac{1}{[(t_1t_2t_3)^{1/2}q][(t_1t_2t_3)^{1/2}q^{-1}]}\right).
    \end{align*}
    Applying the correct change of variables and restricting to $t_3=1$ yields
    \begin{align*}
       & \left.\left(\widehat{\mathsf{V}}_\lambda(-q)\cdot \widehat{\mathsf{V}}_\lambda(-q)|_{ t_3=t^{-1}_3}\right)\right|_{t_3=1}=\Exp\left(2\frac{[t_1t_2]^2}{[t_1][t_2]}\frac{1}{[(t_1t_2)^{1/2}q][(t_1t_2)^{1/2}q^{-1}]} \right),\\
       & \left.\left(\widehat{\mathsf{V}}_\varnothing(-q)|_{ t_3=t_3^{-1}}\cdot \widehat{\mathsf{V}}_\varnothing(-q)^{-1}|_{t_1=t_1t^2_3, t_3=t^{-1}_3}\right)\right|^{\frac{1}{2}}_{t_3=1}=\Exp\left(\frac{1}{2}\frac{(t_1t_2)^{\frac{1}{2}}+(t_1t_2)^{-\frac{1}{2}}}{[(t_1t_2)^{1/2}q][(t_1t_2)^{1/2}q^{-1}]} \right).
    \end{align*}
   By symmetry, an analogous calculation computes $ \widehat{C}_\varnothing(-q)$ as well.

   The first three universal series are computed analogously. By \cite[Thm.~1]{MNOP_2} there is an identity
   \begin{align*}
      \mathsf{V}_\varnothing(q)=\mathsf{M}(-q)^{-\frac{(s_1+s_2)(s_2+s_3)(s_1+s_3)}{s_1s_2s_3}},
   \end{align*}
   where $\mathsf{M}(q)$ denotes MacMahon's function. Applying the correct change of variables and restricting to $s_3=0$ yields the desired results.
\end{proof}
\subsection{Anti-diagonal restriction}
Similarly to \Cref{sec: andtifiagonal edge}, we show now that under the anti-diagonal restriction of the equivariant parameters, the universal series in \Cref{thm: DT_ explicit univ series} admit  more compact formulas. 
\begin{corollary}\label{cor: antidiagonal DT}
     In the anti-diagonal restriction, the universal series satisfy
\begin{align*}
    \left.A_\lambda(q)\right|_{s_1+s_2=0}&=(-s_1^{-2})^{|\lambda|}\prod_{\Box\in \lambda}\frac{1}{h(\Box)^2},\\
    \left.B_\lambda(-q)\right|_{s_1+s_2=0}&=(-1)^{n(\lambda)}s_1^{-|\lambda|}\prod_{\Box\in \lambda}\frac{1}{h(\Box)} \cdot \mathsf{M}(q)^{-1} \cdot \prod_{\Box\in \lambda}(1-q^{h(\Box)}), \\
    \left.C_\lambda(-q)\right|_{s_1+s_2=0}&=(-1)^{n(\overline{\lambda})}(-s_1)^{-|\lambda|}\prod_{\Box\in \lambda}\frac{1}{h(\Box)}\cdot \mathsf{M}(q)^{-1} \cdot \prod_{\Box\in \lambda}(1-q^{h(\Box)}),\\
    \left.  \widehat{A}_\lambda(q)\right|_{t_1t_2=1}&=(-1)^{|\lambda|}\prod_{\Box\in \lambda}\frac{1}{[t_1^{h(\Box)}]^2},\\
   \left. \widehat{B}_\lambda(-q)\right|_{t_1t_2=1}&=(-1)^{n(\lambda)}\prod_{\Box\in \lambda}\frac{1}{[t_1^{h(\Box)}]}\cdot \mathsf{M}(q)^{-1} \cdot \prod_{\Box\in \lambda}(1-q^{h(\Box)}),  \\
  \left.  \widehat{C}_\lambda(-q)\right|_{t_1t_2=1}&=(-1)^{n(\overline{\lambda})+|\lambda|}\prod_{\Box\in \lambda}\frac{1}{[t_1^{h(\Box)}]}\cdot \mathsf{M}(q)^{-1} \cdot \prod_{\Box\in \lambda}(1-q^{h(\Box)}).
\end{align*}
\end{corollary}
\begin{proof}
    We prove the claim for the last three identities, as the first three will follow by an analogous reasoning.  Notice that  the anti-diagonal restriction of the constant terms was computed in \Cref{cor: edge antidiagonal}.

   By \cite[Lemma 6]{OP_local_theory_curves}, the vertex partition function $ \widehat{\mathsf{V}}_\lambda(q)$ is divisible\footnote{To be precise, \cite[Lemma 6]{OP_local_theory_curves} proves that $ \mathsf{V}_\lambda(q)$ is divisible by $s_1+s_2 $, but the same reasoning applies in this case.} by $(1-t_1t_2)$, by which follows the vanishing
    \begin{align*}
        \left.\left(\widehat{\mathsf{V}}_\lambda(q)\cdot \widehat{\mathsf{V}}_\lambda(q)|_{ t_3=t^{-1}_3}\right)\right|_{t_3=1, t_1t_2=1}=1.
    \end{align*}

Let now $L_1, L_2$ be line bundles on $\BP^1$ so that $L_1\otimes L_2\cong \omega_{\BP^1}$. By \Cref{cor: vertex top}, we have that
\begin{align}\label{eqn: loc and restr}
      \left.\sum_{\mathbf{n}}q^{|\mathbf{n}|} \widehat{\chi}\left((\BP^1)^{[\mathbf{n}]},\widehat{\mathfrak{e}}(-N_{\BP^1,L_1,L_2, \bn}^{\mathrm{red}})  \right)\right|_{t_1t_2=1}&=\left.\left(\widehat{\mathsf{V}}_\lambda(q)\cdot \widehat{\mathsf{V}}_\lambda(q)|_{t_1=t_1t_3^{-\deg L_1}, t_2=t_2t_3^{-\deg L_2}, t_3=t_3^{-1}}\right)\right|_{t_3=1, t_1t_2=1}.
\end{align}
Notice that, in particolar, if we define
\begin{align*}
    t'_1&=t_1t_3^{-\deg L_1},\\
    t'_2&=t_2t_3^{-\deg L_2},\\
    t'_3&=t_3^{-1},
\end{align*}
we have $t'_1t'_2t'_3=t_1t_2t_3$.

For every skew plane partition $\bn$ of shape $\BZ^2_{\geq 0}\setminus \lambda$, there is a decomposition 
\begin{align}\label{eqn: splitting V}
\begin{split}
    \mathsf{v}_{\bn}&= \mathsf{v}_{\bn}^+-\overline{ \mathsf{v}_{\bn}^+}\cdot t_1t_2t_3,\\
    \mathsf{v}_{\bn}^+&=   \mathsf{Z}_{\bn} -(1-t_1)(1-t_2)\left( \overline{\mathsf{Z}}_{\lambda}\mathsf{Z}_{\bn}+\mathsf{Z}_{\bn}\overline{\mathsf{Z}}_{\bn}\right).
    \end{split}
\end{align}
By the discussion in \cite[pag.~1279]{MNOP_1} we have that $\mathsf{v}_{\bn}$ contains no (virtual) weight of the form $(t_1t_2t_3)^a$, for any $a\in \BZ$. Let now $a>0$. By the splitting in \eqref{eqn: splitting V}, we have that for all weights of the form $ (t_1t_2t_3)^a$ in $\mathsf{v}^+_{\bn}$, there is a weight $(t_1t_2t_3)^{1-a} $ is $\mathsf{v}^+_{\bn}$. Set therefore $\tilde{\mathsf{v}}^+_{\bn}$ to be the virtual $\overline{\TT}$-representation obtained from $ \mathsf{v}^+_{\bn}$ by removing all the weights of the form $(t_1t_2t_3)^a$, for some $a\in \BZ$. It follows that 
\begin{align*}
     & \mathsf{v}_{\bn}= \tilde{\mathsf{v}}_{\bn}^+-\overline{ \tilde{\mathsf{v}}_{\bn}^+}\cdot t_1t_2t_3,\\
     & \rk \tilde{\mathsf{v}}_{\bn}^+=\rk \mathsf{v}_{\bn}^+,\\
     &\rk \mathsf{v}_{\bn}^+=|\mathbf{n}|,\\
   &   \left(\tilde{\mathsf{v}}_{\bn}^+\right)^{\fix}=0.
\end{align*}
Since there are no weights of the form $(t_1t_2t_3)^a$ in $ \mathsf{v}_{\bn}$, we can compute
\begin{align*}
   \left. \widehat{\mathsf{V}}_\lambda(q)\right|_{t_3=(t_1t_2)^{-1}}&=\sum_{\bn}\widehat{\mathfrak{e}}(-\mathsf{v}_{\bn})|_{t_3=(t_1t_2)^{-1}}\cdot q^{|\mathbf{n}|}\\
     &=\sum_{\bn}\left.\frac{\widehat{\mathfrak{e}}(\overline{ \tilde{\mathsf{v}}_{\bn}^+}\cdot t_1t_2t_3)}{\widehat{\mathfrak{e}}(\tilde{\mathsf{v}}^+_{\bn})}\right|_{t_3=(t_1t_2)^{-1}}\cdot q^{|\mathbf{n}|}\\
     &=\sum_{\bn}\frac{\widehat{\mathfrak{e}}(\overline{ \tilde{\mathsf{v}}_{\bn}^+}|_{t_3=(t_1t_2)^{-1}})}{\widehat{\mathfrak{e}}(\tilde{\mathsf{v}}^+_{\bn}|_{t_3=(t_1t_2)^{-1}})}\cdot q^{|\mathbf{n}|}\\
     &=\sum_{\bn}(-q)^{ |\mathbf{n}|}\\
     &=\prod_{d\geq 1}\frac{1}{(1-(-q)^d)^d} \prod_{\Box\in \lambda} \frac{1}{1-(-q)^{h(\Box)}}.
\end{align*}
Notice that, in particular, the depends on the equivariant parameters $t_1, t_2$ drops. Therefore we compute
\begin{align}\label{eqn: final restrc DT}
\begin{split}
        \left.\left(\widehat{\mathsf{V}}_\lambda(q)\cdot \widehat{\mathsf{V}}_\lambda(q)|_{t_1=t'_1, t_2=t'_2, t_3=t'_3}\right)\right|_{t_3=1, t_1t_2=1}&=
  \left.\left(\widehat{\mathsf{V}}_\lambda(q)\cdot \widehat{\mathsf{V}}_\lambda(q)|_{t_1=t'_1, t_2=t'_2, t_3=t'_3}\right)\right|_{t_3=(t_1t_2)^{-1}, t_1t_2=1}\\
    &=\left.\left.\left( \widehat{\mathsf{V}}_\lambda(q)\right|_{t_3=(t_1t_2)^{-1}}\cdot  \left.\widehat{\mathsf{V}}_\lambda(q)|_{t_1=t'_1, t_2=t'_2, t_3=t'_3}\right|_{t'_3=(t'_1t'_2)^{-1}} \right)\right|_{t_1t_2=1}\\
    &=\left(\prod_{d\geq 1}\frac{1}{(1-(-q)^d)^d} \prod_{\Box\in \lambda} \frac{1}{1-(-q)^{h(\Box)}}\right)^{2}.
    \end{split}
\end{align}
Set the normalised series $ \widehat{A}'_\lambda(q),\widehat{B}'_\lambda(q),\widehat{C}'_\lambda(q)$ to be the series $ \widehat{A}_\lambda(q),\widehat{B}_\lambda(q),\widehat{C}_\lambda(q)$ divided by the constant terms as in the proof of \Cref{thm: DT_ explicit univ series}. By choosing $(L_1, L_2)=(\oO(-2), \oO), (\oO, \oO(-2))$, by \eqref{eqn: loc and restr},\eqref{eqn: final restrc DT} we compute
\begin{align*}
\left.\widehat{B}'_\lambda(q)\right|_{t_1t_2=1}^{-2}&=\left(\prod_{d\geq 1}\frac{1}{(1-(-q)^d)^d} \prod_{\Box\in \lambda} \frac{1}{1-(-q)^{h(\Box)}}\right)^{2},\\  \left.\widehat{C}'_\lambda(q)\right|_{t_1t_2=1}^{-2}&=\left(\prod_{d\geq 1}\frac{1}{(1-(-q)^d)^d} \prod_{\Box\in \lambda} \frac{1}{1-(-q)^{h(\Box)}}\right)^{2},
\end{align*}
by which we conclude the proof.
\end{proof}
We remark that if $X$ is Calabi-Yau -- in other words, when $L_1\otimes L_2\cong \omega_C$ -- combining \Cref{motives of spp} and \Cref{cor: antidiagonal DT} yields the identity
\begin{align*}
   \left.   \widehat{\chi}\left(C^{[\mathbf{n}]},\widehat{\mathfrak{e}}(-N_{C,L_1,L_2, \bn})  \right)\right|_{t_1t_2=1}=(-1)^{|\lambda|(1-g)+n(\lambda)\deg L_1+n(\overline{\lambda})\deg L_2}\cdot e(C^{[\mathbf{n}]}),
\end{align*}
which shows that the localised contribution to the Donaldson-Thomas partition function is topological, generalising the original computations  of \cite{MNOP_1} for toric Calabi-Yau threefold and in equivariant cohomology.
\subsection{Donaldson-Thomas partition functions}
By summing up the contribution of each Young diagram computed in \Cref{sec: computations}, we give a closed expression for the Donaldson-Thomas partition functions of local curves. 
\subsubsection{Degree 0}
In degree 0, the Donaldson-Thomas partition functions are essentially computed by \Cref{thm: explicit DT 0 universal}. We prove here a more compact formula for the partition functions.
\begin{corollary}\label{cor: as in Oko}
Let     $X=\Tot_C(L_1 \oplus L_2)$, where $L_1, L_2$ are line bundles over a smooth projective curve $C$. We have
\begin{align*}
    \DT_0(X, -q)&=\mathsf{M}(q)^{\int_X c_3(T_X\otimes \omega_X)},\\
      \widehat{\DT}_0(X, -q)&=\Exp\left(\chi(X, T_X+\omega_X-T_X^*-\omega_X^{-1}) \cdot \frac{1}{[(t_1t_2)^{1/2}q][(t_1t_2)^{1/2}q^{-1}]}\right).
\end{align*}
\end{corollary}
\begin{proof}
    We provide a proof for the second equality, as the first one follows analogously. 

The expression 
\[
\Exp\left(\chi(X, T_X+\omega_X-T_X^*-\omega_X^{-1}) \cdot \frac{1}{[(t_1t_2)^{1/2}q][(t_1t_2)^{1/2}q^{-1}]}\right)
\]
 is clearly multiplicative, depends only on the Chern numbers of $(C, L_1, L_2)$ and its constant term is 1. Therefore, by the same argument  as in the proof of \Cref{thm: universal series}, it is controlled by three universal series, which can be computed in the toric case $C=\BP^1$. The toric case follows immediately by a direct comparison with  \Cref{thm: explicit DT 0 universal} by further localisation with respect to  the $\BC^*$-action on $\BP^1$.
\end{proof}

\subsubsection{The full partition function}
Summing up the contributions of each Young diagram computed in \Cref{thm: DT_ explicit univ series}, we finally prove a closed formula for the Donaldson-Thomas partition functions.
\begin{corollary}\label{cor: full DT}
    Let  $d\geq 0$ and    $X=\Tot_C(L_1 \oplus L_2)$, where $L_1, L_2$ are line bundles over a smooth projective curve $C$ of genus $g$. We have
    \begin{align*}
     \DT_d(X, q)&= \sum_{|\lambda|=d} \left(q^{|\lambda|}{A}_{\lambda}(q)\right)^{1-g}\cdot \left(q^{-n(\lambda)}{B}_{\lambda}(q)\right)^{\deg L_1}\cdot \left(q^{-n(\overline{\lambda})}{C}_{\lambda}(q)\right)^{\deg L_2},\\
     \widehat{\DT}_d(X, q)&= \sum_{|\lambda|=d} \left(q^{|\lambda|}\widehat{A}_{\lambda}(q)\right)^{1-g}\cdot \left(q^{-n(\lambda)}\widehat{B}_{\lambda}(q)\right)^{\deg L_1}\cdot \left(q^{-n(\overline{\lambda})}\widehat{C}_{\lambda}(q)\right)^{\deg L_2},
\end{align*}
where the universal series are computed in \Cref{thm: DT_ explicit univ series}. Set $\deg L_i=k_i$.  In the anti-diagonal restriction, we have
\begin{multline*}
     \left.\DT_d(X,  -q)\right|_{s_1+s_2=0}=\\(-1)^{d \cdot k_2}\sum_{|\lambda|=d}q^{|\lambda|(1-g)-n(\lambda)k_1-n(\overline{\lambda})k_2}\cdot 
     \prod_{\Box \in \lambda}(s_1 h(\Box))^{2g-2-k_1-k_2}\cdot \left( \mathsf{M}(q)^{-1} \cdot \prod_{\Box\in \lambda}(1-q^{h(\Box)})\right)^{k_1+k_2},
\end{multline*}
\begin{multline*}
     \left.\widehat{\DT}_d(X,  -q)\right|_{t_1t_2=1}=\\
     (-1)^{d \cdot k_2}\sum_{|\lambda|=d}q^{|\lambda|(1-g)-n(\lambda)k_1-n(\overline{\lambda})k_2}\cdot
     \prod_{\Box \in \lambda}[t_1^{h(\Box)}]^{2g-2-k_1-k_2}\cdot \left( \mathsf{M}(q)^{-1} \cdot \prod_{\Box\in \lambda}(1-q^{h(\Box)})\right)^{k_1+k_2}.
\end{multline*}
\end{corollary}
\begin{proof}
    The claimed identities follow by combining Theorem \ref{thm: fixed}, \ref{thm: equality virtual classes}, \ref{thm: universal series}, \ref{thm: DT_ explicit univ series} and Corollary \ref{cor: antidiagonal DT}.
\end{proof}
We remark that, in the case $X$ is Calabi-Yau, the partition function in the anti-diagonal restriction of the equivariant parameters can be expressed (up to a sign) as a \emph{signed}  \emph{topological Euler characteristic} partition function.
\begin{corollary}\label{cor: DT is top}
    Let $X=\Tot_C(L_1\oplus L_2)$ be a local curve, where $L_1, L_2$ are line bundles over a smooth projective curve $C$ such that $L_1\otimes L_2\cong \omega_C$. Then we have
    \begin{align*}
          \left.\widehat{\DT}_d(X,  q)\right|_{t_1t_2=1}=(-1)^{d \cdot k_2} \DT^{\mathrm{top}}_d(X,- q).
    \end{align*}
\end{corollary}
\begin{proof}
    The proof follows from \Cref{cor: full DT}, \ref{cor: top DT} from direct comparison.
\end{proof}
\subsection{Refined limit}
While \Cref{cor: full DT} does compute the \emph{fully} $\TT$-equivariant   Donaldson-Thomas partition $ \widehat{\DT}_d(X,  q)$, closed explicit formulas are in general difficult to extract, away from the degree zero case and the anti-diagonal restriction of the equivariant parameters.

We study in this section the \emph{refined limit} of the Donaldson-Thomas partition function of \emph{Calabi-Yau} local curves $X$, by suitably scaling the equivariant parameters $t_1, t_2$ to infinity while keeping the Calabi-Yau weight $\kappa=t_1t_2$ constant, following the original treatment of the \emph{refined topological vertex} \cite{IKV_topological_vertex}.
\smallbreak
Let $X=\Tot_C(L_1\oplus L_2)$ be a local curve, where $L_1, L_2$ are line bundles over a smooth projective curve $C$ such that $L_1\otimes L_2\cong \omega_C$.
We introduce the \emph{refined Donaldson-Thomas partition function} of  $X$ as
\begin{align*}
    \DT^{\mathsf{ref}}_d(X, q)=\lim_{L\to \infty}\widehat{\DT}_d(X, q)|_{t_1=L\kappa^{\frac{1}{2}}, t_2=L^{-1}\kappa^{\frac{1}{2}}}\in \BQ(\kappa^{1/2})(\!( q )\!),
\end{align*}

We define the \emph{reduced} refined  partition function to be the normalisation by the degree zero partition function
\begin{align*}
 \overline{\DT}^{\mathsf{ref}}_d(X, q)&=\frac{ \DT^{\mathsf{ref}}_d(X, q)}{ \DT^{\mathsf{ref}}_0(X, q)}\in \BQ(\kappa^{1/2})(\!( q )\!).
\end{align*}
Recall that we defined two new variables $t_4, t_5$ in \eqref{eqn: new variables}.
\begin{theorem}\label{thm: refined DT full Aga}
    Let $X=\Tot_C(L_1\oplus L_2)$ be a local curve, where $L_1, L_2$ are line bundles over a smooth projective curve $C$ of genus $g$, such that $L_1\otimes L_2\cong \omega_C$. Set $\deg L_1=k_1$. We have
    \begin{align*}
       \DT^{\mathsf{ref}}_0(X, -q)&=\Exp\left((1-g)\frac{ (t_4t_5)^{\frac{1}{2}}+(t_4t_5)^{-\frac{1}{2}}
       }{[t_4][t_5]}
             \right),\\
         \overline{\DT}^{\mathsf{ref}}_d(X, -q)&=(-1)^{dk_1}\cdot 
         \sum_{|\lambda|=d}\left(t_4^{\left \lVert \lambda \right \rVert^2} \prod_{\Box\in \lambda} \frac{1}{(1- t_4^{a(\Box)+1}t_5^{-\ell(\Box)})(1-t_4^{a(\Box)} t_5^{-\ell(\Box)-1})} \right)^{1-g}\cdot t_4^{\frac{k_1\cdot \left \lVert \lambda \right \rVert^2}{2}}t_5^{ \frac{k_1\cdot \left \lVert \overline{\lambda} \right \rVert^2}{2}}.
    \end{align*}
    In particular, $   \overline{\DT}^{\mathsf{ref}}_d(X, -q)$ is  rational.
\end{theorem}
\begin{proof}
For the first identity,    by \Cref{thm: DT_ explicit univ series} we have
    \begin{align*}
             \widehat{\DT}_0(X, -q)=\Exp\left((1-g)\frac{1}{[(t_1t_2)^{1/2}q][(t_1t_2)^{1/2}q^{-1}]}\cdot \left(2\frac{[t_1t_2]^2}{[t_1][t_2]}-(t_1t_2)^{\frac{1}{2}}-(t_1t_2)^{-\frac{1}{2}} \right)
             \right).
    \end{align*}
The claimed identity follows by the vanishing of the limit
\begin{align*}
   \lim_{L\to \infty}\left. 2\frac{[t_1t_2]^2}{[t_1][t_2]}\right|_{t_1=L\kappa^{1/2}, t_2=L^{-1}\kappa^{1/2}}=0.
\end{align*}

For the second identity,    by the Calabi-Yau condition and \Cref{cor: full DT}, we can write
    \begin{align*}
         \widehat{\DT}_d(X, q)=\sum_{|\lambda|=d}W_{\lambda, 1}(q)^{1-g}\cdot W_{\lambda, 2}(q)^{\deg L_1},
    \end{align*}
    where $W_{\lambda, 1}(q), W_{\lambda,2}(q)$ are two universal series. We compute the two series in the refined limit by evaluating the Donaldson-Thomas partition function $    \widehat{\DT}_d(X, q)$ in the refined limit for $C\cong \BP^1$ and $L_1=\oO, \oO(-1)$.

   By construction, we have the relations
   \begin{align*}
       W_{\lambda, 1}(q)&=q^{|\lambda|+2n(\overline{\lambda})}\widehat{A}_{\lambda}(q)\cdot  \widehat{C}_{\lambda}(q)^{-2},\\
          W_{\lambda, 2}(q)&=q^{n(\overline{\lambda})-n(\lambda)}\widehat{B}_{\lambda}(q)\cdot\widehat{C}_{\lambda}(q)^{-1}.
   \end{align*}
   \underline{Step I.} We compute the first non-zero coefficients of the series $   W_{\lambda, 1}(q),   W_{\lambda, 2}(q)$ in the refined limit. 

   We start with the case $L_1=\oO(-1)$. Then by \Cref{thm: DT_ explicit univ series} we have that 
   \[
   W_{\lambda, 1}(q)\cdot W_{\lambda, 2}(q)^{-1}=q^{|\lambda|+n(\overline{\lambda})+n(\lambda)}\cdot \prod_{\Box\in \lambda}\left(\frac{[t_1^{-\ell(\Box)}t_2^{a(\Box)+1}]}{[t_1^{\ell(\Box)}t_2^{-a(\Box)-1}t_1t_2]}\right)^{a(\Box)-\ell(\Box)}+o(q^{|\lambda|+n(\overline{\lambda})+n(\lambda)+1}).
   \]
   Set $t_1=L\kappa^{1/2}$ and $t_2=L^{-1}\kappa ^{1/2}$, so that $t_1t_2=\kappa$. We have
   \begin{align*}
     \lim_{L\to \infty} \left. \prod_{\Box\in \lambda}\left(\frac{[t_1^{-\ell(\Box)}t_2^{a(\Box)+1}]}{[t_1^{\ell(\Box)}t_2^{-a(\Box)-1}t_1t_2]}\right)\right|_{t_1=L\kappa^{1/2},t_2=L^{-1}\kappa ^{1/2} }^{a(\Box)-\ell(\Box)}&=\lim_{L\to \infty}\prod_{\Box\in \lambda} \left(\frac{[(\kappa^{1/2})^{-\ell(\Box)+a(\Box)+1}L^{-h(\Box)}]}{[(\kappa^{1/2})^{\ell(\Box)-a(\Box)-1}L^{h(\Box)}\kappa]}\right)^{a(\Box)-\ell(\Box)}\\
     &=\prod_{\Box \in \lambda}\left(\frac{-(\kappa^{1/2})^{\frac{\ell(\Box)-a(\Box)-1}{2}}}{(\kappa^{1/2})^{\frac{\ell(\Box)-a(\Box)-1}{2}}\kappa^{1/2}}\right)^{a(\Box)-\ell(\Box)}\\
     &=\prod_{\Box\in \lambda}(-\kappa^{1/2})^{\ell(\Box)-a(\Box)}\\
     &= (-\kappa^{1/2})^{n(\lambda)}\cdot (-\kappa^{1/2})^{-n(\overline{\lambda})},
   \end{align*}
   which yields
   \begin{align*}
       W_{\lambda, 1}(q)\cdot W_{\lambda, 2}(q)^{-1}= (-1)^{|\lambda|}\cdot (-q\kappa^{1/2})^{n(\lambda)+\frac{|\lambda|}{2}}\cdot (-q\kappa^{-1/2})^{n(\overline{\lambda})+\frac{|\lambda|}{2}} + o(q^{|\lambda|+n(\overline{\lambda})+n(\lambda)+1}).
   \end{align*}
   Next, we address the case of $L_1=\oO$. By \Cref{thm: DT_ explicit univ series} we have that 
   \begin{align*}
        W_{\lambda, 1}(q)=q^{|\lambda|+2n(\overline{\lambda})}\cdot \prod_{\Box\in \lambda}\left(\frac{[t_1^{-\ell(\Box)}t_2^{a(\Box)+1}]}{[t_1^{\ell(\Box)}t_2^{-a(\Box)-1}t_1t_2]}\right)^{2a(\Box)+1}+o(q^{|\lambda|+2n(\overline{\lambda})+1}).
   \end{align*}
   The limit satisfies
   \begin{align*}
       \lim_{L\to \infty}\left.\prod_{\Box\in \lambda}\left(\frac{[t_1^{-\ell(\Box)}t_2^{a(\Box)+1}]}{[t_1^{\ell(\Box)}t_2^{-a(\Box)-1}t_1t_2]}\right)\right|^{2a(\Box)+1}_{t_1=L\kappa^{1/2},t_2=L^{-1}\kappa ^{1/2}}&=\lim_{L\to \infty}\prod_{\Box\in \lambda} \left(\frac{[(\kappa^{1/2})^{-\ell(\Box)+a(\Box)+1}L^{-h(\Box)}]}{[(\kappa^{1/2})^{\ell(\Box)-a(\Box)-1}L^{h(\Box)}\kappa]}\right)^{2a(\Box)+1}\\
     &=\prod_{\Box \in \lambda}\left(\frac{-(\kappa^{1/2})^{\frac{\ell(\Box)-a(\Box)-1}{2}}}{(\kappa^{1/2})^{\frac{\ell(\Box)-a(\Box)-1}{2}}\kappa^{1/2}}\right)^{2a(\Box)+1}\\
     &=\prod_{\Box\in \lambda}(-\kappa^{-1/2})^{2a(\Box)+1}\\
     &=  (-\kappa^{-1/2})^{2n(\overline{\lambda})+|\lambda|},
   \end{align*}
   which yields
   \begin{align*}
          W_{\lambda, 1}(q)=(-q\kappa^{-1/2})^{|\lambda|+2n(\overline{\lambda})}+o(q^{|\lambda|+2n(\overline{\lambda})+1}).
   \end{align*}
   Combining with the computation for $L_1=\oO(-1)$ with the combinatorial identities \eqref{eqn: comb ident}, we deduce that 
   \begin{align*}
   W_{\lambda, 1}(q)&=(-q\kappa^{-1/2})^{\left \lVert \lambda \right \rVert^2}\cdot(1+\dots),\\
       W_{\lambda, 2}(q)&=(-1)^{|\lambda|}\cdot (-q\kappa^{-1/2})^{\frac{\left \lVert \lambda \right \rVert^2}{2}}(-q\kappa^{1/2})^{-\frac{\left \lVert \overline{\lambda} \right \rVert^2}{2}}\cdot(1+\dots).
   \end{align*}
   \underline{Step II.} Recall that we defined the series  $ \widehat{A}'_\lambda(q),\widehat{B}'_\lambda(q),\widehat{C}'_\lambda(q)$ to be the series $ \widehat{A}_\lambda(q),\widehat{B}_\lambda(q),\widehat{C}_\lambda(q)$ divided by their constant terms in the proof of \Cref{cor: antidiagonal DT}. Set the normalised series
   \begin{align*}
   W_{\lambda, 1}(q)&=(-q\kappa^{-1/2})^{\left \lVert \lambda \right \rVert^2}\cdot  W'_{\lambda, 1}(q),\\
       W_{\lambda, 2}(q)&=(-1)^{|\lambda|}\cdot(-q\kappa^{-1/2})^{\frac{\left \lVert \lambda \right \rVert^2}{2}}(-q\kappa^{1/2})^{-\frac{\left \lVert \overline{\lambda} \right \rVert^2}{2}}W'_{\lambda, 2}(q).
   \end{align*}
   By \Cref{thm: DT_ explicit univ series}, we have that
   \begin{align*}
        W'_{\lambda, 1}(q)&=\widehat{A}'_{\lambda}(q)\cdot  \widehat{C}'_{\lambda}(q)^{-2}\\
        &=\left.\left(\widehat{\mathsf{V}}_\lambda(q)\cdot \widehat{\mathsf{V}}_\lambda(q)|_{t_2=t_2t^2_3, t_3=t^{-1}_3}\right)\right|_{t_3=1}.
   \end{align*}
and similarly
\begin{align*}
      W'_{\lambda, 2}(q)&= \left.\left(\widehat{\mathsf{V}}_\lambda(q)^{-1}|_{t_1=t_1t^2_3, t_3=t^{-1}_3 }\cdot \widehat{\mathsf{V}}_\lambda(q)|_{t_2=t_2t^2_3, t_3=t^{-1}_3}\right)\right|^{1/2}_{t_3=1}.
\end{align*}
Set the reduced vertex
\begin{align*}
\widehat{\overline{\mathsf{V}}}_\lambda(q)=\frac{\widehat{\mathsf{V}}_\lambda(q)}{\widehat{\mathsf{V}}_\varnothing(q)}.
\end{align*}
   To conclude the computation of $   \overline{\DT}^{\mathsf{ref}}_d(X, q)$, we are left to evaluate the reduced series
   \begin{align*}
         \frac{  W'_{\lambda, 1}(q)}{  W'_{\varnothing, 1}(q)}&=\left.\left(\widehat{\overline{\mathsf{V}}}_\lambda(q)\cdot \widehat{\overline{\mathsf{V}}}_\lambda(q)|_{t_2=t_2t^2_3, t_3=t^{-1}_3}\right)\right|_{t_3=1},\\
            \frac{  W'_{\lambda, 2}(q)}{  W'_{\varnothing, 2}(q)}&=\left.\left(\widehat{\overline{\mathsf{V}}}_\lambda(q)^{-1}|_{t_1=t_1t^2_3, t_3=t^{-1}_3 }\cdot \widehat{\overline{\mathsf{V}}}_\lambda(q)|_{t_2=t_2t^2_3, t_3=t^{-1}_3}\right)\right|^{1/2}_{t_3=1}
   \end{align*}
   in the refined limit. Fix integers $N\gg \epsilon >0$. We have
   \begin{align*}
       \lim_{L\to \infty}   \left.\frac{  W'_{\lambda, 1}(q)}{  W'_{\varnothing, 1}(q)}\right|_{t_1=L\kappa^{1/2},t_2=L^{-1}\kappa ^{1/2}}&=    \lim_{L\to \infty}  \left.\left(\widehat{\overline{\mathsf{V}}}_\lambda(q)\cdot \widehat{\overline{\mathsf{V}}}_\lambda(q)|_{t_2=t_2t^2_3, t_3=t^{-1}_3}\right)\right|_{t_1=L\kappa^{1/2},t_2=L^{-1}\kappa ^{1/2}, t_3=1}\\
       &=  \lim_{L\to \infty}  \left.\left(\widehat{\overline{\mathsf{V}}}_\lambda(q)\cdot \widehat{\overline{\mathsf{V}}}_\lambda(q)|_{t_2=t_2t^2_3, t_3=t^{-1}_3}\right)\right|_{t_1=L^N\kappa^{1/2},t_2=L^{-N}\kappa ^{1/2}, t_3=1}\\
       &=  \lim_{L\to \infty}  \left.\left(\widehat{\overline{\mathsf{V}}}_\lambda(q)\cdot \widehat{\overline{\mathsf{V}}}_\lambda(q)|_{t_2=t_2t^2_3, t_3=t^{-1}_3}\right)\right|_{t_1=L^N\kappa^{1/2},t_2=L^{-N-\epsilon}\kappa ^{1/2}, t_3=L^{\epsilon}},
   \end{align*}
   where the independence of the choice of slope in the last equality follows from \cite[Prop.~3.4]{Arb_K-theo_surface}.    Analogously, we have that
  \begin{multline*}
       \lim_{L\to \infty}   \left.\frac{  W'_{\lambda, 2}(q)}{  W'_{\varnothing, 2}(q)}\right|_{t_1=L\kappa^{1/2},t_2=L^{-1}\kappa ^{1/2}}=
  \\     \lim_{L\to \infty}  \left.\left(\widehat{\overline{\mathsf{V}}}_\lambda(q)^{-1}|_{t_1=t_1t^2_3, t_3=t^{-1}_3 }\cdot \widehat{\overline{\mathsf{V}}}_\lambda(q)|_{t_2=t_2t^2_3, t_3=t^{-1}_3}\right)\right|^{1/2}_{t_1=L^N\kappa^{1/2},t_2=L^{-N-\epsilon}\kappa ^{1/2}, t_3=L^{\epsilon}}.
   \end{multline*}
 By \cite[Prop.~4.6]{Arb_K-theo_surface} we have that
   \begin{align*}
        &\lim_{L\to \infty}  \left.\left(\widehat{\overline{\mathsf{V}}}_\lambda(-q)\right)\right|_{t_1=L^N\kappa^{1/2},t_2=L^{-N-\epsilon}\kappa ^{1/2}, t_3=L^{\epsilon}}=\frac{1}{\prod_{\Box\in \lambda}(1-(q\kappa^{1/2})^{\ell(\Box)} (q\kappa^{-1/2})^{a(\Box)+1})},\\
       &  \lim_{L\to \infty}  \left.\left(\widehat{\overline{\mathsf{V}}}_\lambda(-q)|_{t_2=t_2t^2_3, t_3=t^{-1}_3}\right)\right|_{t_1=L^N\kappa^{1/2},t_2=L^{-N-\epsilon}\kappa ^{1/2}, t_3=L^{\epsilon}}=\frac{1}{\prod_{\Box\in \overline{\lambda}}(1-(q\kappa^{-1/2})^{\ell(\Box)} (q\kappa^{1/2})^{a(\Box)+1})},\\
           &\lim_{L\to \infty}  \left.\left(\widehat{\overline{\mathsf{V}}}_\lambda(-q)|_{t_1=t_1t^2_3, t_3=t^{-1}_3}\right)\right|_{t_1=L^N\kappa^{1/2},t_2=L^{-N-\epsilon}\kappa ^{1/2}, t_3=L^{\epsilon}}=\frac{1}{\prod_{\Box\in \overline{\lambda}}(1-(q\kappa^{-1/2})^{\ell(\Box)} (q\kappa^{1/2})^{a(\Box)+1})}.
   \end{align*}
   Combining everything together yields
   \begin{align*}
           \lim_{L\to \infty}   \left.\frac{  W'_{\lambda, 1}(-q)}{  W'_{\varnothing, 1}(-q)}\right|_{t_1=L\kappa^{1/2},t_2=L^{-1}\kappa ^{1/2}}&=\prod_{\Box\in \lambda} \frac{1}{(1-(q\kappa^{1/2})^{\ell(\Box)} (q\kappa^{-1/2})^{a(\Box)+1})(1-(q\kappa^{-1/2})^{a(\Box)} (q\kappa^{1/2})^{\ell(\Box)+1})},\\
           \lim_{L\to \infty}   \left.\frac{  W'_{\lambda, 2}(-q)}{  W'_{\varnothing, 2}(-q)}\right|_{t_1=L\kappa^{1/2},t_2=L^{-1}\kappa ^{1/2}}&=1.
   \end{align*}
   Summing over all Young diagrams concludes the proof.
\end{proof}
Remarkably, the refined Donaldson-Thomas partition functions computed in  \Cref{thm: refined DT full Aga} reproduces a formula for the refined topological string of $\Tot_{C}(L_1\oplus L_2)$ proposed by Aganagic-Schaeffer \cite[Eqn.~(4.13)]{AS_black_holes} and studied with a TQFT approach in the context of the refined  Ooguri-Strominger-Vafa conjecture. In particular, in the genus zero case it reproduces the partition functions studied in \cite{IKV_topological_vertex} via the refined topological vertex. Notice, moreover, that the computation of the constant terms of the functions $W_{\lambda, i}(     q)$ in the proof of \Cref{thm: refined DT full Aga} reproduce the refined limit of the edge term in \cite[Prop.~4.2, 4.3]{Arb_K-theo_surface}.
\smallbreak
In some special cases, the formulas of \Cref{thm: refined DT full Aga} can be more compactly rewritten using the plethystic exponential. For instance, if $X=\Tot_{\BP^1}(\oO(-1)\oplus \oO(-1))$ is the resolved conifold, by \Cref{thm: refined DT full Aga} and the combinatorial identity in \cite[Eqn.~(5.6)]{IKV_topological_vertex} we recover the known expression
\begin{align}\label{eq_ res con}
\begin{split}
        \sum_{d\geq 0}Q^{d}\cdot   \overline{\DT}^{\mathsf{ref}}_d(X, -q)&=\Exp\left( \frac{Q}{[t_4][t_5]} \right)\\
    &= \Exp\left(\frac{-Qq}{(1-q\kappa^{1/2})(1-q\kappa^{-1/2})}\right).
    \end{split}
\end{align}
The same formula appears as the \emph{motivic} Pandharipande-Thomas partition function of the resolved conifold computed by Morrison-Mozgovoy-Nagao-Szendr\H{o}i \cite{MMNS_motivic_DT_conifold}. The formula in loc.~cit.~can be obtained by \eqref{eq_ res con} by the refined DT/PT correspondence (see \Cref{thm: DT/PT}) and the $K$-theoretic-to-motivic correspondence\footnote{To be precise, we exploit here that since the moduli space of stable pairs on $X$ is proper, the torus action is \emph{circle-compact}, and therefore the invariants defined by motivic and $K$-theoretic localisation coincide, see also the discussion in \cite[Prop.~8]{CKK_refined_BPS}. When the moduli space of stable pairs is not circle-compact, the motivic and $K$-theoretic invariants differ by the contribution of the complement of the attracting locus, cf.~\cite[Eq.~(1.3), (1.4)]{descombes_hyperbolic_loc}. In particular, we expect  for higher genus local curves a correction term that relates motivic and $K$-theoretic invariants, as the (conjectural) formula for motivic PT invariants proposed in \cite[Eq.~(1.5)]{CDDP_parabolic} suggests.} of Descombes \cite[pag.~8]{descombes_hyperbolic_loc}, exploiting the properness of the moduli space of stable pairs on the resolved conifold.
\subsection{$K$-theoretic-to-cohomological limit}
Throughout \Cref{sec: computations} we performed the computations of the Donaldson-Thomas partition functions in both equivariant cohomology and equivariant $K$-theory. We show now that the former is actually a limit of the latter. This confirms that it would have been enough to prove each of the results of \Cref{sec: computations} only for  the $K$-theoretic Donaldson-Thomas partition function.
\begin{prop}\label{prop: limit equiv}
  Let $X=\Tot_C(L_1\oplus L_2)$ be a local curve, where $L_1, L_2$ are line bundles over a smooth projective curve $C$ of genus $g$. We have
    \[\lim_{b\to 0}b^{d(2-2g+\deg L_1+\deg L_2)}\cdot \widehat{\DT}_d(X, q)|_{t_1=e^{bs_1}, t_2= e^{bs_2}}=\DT_d(X, q).\]
\end{prop}
\begin{proof}
The operators $[\cdot], e(\cdot)$ applied to a torus character $t_1^{\mu_1}t_2^{\mu_2}t_3^{\mu_3}$ obey the relation
\begin{align*}
    [ t_1^{\mu_1}t_2^{\mu_2}t_3^{\mu_3}]|_{t_1=e^{bs_1}, t_2= e^{bs_2}, t_3=e^{bs_3}}&= e^{\frac{b(\mu_1s_1+\mu_2s_2+\mu_3 s_3)}{2}}- e^{-\frac{b(\mu_1s_1+\mu_2s_2+\mu_3s_3)}{2}}\\
&=b (\mu_1s_1+\mu_2s_2+\mu_3 s_3)+o(b^2)\\
&=b \cdot e(t_1^{\mu_1}t_2^{\mu_2}t_3^{\mu_3})+o(b^2).
\end{align*}
By \Cref{thm: DT_ explicit univ series},  each coeffient of $q^n$ of the Donaldson-Thomas partition function $ \widehat{\DT}_d(X, q)$ can be written as a rational function of the form
\begin{align*}
  \left.  \left(\frac{\prod_{\mu}[t^{\mu}]}{\prod_{\nu}[t^\nu]}\right)\right|_{t_3=1},
\end{align*}
where the number of weights $\mu, \nu$ satisfy 
\[
\#\mu-\#\nu=-d(2-2g+\deg L_1+\deg L_2).
\]
Substituting $t_i=e^{b s_i}$ and taking the limit for $b\to 0$ concludes the argument.
\end{proof}
The proof of \Cref{prop: limit equiv} leverages on the fact that we should think of  the operator $e(\cdot)$ as the \emph{linearisation} of the operator $[\cdot]$, since
\begin{align*}
    [t^\mu]\xrightarrow{b\to 0} b\cdot e(t^\mu).
\end{align*}
This point of view seems to be pretty natural from the string-theoretic approach to  Donaldson-Thomas theory, see e.g. \cite{BBPT_elliptic_DT} and the references therein.
\smallbreak
The proof of \Cref{prop: limit equiv} relies on the explicit description of the universal series in terms of the 1-leg $K$-theoretic equivariant vertex and the operator $[\cdot]$. Alternatively, one could prove \Cref{prop: limit equiv} directly from the definition of the invariants \eqref{eqn: localized PT invariants}, \eqref{eqn: localized DT KK invariants} exploiting equivariant virtual  Riemann-Roch \cite{FG_riemann_roch} as in \cite[Thm.~6.4]{CKM_crepant}.
\section{DT/PT correspondence}
\subsection{Moduli space of stable pairs}\label{sec: PT}
Let $C$ be a smooth projective curve, $L_1, L_2$ line bundles on $C$ and set $ X=\Tot_C(L_1\oplus L_2)$ to be a local curve.  For a curve class $\beta=d[C]\in H_2(X, \BZ)$ and $n\in \BZ$, Pandharipande-Thomas \cite{PT_curve_counting_derived} introduced the \emph{moduli space of stable pairs} $P_n(X, \beta)$, parametrising flat families of complexes 
\[[\oO_X\xrightarrow[]{s} F]\in \derived^b(X),\]
such that $\coker (s)$ is zero-dimensional,  $[\Supp F]=\beta$ and $\chi(F)=n$.\\
By the work of   Huybrechts-Thomas \cite{HT_obstruction_theory},  the deformation theory of complexes   give rise to  a \emph{perfect obstruction theory} on $P_n(X, \beta)$ 
\begin{equation}\label{eqn: obstruction theory PT}
     \BE=\RR\pi_*\RR\hom(\CI^\bullet, \CI^\bullet)^\vee_0[-1]\to \BL_{P_n(X, \beta)},
\end{equation}
where 
 $\pi:X\times P_n(X, \beta)\to P_n(X, \beta)$ is the canonical projection and 
 \begin{align*}
     \CI^\bullet = [\oO\to \CF]\in \derived^b(P_n(X, \beta)\times X)
 \end{align*}
is the universal complex on  $P_n(X, \beta)\times X$.  By \cite{BF_normal_cone, CFK_virtual_fundamental_dg} the moduli space  $P_n(X, \beta)$  is endowed with \emph{virtual fundamental cycles}
\begin{align*}
    [P_n(X, \beta)]^{\vir}\in A_*(P_n(X, \beta)),\\
    \oO_{P_n(X, \beta)}^{\vir}\in K_0(P_n(X, \beta)),
\end{align*}
in complete analogy with the case of $\Hilb^n(X, \beta)$.

Similarly to \Cref{sec: Hilbert schemes}, the $\TT$-action on $X$ lifts to the moduli space of stable pairs $P_n(X, \beta)$, making the perfect obstruction theory and the virtual cycles naturally $\TT$-equivariant by \cite{Ric_equivariant_Atiyah}.  Moreover, the $\TT$-fixed locus $ P_n(X, \beta)^\TT$ is proper (cf. \Cref{thm: fixed PT}), therefore by Graber-Pandharipande \cite{GP_virtual_localization} there is a natural  induced perfect obstruction theory
 \begin{align*}
        \BE|^{\fix}_{P_n(X, \beta)^\TT}\to \BL_{P_n(X, \beta)^\TT}
 \end{align*}
 on  $ P_n(X, \beta)^\TT$ along with natural virtual cycles. We define \emph{$\TT$-equivariant  Pandharipande-Thomas invariants} as 
 \begin{align}\label{eqn: localized PT newwww}
    \PT_{d,n}(X)=\int_{[  P_n(X, \beta)]^{\vir}}1\in \BQ(s_1,s_2),
\end{align}
where the right-hand-side is defined by
 Graber-Pandharipande virtual localisation formula \cite{GP_virtual_localization} as
\begin{align*}
 \int_{[  P_n(X, \beta)]^{\vir}}1=\int_{[P_n(X, \beta)^{\TT}]^{\vir}}\frac{1}{e(N^{\vir})}\in \BQ(s_1,s_2).
\end{align*}
Similarly, we define  \emph{$K$-theoretic  Pandharipande-Thomas invariants} as
\begin{align}\label{eqn: localized PT KK invariants}
    \widehat{\PT}_{d,n}(X)=\chi\left(P_n(X, \beta),   \widehat{\oO}^{\vir}  \right)\in \BQ(t_1^{1/2}, t_2^{1/2}),
\end{align}
where the right-hand-side is defined by the virtual localisation formula in $K$-theory \cite{FG_riemann_roch, Qu_virtual_pullback} as
\begin{align*}
   \chi\left(P_n(X, \beta),   \widehat{\oO}^{\vir}  \right)=\chi\left(P_n(X, \beta)^\TT,\frac{  \widehat{\oO}^{\vir}_{P_n(X, \beta)^\TT} }{\widehat{\mathfrak{e}}(N^{\vir})}  \right)\in \BQ(t_1^{1/2}, t_2^{1/2}).
\end{align*}
The associated partition functions are defined as
\begin{equation}\label{eqn: PT partition function}
    \begin{split}
       \PT_d(X, q)&=\sum_{n\in \BZ} \PT_{d,n}(X)\cdot q^n\in \BQ(s_1,s_2)(\!( q )\!),\\
    \widehat{\PT}_d(X, q)&=\sum_{n\in \BZ} \widehat{\PT}_{d,n}(X)\cdot q^n\in \BQ(t_1^{1/2}, t_2^{1/2})(\!( q )\!).  
    \end{split}
\end{equation}
\smallbreak
 Finally, we introduce the  \emph{topological Euler characteristic} partition function  as
\begin{align*}
    \PT^{\mathrm{top}}_d(X, q)=  \sum_{n\in \BZ}e(P_n(X, \beta))\cdot q^n\in \BZ(\!(q)\!).
\end{align*}
\subsubsection{Virtual localisation}
We recall from \cite{Mon_double_nested} the structure of the $\TT$-fixed locus $P_n(X, \beta)^\TT$ and of the induced virtual structure.
 \begin{theorem}[{\cite[Prop.~3.1, Thm.~4.1,  Eqn.~(4.1)]{Mon_double_nested}, \cite[Cor.~4.3]{GLMRS_double-nested-1}}]\label{thm: fixed PT}
Let $C$ be a smooth projective curve of genus $g$ and $L_1, L_2$ line bundles on $C$. Set $X=\Tot_C(L_1\oplus L_2)$ and $\beta=d[C]$. Then,  there exists an isomorphism of schemes
\begin{align*}
    P_n(X, \beta)^\TT= \coprod_{|\lambda|= d}\coprod_{\mathbf{\Bm}} C^{[\mathbf{m}]},
\end{align*}
where the disjoint union is over all Young diagrams $\lambda$ of size $d$ and reverse plane partitions $\Bm$ of shape  $\lambda$ satisfying   $n=|\mathbf{m}|+ \mathbf{f}_{\lambda}( g,\deg L_1, \deg L_2)$.
In particular, $ P_n(X, \beta)^\TT$ is proper.
\smallbreak
Moreover, there are  identities of virtual cycles
\begin{align*}
     [C^{[\mathbf{\Bm}]}]^{\vir}&=[C^{[\mathbf{\Bm}]}]\in A_{\omega(\Bm)}(C^{[\mathbf{\Bm}]}),\\
       \oO_{C^{[\mathbf{\Bm}]}}^{\vir}&=  \oO_{C^{[\mathbf{\Bm}]}}\in K_0(C^{[\mathbf{\Bm}]}),
 \end{align*}
 and an identity 
in $K^0_\TT(C^{[\Bm]})$
 \begin{multline*}
     \BE|_{C^{[\Bm]}}^\vee=\sum_{(i,j)\in \lambda}\RR\pi_*(\oO(\CZ_{ij}) \otimes L_1^{-i} L_2^{-j}) t_1^{-i}t_2^{-j}\\
     -\sum_{(i,j)\in \lambda}\left(\RR\pi_*(\oO(\CZ_{ij}) \otimes \omega_C\otimes L_1^{-i-1} L_2^{-j-1} )\right)^\vee\cdot  t_1^{i+1}t_2^{j+1}\\
     -\sum_{(i,j),(l,k)\in \lambda}\RR\pi_*\left(\CN\otimes \oO(\CZ_{lk}-\CZ_{ij} )\otimes L_1^{i-l}L_2^{j-k} ) \right)t_1^{i-l}t_2^{j-k}.
 \end{multline*}
 \end{theorem}
\Cref{thm: fixed PT} implies that the partition functions \eqref{eqn: PT partition function} can be expressed as
\begin{align*}
    \PT_d(X, q)&=\sum_{|\lambda|= d}  q^{\mathbf{f}_{\lambda}( g,\deg L_1, \deg L_2)}\sum_{\mathbf{m}}q^{|\mathbf{m}|}\cdot   \int_{C^{[\mathbf{m}]}}e(-N_{C,L_1,L_2, \Bm}^{\vir}),\\
     \widehat{\PT}_d(X, q)&=\sum_{|\lambda|= d}  q^{\mathbf{f}_{\lambda}( g,\deg L_1, \deg L_2)}\sum_{\mathbf{m}}q^{|\mathbf{m}|}\cdot  \widehat{\chi}\left(C^{[\mathbf{m}]},\widehat{\mathfrak{e}}(-N_{C,L_1,L_2, \Bm}^{\vir})  \right),
\end{align*}
where $N^{\vir}_{C,L_1,L_2, \Bm}$ denotes the virtual normal bundle of $C^{[\Bm]}\hookrightarrow P_n(X, \beta)$ and $\Bm$ is a reverse plane partition of shape $ \lambda$.
\subsubsection{Universality}
The structure of  the partition functions \eqref{eqn: PT partition function} was studied in \cite{Mon_double_nested}, whose main results we now recall.

\begin{theorem}[{\cite[Thm.~5.1, 9.4]{Mon_double_nested}}]\label{thm: universal series PT}
Let $C$ be a genus $g$ smooth irreducible projective curve and $L_1, L_2$  line bundles over $C$. There is an identity
\begin{align*}
    \sum_{\mathbf{m}}q^{|\mathbf{m}|} \int_{C^{[\mathbf{\Bm}]}}e(-N_{C,L_1,L_2, \mathbf{m}}^{\vir})= D_{\lambda}^{1-g}\cdot E_{\lambda}^{\deg L_1}\cdot F_{\lambda}^{\deg L_2}\in \BQ(s_1,s_2)\llbracket q \rrbracket,
\end{align*}
where the sum is over all reverse plane partitions $\Bm$ of shape $\lambda$, and  $D_{\lambda},E_{\lambda},F_{\lambda}\in \BQ(s_1,s_2)\llbracket q \rrbracket$ are fixed universal series for $i=1,2,3$ which only depend on $ \lambda$. 
Similarly, there is an identity
\begin{align*}
    \sum_{\mathbf{m}}q^{|\mathbf{m}|} \widehat{\chi}\left(C^{[\mathbf{m}]},\widehat{\mathfrak{e}}(-N_{C,L_1,L_2, \mathbf{m}}^{\vir})  \right)= \widehat{D}_{\lambda}^{1-g}\cdot \widehat{E}_{\lambda}^{\deg L_1}\cdot \widehat{F}_{\lambda}^{\deg L_2}\in  \BQ(t_1^{1/2}, t_2^{1/2})\llbracket q \rrbracket,
\end{align*}
where the sum is over all reverse plane partitions $\Bm$ of shape $ \lambda$, and  $\widehat{D}_{\lambda},\widehat{E}_{\lambda},\widehat{F}_{\lambda}\in  \BQ(t_1^{1/2}, t_2^{1/2})\llbracket q \rrbracket$ are fixed universal series for $i=1,2,3$ which only depend on $ \lambda$. 
\end{theorem}
\subsection{Computations}
In \cite{Mon_double_nested}, we computed  the universal series  $D_\lambda, E_\lambda, F_\lambda$ under the anti-diagonal restriction of the equivariant parameters. In this section, we provide a full computation of all the universal series in \Cref{thm: universal series PT}, along the same lines of the techniques adopted in \Cref{sec: computations}.
\smallbreak
For the rest of the section, let $C$ be    a smooth projective curve, $L_1, L_2$ line bundles on $C$ and $\lambda$ a Young diagram.

\subsubsection{Constant term} 
Denote by $ N^{\vir}_{C, L_1, L_2,\mathbf{0}_{\PT} }$ the virtual normal bundle of $C^{[\mathbf{0}_{\PT}]}\cong \pt \hookrightarrow  P_n(X, \beta)$, where  $\mathbf{0}_{\PT}$ is the trivial reverse plane partition of shape $\lambda$ of size $0$.
\begin{lemma}\label{lemma: virtual normal const as univer PT}
    Let $C$ be a genus $g$  irreducible smooth projective curve and $L_1, L_2$ line bundles on $C$. Then 
    \begin{align*}
        N^{\vir}_{C, L_1, L_2,\mathbf{0}_{\PT} }=(1-g)\cdot T_\lambda+\deg L_1\cdot t_1\frac{\partial}{\partial t_1}T_\lambda+\deg L_2\cdot t_2\frac{\partial}{\partial t_2}T_\lambda.
    \end{align*}
    \end{lemma}
    \begin{proof}
        This follows by a direct comparison between the classes in \Cref{prop: K theort class}, \Cref{thm: fixed PT} and \Cref{prop: virtual normal const as univer}.
    \end{proof}
    \Cref{lemma: virtual normal const as univer PT} implies that the constant term of the localised partition functions coming from DT and PT theory coincide. In genus 0, this is a reflection of the identity of the \emph{edge terms} in DT and PT theory of toric threefolds, see \cite{MNOP_1, PT_vertex}.  
\subsubsection{Normalised series}
 By the universal expression of \Cref{thm: universal series PT}, the partition functions $\PT_d(X, q), \widehat{\PT}_d(X, q)$ are reduced to the computations of the \emph{normalised} series
\begin{align}\label{eqn: red Z gen PT}
\begin{split}
        Z_{\PT}(C,L_1, L_2)&=  \left(e(-N_{C,L_1,L_2, \mathbf{0}_{\PT}}^{\vir})\right)^{-1}\cdot\sum_{\mathbf{m}}q^{|\mathbf{m}|} \int_{C^{[\mathbf{m}]}}e(-N_{C,L_1,L_2, \mathbf{m}}^{\vir}),\\
      Z^{K}_{\PT}(C,L_1, L_2)&= \left(\widehat{\mathfrak{e}}(-N_{C,L_1,L_2, \mathbf{0}_{\PT}}^{\vir})  \right)^{-1}\cdot\sum_{\mathbf{m}}q^{|\mathbf{m}|} \widehat{\chi}\left(C^{[\mathbf{m}]},\widehat{\mathfrak{e}}(-N_{C,L_1,L_2, \mathbf{m}}^{\vir})  \right),
      \end{split}
\end{align}
where   $\mathbf{m}$ is a reverse plane partition of shape $ \lambda$.
\smallbreak
Given a reversed plane partition $\mathbf{m}$ of shape $\lambda$,  define the class in $K$-theory
\begin{multline*}
      N^{\mathrm{red}}_{C, L_1, L_2, \mathbf{m}}= \left(\sum_{(i,j)\in \lambda}\left(\pi_*(\oO_{\CZ_{ij}} \otimes L_1^{i} L_2^{j}\otimes \omega_C)\right)^{\vee} t_1^{-i}t_2^{-j}\right.\\
     -\sum_{(i,j)\in \lambda}\pi_*(\oO_{\CZ_{ij}} \otimes L_1^{i+1} L_2^{j+1} )\cdot  t_1^{i+1}t_2^{j+1}\\
     -\sum_{(i,j),(l,k)\in \lambda}\left(\RR\pi_*\RR\hom( \oO_{\CZ_{lk}}\otimes \CN^*,\oO_{\CZ_{ij}}\otimes L_1^{i-l}L_2^{j-k})\cdot t_1^{i-l}t_2^{j-k}\right.\\
   \left.  \left. - \pi_*( \oO_{\CZ_{ij}}\otimes \CN \otimes L_1^{i-l}L_2^{j-k})\cdot t_1^{i-l}t_2^{j-k}  + \left(\pi_*( \oO_{\CZ_{lk}}\otimes \CN^*\otimes L_1^{l-i}L_2^{k-j}\otimes \omega_C)\right)^{\vee}\cdot t_1^{l-i}t_2^{k-j}\right)\right)^{\mov}.
\end{multline*}
\begin{lemma}\label{eqn: lemma reduced PT}
Let $C$ be a smooth irreducible projective curve and $L_1, L_2$  line bundles over $C$. There are identities   
\begin{align*}
 Z_{\PT}(C,L_1, L_2)&= \sum_{\mathbf{m}}q^{|\mathbf{m}|} \int_{C^{[\mathbf{m}]}}e(-N_{C,L_1,L_2, \mathbf{m}}^{\mathrm{red}}),\\
     Z_{\PT}^K(C,L_1, L_2)&=\sum_{\mathbf{m}}q^{|\mathbf{m}|} \widehat{\chi}\left(C^{[\mathbf{m}]},\widehat{\mathfrak{e}}(-N_{C,L_1,L_2, \mathbf{m}}^{\mathrm{red}})  \right).
\end{align*}
\end{lemma}
\begin{proof}
 By the short exact sequence
 \begin{align*}
      0\to \oO(-\CZ_{ij})\to &\oO\to \oO_{\CZ_{ij}} \to 0, \quad (i,j)\in  \lambda,
 \end{align*}
 and Grothendieck duality applied to each summand of $   \BE|_{C^{[\Bm]}}^\vee$, 
 there is an identity
\begin{align}\label{eqn: rinorm PT}
       N_{C,L_1,L_2, \mathbf{m}}^{\vir}=  P_{C, L_1, L_2}+ N_{C,L_1,L_2, \mathbf{m}}^{\mathrm{red}},
\end{align}
where $  P_{C, L_1, L_2}$ was defined in \eqref{eqn: DT prefact normal}.  The conclusion of the  proof is analogous to the one of \Cref{eqn: lemma reduced} and follows directly from \eqref{eqn: rinorm PT}.
\end{proof}
\subsubsection{Genus 0}
As in \Cref{sec: genus 0 DT}, we  compute the universal series of \Cref{thm: universal series PT} evaluating the partition functions \eqref{eqn: red Z gen PT} when  $(C, L_1, L_2)$ is  among
\begin{align*}
    (\BP^1, \oO, \oO), \quad 
     (\BP^1, \oO, \oO(-2)), \quad 
      (\BP^1, \oO(-2), \oO).
\end{align*}
Set $\overline{\TT}=\TT\times \BC^*$ as in \Cref{sec: genus 0 DT}. The $\BC^*$-action on $\BP^1$ naturally lifts to a $\overline{\TT}$-action on  the double nested Hilbert schemes $(\BP^1)^{[\mathbf{m}]}$, where $\TT$ acts trivially.
\smallbreak
Let $\mathbf{m}=(m_{ij})_{(i,j)\in \lambda}$ be a reverse plane partition of shape $\lambda$. Define the $\overline{\TT}$-representations
\begin{align}\label{eqn: PT vertex Lauren}
\begin{split}
    \mathsf{Z}_{\Bm}&=\sum_{(i,j)\in \lambda}\sum_{\alpha=1}^{m_{ij}}t_1^{-i}t_2^{-j}t_3^{\alpha},\\
    \mathsf{v}_{\Bm}&=  \mathsf{Z}_{\Bm}-  \overline{\mathsf{Z}}_{\Bm}\cdot t_1t_2t_3-(1-t_1)(1-t_2)\left(-t_3\cdot     \mathsf{Z}_{\lambda}\overline{\mathsf{Z}}_{\Bm} +\overline{\mathsf{Z}}_{\lambda}\mathsf{Z}_{\Bm}+(1-t_3)\mathsf{Z}_{\Bm}\overline{\mathsf{Z}}_{\Bm}\right),
    \end{split}
\end{align}
where $\mathsf{Z}_\lambda$ was defined in \eqref{eqn: Z lambda}, 
and define the generating series
\begin{align}\label{eqn: vertex PT series}
\begin{split}
    \mathsf{V}^{\PT}_\lambda(q)&=\sum_{\Bm}e(-\mathsf{v}_{\Bm})\cdot q^{|\mathbf{m}|}\in \BQ(s_1, s_2, s_3)\llbracket q \rrbracket,\\
\widehat{\mathsf{V}}^{\PT}_\lambda(q)&=\sum_{\Bm}\widehat{\mathfrak{e}}(-\mathsf{v}_{\Bm})\cdot q^{|\mathbf{m}|}\in \BQ(t_1^{1/2}, t_2^{1/2}, t_3^{1/2})\llbracket q \rrbracket,
\end{split}
\end{align}
where the sum is over all reverse plane partition of shape $\lambda$. 
\begin{prop}\label{prop: vertex PT}
   Let $L_1, L_2$ be line bundles on $\BP^1$. We have identities of generating series
    \begin{align*}
  \sum_{\mathbf{m}}q^{|\mathbf{m}|} \int_{(\BP^1)^{[\mathbf{m}]}}e(-N_{\BP^1,L_1,L_2, \Bm}^{\mathrm{red}})&=  \left(\mathsf{V}^{\PT}_\lambda(q)\cdot \mathsf{V}^{\PT}_\lambda(q)|_{s_1=s_1-\deg L_1 s_3, s_2=s_2-\deg L_2 s_3, s_3=-s_3}\right)|_{s_3=0},\\
    \sum_{\mathbf{m}}q^{|\mathbf{m}|} \widehat{\chi}\left((\BP^1)^{[\mathbf{m}]},\widehat{\mathfrak{e}}(-N_{\BP^1,L_1,L_2, \Bm}^{\mathrm{red}})  \right)&=\left.\left(\widehat{\mathsf{V}}^{\PT}_\lambda(q)\cdot \widehat{\mathsf{V}}^{\PT}_\lambda(q)|_{t_1=t_1t_3^{-\deg L_1}, t_2=t_2t_3^{-\deg L_2}, t_3=t_3^{-1}}\right)\right|_{t_3=1}.
\end{align*}
\end{prop}
\begin{proof}
    The proof goes along the same lines of the one of \Cref{thm: toric loc}. We repeat here the main arguments, and sketch the main differences.

      We prove the first claim, as the second one follows by an analogous reasoning. Denote by $T^{\vir}$ the virtual tangent space of $(\BP^1)^{[\mathbf{m}]}$. By $\overline{\TT}$-equivariant virtual localisation, we have
    \begin{align*}
            \int_{(\BP^1)^{[\mathbf{m}]}}e(-N_{\BP^1,L_1,L_2, \Bm}^{\mathrm{red}})=\left.\left(\int_{((\BP^1)^{[\mathbf{m}]})^{\overline{\TT}}}e(-N_{\BP^1,L_1,L_2, \Bm}^{\mathrm{red}}-T^{\vir})|_{((\BP^1)^{[\mathbf{m}]})^{\overline{\TT}}}\right)\right|_{s_3=0}.
    \end{align*}
    Notice that a $\overline{\TT}$-fixed flag $(Z_\Box)_\Box\in ((\BP^1)^{[\mathbf{m}]})^{\overline{\TT}}$ decomposes as 
    \begin{align*}
         (Z_\Box)_\Box&=(Z^0_\Box)_\Box + (Z^\infty_\Box)_\Box\\
        &= \Bm_1\cdot [0]+\Bm_2\cdot [\infty],
    \end{align*}
    for some reverse plane partitions $\Bm_1, \Bm_2$. 
    This implies that 
    the $\overline{\TT}$-fixed locus $((\BP^1)^{[\mathbf{m}]})^{\overline{\TT}}$ is reduced, zero-dimensional and  in bijection with 
    \begin{align*}
        \set{(\Bm_1, \Bm_2)| \Bm_1, \Bm_2 \mbox{ reverse plane partitions of shape }\lambda \mbox{ such that } \Bm= \mathbf{m}_1+\mathbf{m}_2}.
    \end{align*}
    In particular, this implies that the virtual $\overline{\TT}$-representation $T^{\vir}_{(\Bm_1, \Bm_2)} $ is  $\overline{\TT}$-movable at all fixed points corresponding to $(\Bm_1, \Bm_2)$.

    Fix now a $\overline{\TT}$-fixed point $(\Bm_1, \Bm_2)$. By construction, there is  an identity of virtual $\overline{\TT}$-representations
    \begin{align*}
      N_{\BP^1,L_1,L_2, \Bm}^{\mathrm{red}}+T^{\vir}=\BE|_{(\BP^1)^{[\mathbf{m}]}}^\vee- P_{C, L_1, L_2}.
    \end{align*}
    Taking the fiber over the fixed point corresponding to $(\Bm_1, \Bm_2)$ yields
    \begin{multline}\label{eqn: sum 0 inft PT}
         ( N_{\BP^1,L_1,L_2, \Bm}^{\mathrm{red}}+T^{\vir})|_{(\Bm_1, \Bm_2)}=\\
        \sum_{p\in \{0, \infty\}}\left(\sum_{(i,j)\in \lambda}\left(H^0(\oO_{Z^p_{ij}}\otimes  L_1^{i}L_2^{j}\otimes \omega_{\BP^1})^*\cdot t_1^{-i}t_2^{-j} - H^0( \oO_{Z^p_{ij}}\otimes L_1^{i+1}L_2^{j+1})\cdot t_1^{i+1}t_2^{j+1}\right.\right.\\
    -\sum_{(i,j),(l,k)\in \lambda}\left(\RR\Hom( \oO_{Z^p_{lk}}\otimes \CN^*,\oO_{Z^p_{ij}}\otimes L_1^{i-l}L_2^{j-k})\cdot t_1^{i-l}t_2^{j-k}\right.\\
 \left. - H^0( \oO_{Z^p_{ij}}\otimes \CN \otimes L_1^{i-l}L_2^{j-k})\cdot t_1^{i-l}t_2^{j-k}  + \left(H^0( \oO_{Z^p_{lk}}\otimes \CN^*\otimes L_1^{l-i}L_2^{k-j}\otimes \omega_{\BP^1})\right)^{*}\cdot t_1^{l-i}t_2^{k-j}\right).
    \end{multline}
Notice the identities of $\overline{\TT}$-representations
\begin{align*}
      \mathsf{Z}_{\bn_1}&= \sum_{(i,j)\in \lambda}H^0(\oO_{Z^0_{ij}}\otimes  L_1^{i}L_2^{j}\otimes \omega_{\BP^1})^*\cdot t_1^{-i}t_2^{-j},\\
     \overline{\mathsf{Z}}_{\bn_1}\cdot t_1t_2t_3 &=\sum_{(i,j)\in \lambda} H^0( \oO_{Z^0_{ij}}\otimes L_1^{i+1}L_2^{j+1})\cdot t_1^{i+1}t_2^{j+1}.
    \end{align*}
    By the latter identities, as in the proof of \Cref{thm: toric loc} the contribution of $p=0$ to the right-hand-side of \eqref{eqn: sum 0 inft PT} is
    \begin{align*}
        \mathsf{v}_{\Bm_1}=\mathsf{Z}_{\Bm_1}- \overline{\mathsf{Z}}_{\Bm_1}\cdot t_1t_2t_3-(1-t_1)(1-t_2)(-t_3\mathsf{Z}_{\lambda}\overline{\mathsf{Z}}_{\Bm_1} +\overline{\mathsf{Z}}_{\lambda}\mathsf{Z}_{\Bm_1}+(1-t_3)\mathsf{Z}_{\Bm_1}\overline{\mathsf{Z}}_{\Bm_1}).
    \end{align*}
    Similarly, using the equivariant structure \eqref{eqn: struct at inf} around $p=\infty$, we have that the contribution of $p=\infty$ to the right-hand-side of \eqref{eqn: sum 0 inft PT} is
    \begin{align*}\label{eqn: V2 PT}
          \mathsf{v}_{\Bm_2}|_{t_1=t_1t_3^{-\deg L_1}, t_2=t_2t_3^{-\deg L_2}, t_3=t_3^{-1}}.
    \end{align*}
    Summing all up, we have proved that 
\[
 ( N_{\BP^1,L_1,L_2, \Bm}^{\mathrm{red}}+T^{\vir})|_{(\Bm_1, \Bm_2)}=  \mathsf{v}_{\Bm_1}+    \mathsf{v}_{\Bm_2}|_{t_1=t_1t_3^{-\deg L_1}, t_2=t_2t_3^{-\deg L_2}, t_3=t_3^{-1}}.
\]
By \cite[Sec.~3.3]{PT_vertex} the term $ \mathsf{v}_{\Bm_1}$ is $\overline{\TT}$-movable, which implies that the measure $e(- \mathsf{v}_{\Bm_1}) $ is well-defined. Therefore we conclude that
    \begin{align*}
          \int_{(\BP^1)^{[\mathbf{m}]}}e(-N_{\BP^1,L_1,L_2, \Bm}^{\mathrm{red}})&=\left.\left(\sum_{\Bm=\Bm_1+\Bm_2}e(- \mathsf{v}_{\Bm_1})\cdot e(-  \mathsf{v}_{\Bm_2}|_{t_1=t_1t_3^{-\deg L_1}, t_2=t_2t_3^{-\deg L_2}, t_3=t_3^{-1}})\right)\right|_{s_3=0}\\
          &=\left.\left(\sum_{\Bm=\Bm_1+\Bm_2}e(- \mathsf{v}_{\Bm_1})\cdot e(-  \mathsf{v}_{\Bm_2})|_{s_1=s_1-\deg L_1 s_3, s_2=s_2-\deg L_2 s_3, s_3=-s_3}\right)\right|_{s_3=0}.
    \end{align*}
    Summing the latter identities  over all reverse plane partitions of shape $\lambda$ we obtain the proof of the claimed identities.
\end{proof}
As an immediate corollary, we conclude the computation of the universal series in \ref{thm: universal series PT}.
\begin{theorem}\label{thm: PT_ explicit univ series}
   The universal series are given by
\begin{align*}
    D_\lambda(q)&=\prod_{\Box\in \lambda}\frac{1}{(-\ell(\Box)s_1+(a(\Box)+1)s_2)((\ell(\Box)+1)s_1-a(\Box)s_2)}\cdot \left(\mathsf{V}^{\PT}_\lambda(q)\cdot \mathsf{V}^{\PT}_\lambda(q)|_{ s_3=-s_3}\right)|_{s_3=0},\\
    E_\lambda(q)&=\prod_{\Box\in \lambda}\frac{( -\ell(\Box) s_1+(a(\Box)+1)s_2)^{\ell(\Box)}}{( (\ell(\Box)+1)s_1-a(\Box)s_2)^{\ell(\Box)+1}}\cdot \left.\left(\mathsf{V}^{\PT}_\lambda(q)|_{ s_3=-s_3}\cdot \mathsf{V}^{\PT}_\lambda(q)^{-1}|_{s_1=s_1+2s_3, s_3=-s_3}\right)\right|^{\frac{1}{2}}_{s_3=0}, \\
    F_\lambda(q)&=\prod_{\Box\in \lambda}\frac{(\ell(\Box)+1)s_1-a(\Box)s_2)^{a(\Box)}}{(-\ell(\Box)s_1+(a(\Box)+1)s_2)^{a(\Box)+1}}\cdot \left.\left(\mathsf{V}^{\PT}_\lambda(q)|_{ s_3=-s_3}\cdot \mathsf{V}^{\PT}_\lambda(q)^{-1}|_{s_2=s_2+2s_3, s_3=-s_3}\right)\right|^{\frac{1}{2}}_{s_3=0},\\
      \widehat{D}_\lambda(q)&=\prod_{\Box\in \lambda}\frac{1}{[t_1^{-\ell(\Box)}t_2^{a(\Box)+1}][t_1^{\ell(\Box)+1}t_2^{-a(\Box)}]}\cdot \left.\left(\widehat{\mathsf{V}}^{\PT}_\lambda(q)\cdot \widehat{\mathsf{V}}^{\PT}_\lambda(q)|_{ t_3=t^{-1}_3}\right)\right|_{t_3=1}\\
    \widehat{E}_\lambda(q)&=\prod_{\Box\in \lambda}\frac{[t_1^{-\ell(\Box)}t_2^{a(\Box)+1}]^{\ell(\Box)}}{[t_1^{\ell(\Box)+1}t_2^{-a(\Box)}]^{\ell(\Box)+1}}\cdot \left.\left(\widehat{\mathsf{V}}^{\PT}_\lambda(q)|_{ t_3=t_3^{-1}}\cdot \widehat{\mathsf{V}}^{\PT}_\lambda(q)^{-1}|_{t_1=t_1t^2_3, t_3=t^{-1}_3}\right)\right|^{\frac{1}{2}}_{t_3=1},  \\
    \widehat{F}_\lambda(q)&=\prod_{\Box\in \lambda}\frac{[t_1^{\ell(\Box)+1}t_2^{-a(\Box)}]^{a(\Box)}}{[t_1^{-\ell(\Box)}t_2^{a(\Box)+1}]^{a(\Box)+1}}\cdot \left.\left(\widehat{\mathsf{V}}^{\PT}_\lambda(q)|_{ t_3=t_3^{-1}}\cdot \widehat{\mathsf{V}}^{\PT}_\lambda(q)^{-1}|_{t_2=t_2t^2_3, t_3=t^{-1}_3}\right)\right|^{\frac{1}{2}}_{t_3=1}.
\end{align*}
\end{theorem}
\begin{proof}
The proof goes along the same lines of the one of \Cref{thm: DT_ explicit univ series}. The constant terms of the series  in \Cref{thm: universal series PT} is computed by combining \Cref{lemma: virtual normal const as univer PT} and  \Cref{eqn: costant DT expl}. By applying \Cref{prop: vertex PT} to $(C, L_1, L_2)$   among
\begin{align*}
    (\BP^1, \oO, \oO), \quad 
     (\BP^1, \oO, \oO(-2)), \quad 
      (\BP^1, \oO(-2), \oO),
\end{align*}
we end up with two  systems of three equations. Extracting the logarithm and solving the resulting systems uniquely determines the  required universal series, thus proving the claim.
\end{proof}
\subsubsection{The 1-leg vertex}
Analogously to \Cref{Sec: DT 1 leg vert}, the generating series \eqref{eqn: vertex PT series} reproduce the \emph{vertex partition functions} obtained by starting from the vertex/edge formalism introduced in \cite{PT_vertex} in Pandharipande-Thomas theory, in the case of  one infinite leg. We briefly recall the combinatorial data used to express the (Pandharipande-Thomas)  vertex formalism for one leg.
\smallbreak
A \emph{1-leg PT box arrangement with asymptotic profile} $\lambda$ is a  collection of boxes $\sigma\in \BZ^3$ of the form 
\[
\sigma=\sigma'+\sigma'',
\]
where 
\begin{align*}
    \sigma'= \set{(i,j,k)\in \BZ_{\geq 0}^3| (i,j)\in \lambda},
\end{align*}
and
$\sigma''\in \BZ^3$ is a finite collection of boxes satisying
\begin{itemize}
    \item for every $(i,j,k)\in \sigma''$, we have $(i,j)\in \lambda$,
    \item for every $(i,j,k)\in \sigma''$, we have $k<0$, 
    \item for all $(i,j,k-1)\in \sigma''$ with $k<0$, we have $(i,j,k)\in \sigma''$,
    \item for all $(i,j)\in \lambda$, if either $(i-1,j,k)$ or $(i,j-1, k)$ are in $\sigma''$, then $(i,j,k)\in \sigma''$. 
\end{itemize}
Pictorially, a 1-leg PT box arrangement with asymptotic profile $\lambda$ consists of an infinite leg in the positive $x_3$-direction, along with a collection of finitely many boxes confined in the prolongment of the infinite leg in the negative $x_3$-direction, with \emph{gravity}\footnote{We borrowed the graphical idea of \emph{gravity} from \cite[Sec.~3]{BCY_orbifold}.} $(1,1,-1)$, see e.g.~\Cref{fig:PT}.

We define the \emph{normalised size} of $\sigma$ to be
\begin{align*}
    |\sigma|=|\sigma''|.
\end{align*}
To each reverse plane partition $\Bm=(m_{ij})_{(i,j)\in \lambda}$ of shape $\lambda$, we associate a  1-leg PT box arrangement $\sigma$   with asymptotic profile $\lambda$ by setting
\begin{align*}
    \sigma&=\sigma'+\sigma'',\\
    \sigma'&=  \set{(i,j,k)\in \BZ_{\geq 0}^3| (i,j)\in \lambda},\\
    \sigma''&= \set{(i,j,k)\in \BZ^3| (i,j)\in \lambda \mbox{ and } 0>k\geq -m_{ij}}
\end{align*}
This association is easily seen to provide a bijection  between the sets
\begin{align*}
    \set{\mbox{reverse plane partitions of shape } \lambda} \longleftrightarrow \set{\mbox{1-leg PT box arrangement with asymptotic profile } \lambda },
\end{align*}
which furthermore preserves the size.

\begin{figure}[H]
\centering
\tdplotsetmaincoords{70}{120}
\begin{tikzpicture}[tdplot_main_coords, scale=0.8]

\newcommand{\drawcubecolored}[4]{ 
  \pgfmathsetmacro{\x}{#1}
  \pgfmathsetmacro{\y}{#2}
  \pgfmathsetmacro{\z}{#3}
  \def\basecolor{#4}

  \fill[\basecolor!40] (\x,\y,\z+1) -- ++(1,0,0) -- ++(0,1,0) -- ++(-1,0,0) -- cycle; 
  \fill[\basecolor!30] (\x+1,\y,\z) -- ++(0,0,1) -- ++(0,1,0) -- ++(0,0,-1) -- cycle; 
  \fill[\basecolor!20] (\x,\y+1,\z) -- ++(1,0,0) -- ++(0,0,1) -- ++(-1,0,0) -- cycle; 

  \draw[thick] (\x,\y,\z) -- ++(1,0,0) -- ++(0,1,0) -- ++(-1,0,0) -- cycle;
  \draw[thick] (\x,\y,\z) -- ++(0,0,1);
  \draw[thick] (\x+1,\y,\z) -- ++(0,0,1);
  \draw[thick] (\x+1,\y+1,\z) -- ++(0,0,1);
  \draw[thick] (\x,\y+1,\z) -- ++(0,0,1);
  \draw[thick] (\x,\y,\z+1) -- ++(1,0,0) -- ++(0,1,0) -- ++(-1,0,0) -- cycle;
}

\foreach \x in {-10,...,-1} {
  \ifnum\x=-10
    \drawcubecolored{\x}{1}{0}{red} 
  \else
    \ifnum\x>-11
      \ifnum\x=-9
        \drawcubecolored{\x}{0}{1}{red}
        \drawcubecolored{\x}{1}{0}{red}
      \else
        \drawcubecolored{\x}{0}{0}{gray}
        \drawcubecolored{\x}{1}{0}{gray}
        \drawcubecolored{\x}{0}{1}{gray}
      \fi
    \fi
  \fi
}

\draw[dotted, thick] (0,0,0) -- ++(2,0,0);

\end{tikzpicture}

\caption{A 1-leg PT box arrangement of size 3, with asymptotic profile a Young diagram of size 3. In grey, the infinite leg. In red, the 3 boxes with negative $x_3$-coordinate.}
\label{fig:PT}
\end{figure}
Under the above correspondence it is immediate to verify that, given a reverse plane partition $\Bm$, the vertex term $\mathsf{v}_{\Bm}$ reproduces precisely\footnote{Due to different conventions on the torus $\overline{\TT}$-action in \cite{PT_vertex}, the vertex terms $\mathsf{v}_{\Bm}$ are matched only after the change of variables $t_i\mapsto t_i^{-1}$, for $i=1, 2,3$. } the \emph{normalised vertex} term of \cite[Sec.~4.6]{PT_vertex} of the corresponding 1-leg plane partition $\sigma$.
As already recalled in \Cref{Sec: DT 1 leg vert}, a formula for the \emph{capped} 1-leg vertex $\widehat{\mathsf{V}}^{\PT}_\lambda(q)$ can be derived from \cite[Eqn.~(29)]{KOO_2_legDT}. 

\subsubsection{The full partition function}
Summing over Young diagram of size $d$, we obtain the following expression for the PT partition functions.
\begin{corollary}\label{cor: full PT}
    Let  $d\geq 0$ and    $X=\Tot_C(L_1 \oplus L_2)$, where $L_1, L_2$ are line bundles over a smooth projective curve $C$ of genus $g$. We have
    \begin{align*}
     \PT_d(X, q)&= \sum_{|\lambda|=d} \left(q^{|\lambda|}{D}_{\lambda}(q)\right)^{1-g}\cdot \left(q^{-n(\lambda)}{E}_{\lambda}(q)\right)^{\deg L_1}\cdot \left(q^{-n(\overline{\lambda})}{F}_{\lambda}(q)\right)^{\deg L_2},\\
     \widehat{\PT}_d(X, q)&= \sum_{|\lambda|=d} \left(q^{|\lambda|}\widehat{D}_{\lambda}(q)\right)^{1-g}\cdot \left(q^{-n(\lambda)}\widehat{E}_{\lambda}(q)\right)^{\deg L_1}\cdot \left(q^{-n(\overline{\lambda})}\widehat{F}_{\lambda}(q)\right)^{\deg L_2},
\end{align*}
where the universal series are computed in \Cref{thm: PT_ explicit univ series}.
\end{corollary}
\subsection{Degree 1} We  establish a more explicit formula for the universal series \Cref{thm: PT_ explicit univ series} in degree $d=1$.
\begin{prop}\label{prop: degree 1}
    For $\lambda=\Box$,   the universal series are given by
\begin{align*}
    D_\Box(q)&=s_1^{-1}s_2^{-1},\\
    E_\Box(-q)&=s_1^{-1}\cdot  (1-q), \\
    F_\Box(-q)&=s_2^{-1}\cdot (1-q),\\
      \widehat{D}_\Box(q)&=\frac{1}{[t_1][t_2]}\\
    \widehat{E}_\Box(-q)&=\frac{1}{[t_1]}\cdot \Exp\left(-\frac{(t_1t_2)^{1/2}+(t_1t_2)^{-1/2}}{2}\cdot q \right),  \\
    \widehat{F}_\Box(-q)&=\frac{1}{[t_2]}\cdot \Exp\left(-\frac{(t_1t_2)^{1/2}+(t_1t_2)^{-1/2}}{2}\cdot q \right).
\end{align*}
\end{prop}
\begin{proof}
We start by computing the vertex term $\widehat{\mathsf{V}}^{\PT}_\Box(q) $. For every $n$, we have a unique reverse plane partition $\Bm$ of shape $\BZ^2_{\geq 0}\setminus \Box$ of size $n$. We have
\begin{align*}
    \mathsf{Z}_{\Bm}&=\sum_{\alpha=1}^{n}t_3^{\alpha},\\
    \mathsf{v}_{\Bm}&= \sum_{\alpha=1}^{n}t_3^{\alpha}-  \sum_{\alpha=1}^{n}t_3^{-\alpha}\cdot t_1t_2t_3,
\end{align*}
which yields
\begin{align*}
    \widehat{\mathsf{V}}^{\PT}_\Box(-q)&=\sum_{n\geq 0}(-q)^{n}\cdot [-  \mathsf{v}_{\Bm}]\\
    &=\sum_{n\geq 0}\left(q(t_1t_2t_3)^{1/2}\right)^n\prod_{i=1}^n\frac{1-(t_1t_2t_3)^{-1}t_3^i}{1-t_3^i}\\
    &=\Exp\left(q(t_1t_2t_3)^{1/2}\frac{1-(t_1t_2t_3)^{-1}t_3}{1-t_3} \right)\\
    &=\Exp\left(q\frac{[(t_1t_2t_3)^{-1}t_3]}{[t_3]} \right),
\end{align*}
 where in the third line we used \cite[Ex.~5.1.22]{Okounk_Lectures_K_theory}. The last  three series identities follow by direct calculation from \Cref{thm: PT_ explicit univ series}, while the first three follow by substituting $t_i=e^{bs_i}$ and 
taking the limit $b\to 0$, as explained in \Cref{prop: limit equiv}.
\end{proof}
The universal series in \Cref{prop: degree 1} recover and generalise the series computed in \cite[Sec.~8.2]{Mon_double_nested} via tautological integrals.
\smallbreak
Recall that we defined two new variables $t_4, t_5$ in \eqref{eqn: new variables}.
\begin{corollary}\label{cor: PT a la Oko deg 1}
Let     $X=\Tot_C(L_1 \oplus L_2)$, where $L_1, L_2$ are line bundles over a smooth projective curve $C$ of genus $g$. Set $k_i=\deg L_i$. We have    
\begin{align*}
    \widehat{\PT}_1(X,-q)=(-1)^{1-g}\cdot q^{1-g}[t_1]^{g-1-k_1}[t_2]^{g-1-k_2}\cdot\left(\frac{1}{(1-t_4)(1-t_5^{-1})}\right)^{ -\frac{k_1+k_2}{2}}.
\end{align*}
In particular, if $L_1\otimes L_2 \cong \omega_C$, we have
\begin{align*}
    \widehat{\PT}_1(X,-q)=
    [t_1]^{g-1-k_1}[t_2]^{g-1-k_2}[t_4]^{g-1}[t_5]^{g-1}.
\end{align*}
\end{corollary}
\begin{proof}
    By \Cref{cor: full DT} and \Cref{prop: degree 1}, we have
    \begin{align*}
          \widehat{\PT}_1(X,-q)&=(-1)^{1-g}\cdot q^{1-g}[t_1]^{g-1-k_1}[t_2]^{g-1-k_2}\cdot\Exp\left( -\frac{k_1+k_2}{2}(t_4+t_5^{-1})\right)\\
          &=(-1)^{1-g}\cdot q^{1-g}[t_1]^{g-1-k_1}[t_2]^{g-1-k_2}\cdot\left(\frac{1}{(1-t_4)(1-t_5^{-1})}\right)^{ -\frac{k_1+k_2}{2}}\\
          &=(-1)^{1-g}\cdot q^{1-g}[t_1]^{g-1-k_1}[t_2]^{g-1-k_2}\cdot\left(\frac{q^{-1}}{[t_4^{-1}][t_5]}\right)^{ -\frac{k_1+k_2}{2}}.
    \end{align*}
    We conclude the proof of the second claim by  exploiting $[t_4^{-1}]=-[t_4]$.
\end{proof}
 \subsection{DT/PT correspondence}\label{sec: DT/PT}
As a final application, we immediately deduce from our structural results the  DT/PT correspondences for local curves in equivariant cohomology,  in equivariant $K$-theory and for the topological Euler characteristic.

\begin{corollary}\label{thm: DT/PT}
    Let $X=\Tot_C(L_1\oplus L_2)$ be a local curve, where $L_1, L_2$ are line bundles over a smooth projective curve. Then, the following correspondences hold
\begin{align*}
    \DT_d(X, q)&=\DT_0(X, q)\cdot \PT_d(X, q),\\
        \widehat{\DT}_d(X, q)&=\widehat{\DT}_0(X, q)\cdot \widehat{\PT}_d(X, q),\\
            \DT^{\mathrm{top}}_d(X, q)&=\DT^{\mathrm{top}}_0(X, q)\cdot \PT^{\mathrm{top}}_d(X, q).
\end{align*}
\end{corollary}
\begin{proof}
    The proof of the last identity follows by combining \Cref{cor: top DT} and \cite[Cor.~3.2]{Mon_double_nested}, while the  first identity follows from the second one by applying  \Cref{prop: limit equiv}.

    The second identity follows by  combining \Cref{thm: DT_ explicit univ series}, \ref{thm: PT_ explicit univ series} with the \emph{vertex} DT/PT correspondence 
    \[\widehat{\mathsf{V}}_\lambda(q)=\widehat{\mathsf{V}}_{\varnothing}(q)\cdot \widehat{\mathsf{V}}^{\PT}_\lambda(q),
    \]
    proven in \cite{KLT_DTPT}.
\end{proof}
Combining the DT/PT correspondence in  \Cref{thm: DT/PT} with the explicit formula of \Cref{cor: antidiagonal DT}, we obtain an explicit description of the PT generating series under the anti-diagonal restriction $t_1t_2=1$, which was left in an implicit combinatorial form in \cite[Sec.~9.2]{Mon_double_nested}. Further specialising to equivariant cohomology by \Cref{prop: limit equiv}, we recover the invariants previously computed in \cite[Thm.~8.2]{Mon_double_nested}, therefore simplifying the long combinatorial analysis of \cite[App.~A]{Mon_double_nested}.

\bibliographystyle{amsplain-nodash}
\bibliography{The_Bible}

\end{document}